\documentclass[11pt,leqno]{amsart}
\usepackage{a4wide}
\usepackage{amssymb,upgreek}
\usepackage{amsmath}
\usepackage{amsthm}
\usepackage{mathrsfs}
\usepackage{bm,euscript,csquotes,comment,combelow}

\usepackage[curve]{xypic}
\usepackage[usenames,dvipsnames]{xcolor}
\usepackage{hyperref}
\usepackage{bbm}
\theoremstyle{plain}

\makeatletter
\def\cal@symb#1|#2{\expandafter\def\csname #2#1\endcsname{\mathcal{#1}}}
\def\calsymbols#1#2{\@for\@tmpz:=#2\do{\expandafter\cal@symb\@tmpz|{#1}}}
\def\bb@symb#1|#2{\expandafter\def\csname #2#1\endcsname{\mathbb{#1}}}
\def\bbsymbols#1#2{\@for\@tmpz:=#2\do{\expandafter\bb@symb\@tmpz|{#1}}}
\def\bold@symb#1|#2{\expandafter\def\csname #2#1\endcsname{\mathbf{#1}}}
\def\boldsymbols#1#2{\@for\@tmpz:=#2\do{\expandafter\bold@symb\@tmpz|{#1}}}
\def\scr@symb#1|#2{\expandafter\def\csname #2#1\endcsname{\mathscr{#1}}}
\def\scrsymbols#1#2{\@for\@tmpz:=#2\do{\expandafter\scr@symb\@tmpz|{#1}}}
\def\frak@symb#1|#2{\expandafter\def\csname #2#1\endcsname{\mathfrak{#1}}}
\def\fraksymbols#1#2{\@for\@tmpz:=#2\do{\expandafter\frak@symb\@tmpz|{#1}}}

\def\dmth@p#1|{\expandafter\let\csname#1\endcsname\relax
  \expandafter\DeclareMathOperator\csname#1\endcsname{#1}}
\def\operators#1{\@for\@tmpz:=#1\do{\expandafter\dmth@p\@tmpz|}}
\makeatother
\calsymbols{c}{A,B,C,D,E,F,G,H,I,J,K,L,M,N,O,P,Q,R,S,T,U,V,W,X,Y,Z}
\bbsymbols{b}{A,B,C,D,E,F,G,H,I,J,K,L,M,N,O,P,Q,R,S,T,U,V,W,X,Y,Z}
\boldsymbols{bf}{A,B,C,D,E,F,G,H,I,J,K,L,M,N,O,P,Q,R,S,T,U,V,W,X,Y,Z}
\scrsymbols{s}{A,B,C,D,E,F,G,H,I,J,K,L,M,N,O,P,Q,R,S,T,U,V,W,X,Y,Z}
\fraksymbols{fr}{a,b,c,d,e,f,g,h,i,j,k,l,m,n,o,p,q,r,s,t,u,v,w,x,y,z,X,B}
\operators{an,GL,End,Hom,colim,Diff,gr,Spf,ur,Spec,Hom,alg,Gal,div,fgp,Sp,rig,ev,cts,im,Res,ab,Norm,cyc,Art,LT,Tr,Pol,Int,op,diff,la,LC,Nm,Id,diag,triv}

\newcommand{\col}{\operatorname{Col}}
\newcommand{\ts}[1]{\texorpdfstring{$#1$}{}}
\newcommand{\qmb}[1]{\quad\mbox{#1}\quad}
\newcommand{\Cp}{\bC_p}
\newcommand{\Qp}{\bQ_p}
\newcommand{\Zp}{\bZ_p}

\newcommand{\OO}{\mathcal{O}}
\newcommand{\hsp}{\hspace{0.1cm}}

\newcommand{\be}{\begin{enumerate}}
\newcommand{\ee}{\end{enumerate}}
\newcommand{\h}[1]{\widehat{#1}}
\newcommand{\dcroc}[1]{[\![ #1 ]\!]}
\newcommand{\Qpbar}{\overline{\bQ}_p}
\newcommand{\Gm}{\bfG_\mathrm{m}}
\newcommand{\bigO}{\mathrm{O}}
\newcommand{\res}{\operatorname{res}}
\newcommand{\val}{\operatorname{val}}
\newcommand{\vp}{\val_p}

\newcommand{\ul}[1]{\underline{#1}}
\newcommand{\psiqint}{\psi_q\!\operatorname{-int}}

\theoremstyle{definition}
\newtheorem{definition}{Definition}[subsection]
\newtheorem{theorem}[definition]{Theorem}
\newtheorem{corollary}[definition]{Corollary}
\newtheorem{proposition}[definition]{Proposition}
\newtheorem{conjecture}[definition]{Conjecture}
\newtheorem{lemma}[definition]{Lemma}
\newtheorem{example}[definition]{Example}
\newtheorem{remark}[definition]{Remark}
\newtheorem{question}[definition]{Question}

\usepackage{blkarray}
\usepackage{comment}
\usepackage{enumerate}
\usepackage{float}
\usepackage{graphicx}
\usepackage[utf8]{inputenc}
\usepackage{listings}
\newcommand{\stoptocwriting}{%
  \addtocontents{toc}{\protect\setcounter{tocdepth}{-5}}}
\newcommand{\resumetocwriting}{%
  \addtocontents{toc}{\protect\setcounter{tocdepth}{\arabic{tocdepth}}}}
\title{Bounded functions on the character variety}

\author{Konstantin Ardakov}
\address{Konstantin Ardakov, Mathematical Institute, University of Oxford}
\email{ardakov@maths.ox.ac.uk}

\author{Laurent Berger}
\address{Laurent Berger, UMPA ENS de Lyon, UMR 5669 du CNRS}
\email{laurent.berger@ens-lyon.fr}

\date{January 31, 2023}

\begin{document}

\begin{abstract}
This paper is motivated by an open question in $p$-adic Fourier theory, that seems to be more difficult than it appears at first glance.
Let $L$ be a finite extension of $\mathbb{Q}_p$ with ring of integers $o_L$ and let $\mathbb{C}_p$ 
denote the completion of an algebraic closure of $\mathbb{Q}_p$. 
In their work on $p$-adic Fourier theory, Schneider and Teitelbaum defined and studied the character variety $\mathfrak{X}$. 
This character variety is a rigid analytic curve over $L$ that parameterizes the set of locally $L$-analytic characters 
$\lambda : (o_L,+) \to (\mathbb{C}_p^\times,\times)$. One of the main results of Schneider and Teitelbaum is that over $\mathbb{C}_p$, 
the curve $\mathfrak{X}$ becomes isomorphic to the open unit disk. Let $\Lambda_L(\mathfrak{X})$ denote the ring of bounded-by-one functions
on $\mathfrak{X}$. If $\mu \in o_L [\![o_L]\!]$ is a measure on $o_L$, then $\lambda \mapsto \mu(\lambda)$ gives rise to an element of $\Lambda_L(\mathfrak{X})$. 
The resulting map $o_L [\![o_L]\!] \to \Lambda_L(\mathfrak{X})$ is injective. The question is: do we have $\Lambda_L(\mathfrak{X}) = o_L [\![o_L]\!]$?

In this paper, we prove various results that were obtained while studying this question. In particular, 
we give several criteria for a positive answer to the above question. We also 
recall and prove the ``Katz isomorphism'' that describes the dual of a certain space of continuous functions on $o_L$. 
An important part of our paper is devoted to providing a proof of this theorem which was stated in 1977 by Katz. 
We then show how it applies to the question. 
Besides $p$-adic Fourier theory, 
the above question is related to the theory of formal groups, the theory of integer valued polynomials on $o_L$, $p$-adic Hodge theory,
and Iwasawa theory.
\end{abstract}

\dedicatory{With an appendix by Drago\cb{s} Cri\cb{s}an and Jingjie Yang}

\subjclass{11F80; 11S15; 11S20; 11S31; 13F20; 13J07; 14G22; 22E50; 46S10}

\maketitle

\tableofcontents

\section{Introduction}
\label{introwriteup}

\subsection{Motivation}
Let $L$ be a finite extension of $\Qp$ and let $\Cp$ denote the completion of an algebraic closure of $\Qp$. In their work on $p$-adic Fourier theory \cite{ST}, Schneider and Teitelbaum defined and studied the character variety $\frX$. This character variety is a rigid analytic curve over $L$ that parameterizes the set of locally $L$-analytic characters $\lambda : (o_L,+) \to (\Cp^\times,\times)$. One of the main results of Schneider and Teitelbaum is that over $\Cp$, the curve $\frX$ becomes isomorphic to the open unit disk. 

The ring $\cO_L(\frX)$ of holomorphic functions on $\frX$ is a Pr\"ufer domain, with an action of $o_L$ coming from the natural action of $o_L$ on the set of locally $L$-analytic characters. One can then localize and complete $\cO_L(\frX)$ in order to obtain the Robba ring $\sR_L(\frX)$, and define $(\varphi,o_L^\times)$-modules over that ring and some of its subrings. These objects are defined and studied in Berger--Schneider--Xie \cite{BSX}, with the hope that they will be useful for a generalization of the $p$-adic local Langlands correspondence from $\GL_2(\Qp)$ to $\GL_2(L)$.

In this paper, we instead consider a natural subring of $\cO_L(\frX)$, the ring $\Lambda_L(\frX)$ of functions whose norms are bounded above by $1$. If $\mu \in o_L \dcroc{o_L}$ is a measure on $o_L$, then $\lambda \mapsto \mu(\lambda)$ gives rise to such a function. The resulting map $o_L \dcroc{o_L} \to \Lambda_L(\frX)$ is injective. We do not know of any example of an element of $\Lambda_L(\frX)$ that is not in the image of the above map. 

\begin{question}
\label{mainquest}
Do we have $\Lambda_L(\frX) = o_L \dcroc{o_L}$?
\end{question}

This question seems to be more difficult than it appears at first glance, and so far we have not been able to answer it (except of course for $L=\Qp$). The results of this paper were obtained while we were studying this problem. 
A related question is raised in remark 2.5 of \cite{CGL}.
We now give more details about the character variety $\frX$, and then explain our main results.

\subsection{The character variety}
Let $\frB$ denote the open unit disk, seen as a rigid analytic variety. This space naturally parameterizes the set of locally $\Qp$-analytic characters $\lambda : (\Zp,+) \to (\Cp^\times,\times)$. Indeed, if $K$ is a closed subfield of $\Cp$ and $z \in \frm_K = \frB(K)$, then the map $\lambda_z : a \mapsto (1+z)^a$ is a $K$-valued locally $\Qp$-analytic character on $\Zp$, and every such character arises in this way. Note that $\lambda_z'(0) = \log(1+z)$. If $d=[L:\Qp]$, then $o_L \simeq \Zp^d$ and hence $\frB^d$ parameterizes the set of locally $\Qp$-analytic characters $\lambda : (o_L,+) \to (\Cp^\times,\times)$. Such a character is locally $L$-analytic if and only if $\lambda'(0)$ is $L$-linear. In coordinates $z = (z_1,\hdots,z_d)$, there exists $\alpha_2,\hdots,\alpha_d \in L$ such that the character corresponding to $z$ is locally $L$-analytic if and only if $\log(1+z_i) = \alpha_i \cdot \log(1+z_1)$ for all $i=2,\hdots,d$. These $d-1$ Cauchy--Riemann equations cut out the character variety $\frX$ inside $\frB^d$. Schneider and Teitelbaum showed \cite{ST} that $\frX$ is a smooth rigid analytic group curve over $L$. 

The ring of $\Qp$-analytic distributions $D^{\Qp-\an}(o_L,L)$ on $o_L$ is isomorphic to the ring of power series in $d$ variables that converge on the open unit polydisk. Every distribution $\mu \in D^{\Qp-\an}(o_L,L)$ gives rise to an element of $\cO_L(\frX)$ via the map $\lambda \mapsto \mu(\lambda)$. This gives rise to a surjective (but not injective if $L \neq \Qp$) map $D^{\Qp-\an}(o_L,L) \to \cO_L(\frX)$, whose restriction to $o_L \dcroc{o_L}$ is injective and has image contained in $\Lambda_L(\frX)$.

\subsection{Schneider and Teitelbaum's uniformization}
We now explain why over $\Cp$, the curve $\frX$ becomes isomorphic to the open unit disk. Let $G_L = \Gal(\Qpbar/L)$. Choose a uniformizer $\pi$ of $o_L$ and let $\cG$ denote the Lubin--Tate formal group attached to $\pi$. This gives us a Lubin--Tate character $\chi_\pi : G_L \to o_L^\times$ and, once we have chosen a coordinate $Z$ on $\cG$, a formal addition law $X\oplus Y  \in o_L \dcroc{X,Y}$, endomorphisms $[a](Z) \in o_L \dcroc{Z}$ for all $a \in o_L$, and a logarithm $\log_{\LT}(Z) \in L \dcroc{Z}$.

By the work of Tate on $p$-divisible groups, there is a non-trivial homomorphism $\cG \to \Gm$ defined over $o_{\Cp}$. Concretely, there exists a power series $G(Z) \in o_{\Cp} \dcroc{Z}$ (a generator of $\Hom_{o_{\Cp}}(\cG,\Gm)$) such that $G(X \oplus Y)=G(X) \cdot G(Y)$. If $z \in \frm_{\Cp}$, then the map $\lambda_z : a \mapsto G([a](z))$ is a locally $L$-analytic character on $o_L$, and every such character arises in this way. This explains the main idea behind the proof of the statement that over $\Cp$, the curve $\frX$ becomes isomorphic to the open unit disk. 

In particular, $\cO_{\Cp}(\frX)$ is isomorphic to the ring of power series $\sum_{i \geq 0} a_i Z^i$ with $a_i \in \Cp$ that converge on the open unit disk. Let $\chi_{\cyc}$ denote the cyclotomic character, and let $\tau : G_L \to o_L^\times$ denote the character $\tau = \chi_{\cyc} \cdot \chi_\pi^{-1}$. The Galois group $G_L$ acts on $\cO_{\Cp}(\frX)$ by the formula $g(\sum_{i \geq 0} a_i Z^i) = \sum_{i \geq 0} g(a_i) [\tau(g)^{-1}](Z)^i$. This action is called the twisted Galois action, and we write $G_L,*$ to recall the twist. It follows from the Ax-Sen-Tate theorem that $\Cp^{G_L} = L$ and then, by unravelling the definitions, that $\cO_L(\frX) = \cO_{\Cp}(\frX)^{G_L,*}$. At the level of bounded functions, this tells us that $\Lambda_L(\frX) = o_{\Cp} \dcroc{Z}^{G_L,*}$. The natural map $o_L \dcroc{o_L} \to \Lambda_L(\frX)$ sends, for instance, the Dirac measure $\delta_a$ with $a \in o_L$ to $G([a](Z)) \in \Lambda_L(\frX)$.

\subsection{The operators $\varphi_q$, $\psi_q$}
The monoid $(o_L,\times)$ acts on $o_L$ by multiplication, and hence on the set of locally $L$-analytic characters,  on $\frX$, and on the ring $\cO_{\Cp}(\frX)$. If $a \in o_L$, then this action is given by $f(Z) \mapsto f([a](Z))$. Let $q$ denote the cardinality of the residue field $k_L$ of $o_L$ and let $\varphi_q$ denote the action of $\pi$ on $\cO_{\Cp}(\frX)$. The ring $\cO_{\Cp}(\frX)$ is a free $\varphi_q(\cO_{\Cp}(\frX))$-module of rank $q$. Let $\psi_q : \cO_{\Cp}(\frX) \to \cO_{\Cp}(\frX)$ be the map defined by $\varphi_q(\psi_q(f(Z))) = 1/q \cdot \Tr_{\cO_{\Cp}(\frX) / \varphi_q(\cO_{\Cp}(\frX))} (f(Z))$. The action of $o_L$ and the operator $\psi_q$ commute with the twisted action of $G_L$, and therefore preserve $\cO_L(\frX)$. If we consider the image of the map $D^{\Qp-\an}(o_L,L) \to \cO_L(\frX)$, we have $a \cdot \delta_b = \delta_{ab}$ and $\psi_q(\delta_b) = 0$ if $b \in o_L^\times$ and $\psi_q(\delta_b) = \delta_{b/\pi}$ if $b \in \pi o_L$. In particular, $o_L\dcroc{o_L}^{\psi_q=0}$ coincides with $o_L\dcroc{o_L^\times}$, those measures that are supported in $o_L^\times$. We use later on the fact (lemma \ref{psizall})  that $\Lambda_L(\frX) = o_L \dcroc{o_L}$ if and only if $\Lambda_L(\frX)^{\psi_q=0} = o_L \dcroc{o_L^\times}$. Note that if $L \neq \Qp$, then $\psi_q(\Lambda_{\Cp}(\frX))$ is not contained in $\Lambda_{\Cp}(\frX)$ as $\Tr_{\cO_{\Cp}(\frX) / \varphi_q(\cO_{\Cp}(\frX))} (f(Z))$ is divisible by $\pi$, but not always by $q$. Our first result is the following.

\begin{theorem}
\label{intropsicrit}
We have $\Lambda_L(\frX) = o_L \dcroc{o_L}$ if and only if $\psi_q(\Lambda_L(\frX)) \subset \Lambda_L(\frX)$.
\end{theorem}
This is proved at the end of $\S \ref{PsiSection}$. 
\subsection{The polynomials $P_n$}\label{PnSectIntro}
Recall that $G(Z)$ is a generator of $\Hom_{o_{\Cp}}(\cG,\Gm)$ and that $\tau = \chi_{\cyc} \cdot \chi_\pi^{-1}$. In fact, we have $G(Z) = \exp(\Omega \cdot \log_{\LT}(Z)) = 1 +\Omega \cdot Z + \bigO(Z^2)$, where $\Omega$ is a certain special element of $\frm_{\Cp}$ such that $g(\Omega) = \tau(g) \cdot \Omega$. In particular, for all $n \geq 0$, there exists a polynomial $P_n(Y) \in L[Y]$ such that $G(Z) = \sum_{n \geq 0} P_n(\Omega) \cdot Z^n$. For $n \geq 0$, the polynomial $P_n(Y)$ is of degree $n$, and its leading coefficient is $1/n!$. For example, assume that the coordinate $Z$ is chosen in a way that $\log_{\LT}(Z) = \sum_{k \geq 0} Z^{q^k}/\pi^k$. Then we have (see Proposition \ref{polypm} for more details)
\[ P_n(Y) = \sum_{n_0+qn_1+\cdots+q^d n_d=n} \frac{Y^{n_0+\cdots+n_d}}{n_0! \cdots n_d! \cdot \pi^{1 \cdot n_1 + 2 \cdot n_2 + \cdots + d \cdot n_d}}. \]
If $a \in o_L$, then $G([a](Z)) = \sum_{n \geq 0} P_n(\Omega) \cdot [a](Z)^n = \sum_{n \geq 0} P_n(a \Omega) \cdot Z^n$. This implies for instance that $P_n(a \Omega) \in o_{\Cp}$ for all $a \in o_L$. For $n \geq 0$ and $i \geq n$, let $\sigma_{n,i}(Y) \in L[Y]$ denote the polynomials such that $[a](Z)^n = \sum_{i \geq n} \sigma_{n,i}(a) Z^i$ for all $a \in o_L$. The $\sigma_{n,i}(Y)$ are all elements of $\Int$, the $o_L$-submodule of $L[Y]$ of integer valued polynomials on $o_L$. The fact that $\sum_{n \geq 0} P_n(\Omega) \cdot [a](Z)^n = \sum_{n \geq 0} P_n(a \Omega) \cdot Z^n$ implies that $P_n(a \Omega) = \sum_{i=0}^n \sigma_{i,n}(a) P_i(\Omega)$. If $\mu \in D^{\Qp-\an}(o_L,L)$, then its image in $\cO_L(\frX)$ is therefore $f_\mu(Z) = \sum_{n \geq 0} Z^n \cdot \sum_{i=0}^n \mu(\sigma_{i,n}) P_i(\Omega)$. Let $\Pol$ denote the $o_L$-span of the $\sigma_{n,i}(Y)$ inside $L[Y]$, so that $\Pol \subset \Int$. The following gives a relation between our question and the theory of integer valued polynomials (\cite{dSh}, \cite{dSIce}):
\begin{theorem}
\label{intropolint}
If $\Lambda_L(\frX) = o_L \dcroc{o_L}$, then $\Pol=\Int$.
\end{theorem}

The proof can be found at the end of $\S \ref{RhoijSect}$. The converse statement is not true, but ``$\Pol=\Int$'' is equivalent to $U \dcroc{Z}^{G_L,*} = o_L \dcroc{o_L}$, where $U$ is the $o_L$-submodule of $o_{\Cp}$ generated by $\{P_n(\Omega)\}_{n \geq 0}$. We have not been able to prove that $\Pol=\Int$, although we can show that $\Pol$ is $p$-adically dense in $\Int$. Some numerical evidence indicates that $\Pol=\Int$ seems to hold: the details can be found in the Appendix by D. Crisan and J. Yang at the end of our paper.

We now explain how to compute the valuation of $P_n(\Omega)$ for certain $n$. The elements $z \in \frm_{\Cp}$ such that $G(z)=1$ correspond to those locally $L$-analytic characters $\lambda_z$ such that $\lambda_z(1)=1$. Being locally $L$-analytic, they are necessarily trivial on an open subgroup of $o_L$, and correspond to certain torsion points of $\cG$. We know the valuations of these torsion points, and this way we can determine the Newton polygon of $G(Z)-1$. Using this idea, we can prove the following. Let $e$ be the ramification index of $L/\Qp$. If $m \geq 0$, let $k_m = \lfloor (m-1)/e \rfloor$, so that $m=ek_m+r$ with $1 \leq r \leq e$. For $m \geq 0$, let $x_m = q^m/p^{k_m+1}$ (so that $x_0=1$ and $x_1=q/p$). Write $m=en+r$ and let  \[ y_0 = \frac{e}{p-1} - \frac{1}{q-1}  \ \ \text{and}\ \  y_m  = \frac{e}{p^n(p-1)} - \frac{r}{p^{n+1}} - \frac{1}{(q-1)p^{n+1}}. \]

\begin{theorem}
\label{introvalpkgen}
For all $m \geq 0$, we have $\val_\pi(P_{x_m}(\Omega)) = y_m$.
\end{theorem}

For example, if $L=\bQ_{p^2}$, then $\vp(P_{p^k}(\Omega)) = 1/p^{k-1}(q-1)$ for all $k \geq 0$.

\subsection{Galois-continuous functions and the Katz map}
Following Katz \cite{Ka1}, we let $\cC^0_{\Gal}(o_L,o_{\Cp})$ denote the $o_L$-module of Galois-continuous functions, namely those continuous functions $f : o_L \to o_{\Cp}$ such that $g(f(a)) = f(\tau(g) \cdot a)$ for all $a \in o_L$ and $g \in G_L$. If $P(T) \in L[T]$, then $a \mapsto P(a \cdot \Omega)$ is such a function. Let $K$ be a closed subfield of $\Cp$ containing $L$. The \emph{dual Katz map} is the map $\cK^\ast : \Hom_{o_L}(\cC^0_{\Gal}(o_L,o_{\Cp}), o_K) \to o_K \dcroc{Z}$ given by $\mu \mapsto \sum_{n \geq 0} \mu(P_n) \cdot Z^n$. Let $o_K \dcroc{Z}^{\psiqint}$ denote the set of $f(Z) \in o_K \dcroc{Z}$ such that $\psi_q^n(f(Z)) \in o_K \dcroc{Z}$ for all $n \geq 1$. Our main technical result is the following

\begin{theorem} Suppose that $L = \bQ_{p^2}$.
\label{introkatzmap}
\begin{enumerate}
\item The map $\cK^\ast : \Hom_{o_L}(\cC^0_{\Gal}(o_L,o_{\Cp}), o_K) \to o_K \dcroc{Z}$ is injective.
\item Its image is equal to $o_K \dcroc{Z}^{\psiqint}$. 
\end{enumerate}
\end{theorem}
An important part of our paper is devoted to providing a proof of this theorem, which is completed at the end of $\S \ref{Qp^2KatzSection}$. We note that Theorem \ref{introkatzmap} was stated by Katz at \cite[p. 60]{Ka1}, but he did not give a proof. The remarks contained in the last paragraph of \cite[\S IV]{Ka1} seem to indicate that his proof is different to ours.

The hardest part of the theorem is the claim concerning the image of $\cK^\ast$. Note that when $L=\bQ_{p^2}$, the dual of the $p$-divisible group attached to $\cG$ has dimension $1$. Using this and Theorem \ref{introvalpkgen} for $L=\bQ_{p^2}$, we can prove (see Proposition \ref{prop:K1onto}) that every element of $o_\infty = o_{\Cp}^{\ker \tau}$ can be written as $\sum_{n \geq 0} \lambda_n \cdot P_n(\Omega)$ where $\lambda_n \in o_L$ and $\lambda_n \to 0$. 
This important ingredient of the proof of Theorem \ref{introkatzmap} is not known to be available if $L \neq \bQ_{p^2}$.

\subsection{Applications of the Katz isomorphism}\label{KatzApps}
Throughout this section, we assume that $L=\bQ_{p^2}$ and $\pi=p$, so that $\cK^\ast : \Hom_{o_L}(\cC^0_{\Gal}(o_L,o_{\Cp}), o_K) \to o_K \dcroc{Z}^{\psiqint}$ is an isomorphism. Let $L_\infty = \Cp^{\ker \tau}$ and $o_\infty = o_{\Cp}^{\ker \tau}$. Since $\pi=p$, $L_\infty$ is also the completion of $L(\cG[p^\infty])$.

Theorem \ref{introkatzmap} gives us an isomorphism $\cK : \Hom_{o_L}(\cC^0_{\Gal}(o_L^\times,o_{\Cp}), o_K) \to o_K \dcroc{Z}^{\psi_q=0}$, and we 
have a natural isomorphism $\cC^0_{\Gal}(o_L^\times,o_{\Cp}) \to o_\infty$. Applying this to $K=L$, we get the following result  (Theorem \ref{dualoinfty}), where $o_\infty^* = \Hom_{o_L}(o_\infty,o_L)$:

\begin{theorem}
\label{introdualoinfty}
The map $\cK^\ast$ gives rise to an isomorphism $o_\infty^* \simeq o_L \dcroc{Z}^{\psi_q=0}$.
\end{theorem}

Let $\Gamma_L^{\LT} = \Gal(L(\cG[p^\infty])/L)$ and $\Gamma_{\Qp}^{\cyc} = \Gal(\Qp(\mu_{p^\infty})/\Qp)$. In the cyclotomic setting, Perrin-Riou showed \cite[Lemma 1.5]{PR} that $\Zp \dcroc{Z}^{\psi_p=0}$ is a free $\Zp \dcroc{\Gamma_{\Qp}^{\cyc}}$-module of rank $1$. She also raised the question of what happens in the present setting. Using Theorem \ref{introdualoinfty}, we show in Corollary \ref{PerrinRiou} that $o_L \dcroc{Z}^{\psi_q=0}$ is in fact \emph{not} a free $o_L \dcroc{\Gamma_L^{\LT}}$-module of rank $1$.

We can also apply the isomorphism $\Hom_{o_L}(o_\infty, o_K) \simeq o_K \dcroc{Z}^{\psi_q=0}$ to $K = L_\infty$, and we get $\Hom_{o_L}(o_\infty, o_\infty) \simeq o_\infty \dcroc{Z}^{\psi_q=0}$. The natural action of $G_L$ on the left is the twisted Galois action on the right. Since $\Lambda_L(\frX) = o_{\Cp} \dcroc{Z}^{G_L,*} = o_\infty \dcroc{Z}^{G_L,*}$, we get the following result (Theorem \ref{contgalend}): 

\begin{theorem}
\label{introcontgalend}
We have $\End^{G_L}_{o_L}(o_\infty) \simeq \Lambda_L(\frX)^{\psi_q=0}$.
\end{theorem}

Recall that $o_L \dcroc{o_L^\times} \subset \Lambda_L(\frX)^{\psi_q=0}$. If $a \in o_L^\times$, then $\delta_a \in o_L \dcroc{o_L^\times}$ acts on $o_\infty$ by an element $g \in G_L$ such that $\tau(g) = a$. Since $\Lambda_L(\frX) = o_L \dcroc{o_L}$ if and only if $\Lambda_L(\frX)^{\psi_q=0} = o_L \dcroc{o_L^\times}$, we get the following criterion (Theorem \ref{critcge}):

\begin{theorem}
\label{introcritcge}
We have $\Lambda_L(\frX) = o_L \dcroc{o_L}$ if and only if every continuous $L$-linear and $G_L$-equivariant map $f : L_\infty \to L_\infty$ comes from the Iwasawa algebra $L \otimes_{o_L} o_L \dcroc{\Gamma_L^{\LT}}$.
\end{theorem}

In the cyclotomic case, Tate's normalized trace maps $T_n : \Qp^{\cyc} \to \Qp(\mu_{p^n})$ are examples of continuous $\Qp$-linear and $G_{\Qp}$-equivariant maps $f : \Qp^{\cyc} \to \Qp^{\cyc}$ that do not come from the Iwasawa algebra $L \otimes_{o_L} o_L \dcroc{\Gamma_{\Qp}^{\cyc}}$. The lack of normalized trace maps in the Lubin--Tate setting is a source of many complications. In his PhD thesis, Fourquaux considered continuous $L$-linear and $G_L$-equivariant maps $f : L_\infty \to L_\infty$. We generalize some of Fourquaux's results: we prove in Proposition \ref{foulocalg} that if $f \neq 0$ is such a map, then there exists $n \geq 0$ such that $f(L_\infty)$ contains a basis of the $L_n$-vector space $L_n[\log \Omega]$, where $L_n = L(\cG[p^n])$. In particular, $f$ necessarily has a very large image, so there can be no analogue of the equivariant trace maps $T_n$.

The Katz isomorphism also allows us to prove several results about the span of the polynomials $P_n$ in $\cC^0_{\Gal}(o_L,\Cp)$. 
Recall that by \cite[Theorem 4.7]{ST}, every Galois-continuous locally analytic function on $o_L$ can be expanded as an overconvergent series in the $P_n$. One may then wonder about the existence of such an expansion for Galois-continuous functions. Let $\cC^0(L)$ denote the set of sequences $\{ \lambda_n \}_{n \geq 0}$ with $\lambda_n \in L$ and $\lambda_n \to 0$. The Katz isomorphism, and computations involving $\psi_q$, imply the following (Proposition \ref{pndense}, Corollary \ref{pnindep}, and Corollary \ref{notbij}):

\begin{theorem}
\label{intropncogal}
The map $\cC^0(L) \to \cC^0_{\Gal}(o_L,\Cp)$, given by $\{ \lambda_n \}_{n \geq 0} \mapsto \left[a \mapsto \sum\limits_{n = 0}^\infty \lambda_n \cdot P_n(a \Omega) \right]$ is injective, has dense image, but is not surjective. 
\end{theorem}

The same methods imply the following precise estimates for those elements of $\cC^0_{\Gal}(o_L,\Cp)$ that are given by a polynomial function $a \mapsto Q(a\Omega)$ with $Q(T) \in L[T]$. See prop \ref{hupk} and coro \ref{hnotu}.

\begin{theorem}
\label{introcoeffpn}
Assume that $Z$ is a coordinate on $\cG$ such that $[p](Z)=Z^q+pZ$.
Let $Q(T) \in L[T]$ be a polynomial such that $Q(a \Omega) \in o_{\Cp}$ for all $a \in o_L$, and write $Q(T) = \sum_{n=0}^{\deg Q} \lambda_n \cdot P_n(T)$. 
\begin{enumerate}
\item We have $\lambda_n \in p^{-k} o_L$ if $n \leq q^k$.
\item There exists such a polynomial $Q$ for which $\lambda_{q^k-1} = p^{-k}$.
\end{enumerate}
\end{theorem}

\subsection{Other criteria}
The following two criteria for our main question may be of interest. 

Let $\partial : \Cp\dcroc{Z} \to \Cp\dcroc{Z}$ denote the invariant derivative $\partial = \log_{\LT}'(Z)^{-1} \cdot d/dZ$. It does not commute with the twisted action of $G_L$, but $D = \Omega^{-1} \cdot \partial$ does. We get a map $D : \cO_{\Cp}(\frX) \to \cO_{\Cp}(\frX)$ that does not preserve $\Lambda_{\Cp}(\frX)$ if $L \neq \Qp$ since $\vp(\Omega^{-1}) < 0$. Note that $D(\delta_a) = a \cdot \delta_a$ if $a \in o_L$, so that $D$ does preserve $o_L \dcroc{o_L}$. We have the following result.

\begin{theorem}
\label{introderivcrit}
If $L=\bQ_{p^2}$, then $\Lambda_L(\frX) = o_L \dcroc{o_L}$ if and only if $D^{q-1}(\Lambda_L(\frX)) \subset \Lambda_L(\frX)$.
\end{theorem}

This Theorem follows from Theorem \ref{intropsicrit} and the following result, which is inspired by computations of Katz: assume that $L=\bQ_{p^2}$ and that $\pi=p$. Let $\lambda= \Omega^{q-1} / p(q-1)! \in o_{\Cp}^\times$. If $f(Z) \in o_{\Cp} \dcroc{Z}$, then $\varphi \psi_q(f) - \lambda \cdot D^{q-1}(f) \in o_{\Cp} \dcroc{Z}$. 

Here is another result concerning our main question. It says that if the answer is yes for a finite extension $K/L$, then the answer is also yes for $L$.

\begin{theorem}
\label{introfinext}
If $K/L$ is finite and if $\Lambda_K(\frX_K) = o_K \dcroc{o_K}$, then $\Lambda_L(\frX_L) = o_L \dcroc{o_L}$.
\end{theorem}

\subsection{Acknowledgements}

This paper grew out of a project started with Peter Schneider. The authors are very grateful to him for numerous discussions, interesting insights (in particular, considering the Katz isomorphism), and several invitations to M\"unster. Several results in this paper were obtained in collaboration with him.
L.B. also thanks Pierre Colmez for some discussions about the main problem of this paper.

\section{The character variety}

\subsection{Notation}
Let $\bQ_p \subseteq L \subset \bC_p$ be a field of finite degree $d$ over $\bQ_p$, $o_L$ the ring of integers of $L$, $\pi \in o_L$ a fixed prime element, $k_L = o_L/\pi o_L$ the residue field,  $q := |k_L|$ and $e$ the absolute ramification index of $L$. We always use the absolute value $|\ |$ on $\bC_p$ which is normalized by $|p| = p^{-1}$. We let $G_L := \Gal(\overline{L}/L)$ denote the absolute Galois group of $L$. Throughout our coefficient field $K$ is a complete intermediate extension $L \subseteq K \subseteq \bC_p$.

\subsection{The $p$-adic Fourier transform}
We are interested in the \emph{character variety} $\frX$ of the $L$-analytic commutative group $(o_L,+)$. We refer to \cite[\S 2]{ST} for a precise definition, but recall that $\frX$ is a rigid analytic variety defined over $L$, whose set of $K$-points (for $K$ a field extension of $L$ complete with respect to a non-archimedean absolute value extending the one on $L$) is the group $\frX(K)$ of $K$-valued characters $\chi : (o_L,+) \to (K^\times,\times)$ that are also $L$-analytic functions:
\[ \frX(K) := \{f \in C^{L-\an}(o_L, K) : f(a+b) = f(a)f(b) \qmb{for all} a,b \in o_L\}.\]
Here $C^{L-\an}(o_L,K)$ is the space of locally $L$-analytic $K$-valued functions on $o_L$. 

Let $D^{L-\an}(o_L,K)$ be the $K$-algebra of locally $L$-analytic distributions on $o_L$, defined in \cite[\S 2]{ST2}. One of the main results of $p$-adic Fourier Theory --- \cite[Theorem 2.3]{ST} --- tells us that there is a canonical isomorphism 
\[ \cF : D^{L-\an}(o_L, K) \to \cO( \frX \times_L K)\]
called the \emph{$p$-adic Fourier Transform}. This isomorphism is determined by
\[ \cF(\lambda)(\chi) = \lambda(\chi) \qmb{for all} \lambda \in D^{L-\an}(o_L,K), \chi \in \frX(K).\]
Since $\frX$ is a rigid $L$-analytic variety, we have at our disposal the subalgebra $\cO^\circ(\frX)$ of $\cO(\frX)$ consisting of globally-defined, rigid analytic functions on $\frX$ that are \emph{power-bounded} --- see \cite[\S 1.2.5]{BGR}.  
\begin{definition} Write $\Lambda(\frX) := \cO^\circ(\frX)$. \end{definition}
The functorial definition of the character variety does not shed much light on its internal structure. It turns out that the base change $\frX \times_L K$ is isomorphic to the rigid analytic open unit disc over $K$, \emph{provided} the field $K$ is large enough. This isomorphism is obtained with the help of \emph{Lubin-Tate formal groups} and their associated \emph{$p$-divisible groups}.

\subsection{Lubin-Tate formal groups}\label{sect:LT}Let $Z$ be an indeterminate and let 
\[\sF_{\pi} := \left(\pi Z + Z^2 o_L\dcroc{Z}\right) \hsp\cap\hsp \left(Z^q + \pi o_L\dcroc{Z}\right)\]
be the set of possible \emph{Frobenius power series}. Recall \cite[Theorem 8.1.1]{Lan}\footnote{Note that what Lang calls a formal group should really be called a \emph{formal group law}.}. that for every Frobenius power series $\varphi(Z) \in \sF_{\pi}$, there is a unique formal group law $F_{\varphi(Z)} = Z_1 +_{\cG} Z_2 \in o_L\dcroc{Z_1,Z_2}$ such that $\varphi(Z)$ is an endomorphism of $F_{\varphi(Z)}$. Since we have fixed a coordinate $Z$ on the power series ring $o_L\dcroc{Z}$, this formal group law defines a \emph{formal group}\footnote{a group object in the category of formal schemes over $\Spf{o_L}$} $(\cG, \oplus)$ on the underlying formal affine scheme $\Spf o_L\dcroc{Z}$. This formal group is called a \emph{Lubin-Tate formal group}. Up to isomorphism of formal groups, it does not depend on the choice of the Frobenius power series $\varphi(Z)$, however it does depend on the choice of $\pi$. The base change of $\cG$ to the completion $\widehat{L^{\ur}}$ of the maximal unramified extension $L^{\ur}$ of $L$ does not even depend on the choice of $\pi$. 

The Lubin-Tate formal group $\cG$ is in fact a \emph{formal $o_L$-module}. This means that there is a ring homomorphism $o_L \to \End(\cG)$, $a \mapsto [a](Z) \in o_L\dcroc{Z}$, such that $[a](Z) \equiv aZ \mod Z^2 o_L\dcroc{Z}$ for all $a \in o_L$. In other words, the formal group $\cG$ admits an action of $o_L$ by endomorphisms of formal groups, in such a way that the differential of this action at the identity element $1$ of $\cG$ agrees with the natural $o_L$-action on the cotangent space of $\cG$ at $1$. The action of $\pi \in o_L$ is given by the power series $[\pi](Z) = \varphi(Z)$.

\subsection{A review of \ts{p}-divisible groups}
In his seminal paper \cite{Tate66}, Tate introduced \emph{$p$-divisible groups} and considered their relation to formal groups. Here we review some of his fundamental theorems. 

Let $R$ be a commutative base ring and let $\Gamma = (\Spf \cA, \ast)$ is a commutative formal group over $R$ where $\cA = R\dcroc{X_1,\cdots,X_d}$ is a power series ring in $d$ variables over $R$. Then we can associate with $\Gamma$ the $p$-divisible group $\Gamma(p) = (\Gamma(p)_n, i_n)$ over $R$ where $\Gamma(p)_n := \Gamma[p^n]$ is the subgroup of elements of $\Gamma$ killed by $p^n$. More precisely, let $\psi : \cA \to \cA$ be the continuous $R$-algebra homomorphism which corresponds to multiplication by $p$ on $\Gamma$ and let $J_n$ be the ideal $\cA\psi^n(X_1) + \cdots + \cA\psi^n(X_d)$ of $\cA$; then $\cA/J_n$ is a Hopf algebra over $R$ free of finite rank over $R$, and $\Gamma(p)_n = \Spec (\cA / J_n)$ is the corresponding commutative finite flat group scheme over $R$. The closed immersions $i_n : \Gamma(p)_n \to \Gamma(p)_{n+1}$ are obtained from the $R$-algebra surjections $\cA / J_{n+1} \twoheadrightarrow \cA/J_n$. 

\begin{theorem}[\S 2.2, Proposition 1 \cite{Tate66}]\label{thm: Tate1} Let $R$ be a complete Noetherian ring whose residue field $k$ is of characteristic $p > 0$. Then $\Gamma \mapsto \Gamma(p)$ is an equivalence between the category of \emph{divisible} commutative formal groups over $R$ and the category of \emph{connected} $p$-divisible groups over $R$.
\end{theorem}

Recall that the formal group $\Gamma$ is said to be \emph{divisible} if $\cA / J_1$ is finitely generated as an $R$-module, and a $p$-divisble group $(\Gamma_n, i_n)$ is said to be \emph{connected} if every finite flat group scheme $\Gamma_n$ is a connected scheme. 

\begin{remark}\label{rem:GeneralFF} Inspecting the proof of \cite[Proposition 1]{Tate66}, we see that the fact that the functor $\Gamma \mapsto \Gamma(p)$ is fully faithful holds in greater generality: if $R$ is \emph{any} commutative ring and $G,H$ are divisible formal groups defined over $R$ such that $\cO(G)$ and $\cO(H)$ are power series rings in finitely many variables over $R$, then the natural map 
\[ \Hom_{R-\fgp}(G,H) \to \Hom_{p-\div}(G(p), H(p))\]
is a bijection. \end{remark}

Now we specialise to the case where $R$ is our complete discrete valuation ring $o_L$. The \emph{Tate module} associated to a $p$-divisible group $\Gamma = (\Gamma_n, i_n)$ is by definition
\[ T(\Gamma) := \varprojlim \Gamma_n(\overline{L})\]
where $\overline{L}$ is the algebraic closure of $L$, $\Gamma_n(\overline{L}) = \Hom_{o_L-\alg}(\cO(\Gamma_n), \overline{L})$ is the set of $\overline{L}$-points of $\Gamma_n$, and the connecting maps in the inverse limit are induced by the multiplication-by-$p$-maps $j_n : \Gamma_{n+1} \to \Gamma_n$. By functoriality, the Tate module $T(\Gamma)$ carries a natural action of the absolute Galois group $G_L = \Gal(\overline{L}/L)$, making $T(\Gamma)$ into a continuous $\Zp$-linear representation of $G_L$ of rank equal to the \emph{height} $h$ of $\Gamma$. Remarkably, it turns out that this Galois representation completely determines the $p$-divisible group $\Gamma$. More precisely, we have the following

\begin{theorem}[\S 4.2, Corollary 1 \cite{Tate66}] \label{thm: TateFF}The functor $\Gamma \mapsto T(\Gamma)$ is a fully faithful embedding of the category of $p$-divisible groups over $o_L$ into the category of finite rank $\Zp$-linear continuous representations of $G_L$.
\end{theorem}

\subsection{Cartier duality for $p$-divisible groups} The category of commutative finite flat group $R$-schemes admits a duality called \emph{Cartier duality}: if $G$ is a commutative finite flat group scheme over $R$, then its Cartier dual is defined by $G^\vee = \Spec( \cO(G)^\ast )$ where $\cO(G)^\ast := \Hom_R( \cO(G), R)$ is the $R$-linear dual of the coordinate ring $\cO(G)$. The group structure on $G^\vee$ is obtained by dualising the multiplication map on $\cO(G)$ and the scheme structure on $G^\vee$ is obtained by dualising the comultiplication map on $\cO(G)$ encoding the group structure on $G$.

Tate shows in \cite[\S 2.3]{Tate66} that Cartier duality extends naturally to a duality $\Gamma \mapsto \Gamma^\vee$ on the category of $p$-divisible groups. He also shows that in \cite[\S 4]{Tate66} when $R = o_L$, the Tate-module functor to Galois representations converts Cartier duality into what is now called \emph{Tate duality} on Galois representations, namely $V \mapsto \Hom(V, \Zp(1))$. In other words, there is a natural isomorphism of continuous $G_L$-representations on finite rank $\Zp$-modules
\[ T(\Gamma^\vee) \cong \Hom_{\Zp}(T(\Gamma), \Zp(1))\]
where $\Zp(1) := T(\h{\bG}_m(p))$ is the Tate module associated to the formal multiplicative group $\widehat{\bG}_m$, the formal completion at the identity of the group scheme $\bG_m := \Spec o_L[T,T^{-1}]$.

\subsection{The character $\tau : G_L \to o_L^\times$ and the period $\Omega$}
We return to the Lubin-Tate formal group $\cG$ as in $\S \ref{sect:LT}$, which is easily seen to be divisible. Because $\cG$ is a formal $o_L$-module, the functoriality of $T(-)$ implies that the Tate module $T(\cG(p))$ of the $p$-divisible group $\cG(p)$ associated with $\cG$ is actually an $o_L$-module. It is a fundamental fact due to Lubin and Tate --- see \cite[Theorem 2]{LT65} --- that $T(\cG(p))$ is a free $o_L$-module of rank one. Since $o_L$ is itself a free $\Zp$-module of rank $d = [L:\Qp]$, it follows that the underlying $\Zp$-module of $T(\cG(p)^\vee) \cong \Hom_{\Zp}(T(\cG(p)), \Zp)$ is free of rank $d$ as a $\Zp$-module as well. Since it is also an $o_L$-module by the functoriality of $\Hom_{\Zp}(-,\Zp)$, we see that $T(\cG(p)^\vee)$ is also a free $o_L$-module of rank $1$. 

On the way to his proof of Theorem \ref{thm: TateFF}, Tate explains how to compute $T(\cG(p)^\vee)$: using Cartier duality, on \cite[p. 177]{Tate66} he obtains a natural isomorphism of abelian groups
\begin{equation}\label{eq:Tate177} T(\cG(p)^\vee) \cong \Hom_{p-\div/o_{\Cp}}( \cG(p) \times_{o_L} o_{\Cp}, \widehat{\bG}_m(p) \times_{o_L} o_{\Cp}).\end{equation}
On the other hand, applying Remark \ref{rem:GeneralFF} with $R = o_{\Cp}$, we see that the natural map
\begin{equation}\label{eq:genFF} \Hom_{\fgp/o_{\Cp}}(\cG \times_{o_L} o_{\Cp},\widehat{\bG}_m \times_{o_L} o_{\Cp}) \to \Hom_{p-\div/o_{\Cp}}( \cG(p) \times_{o_L} o_{\Cp}, \widehat{\bG}_m(p) \times_{o_L} o_{\Cp})\end{equation}
is a bijection. As a consequence, we see that $ \Hom_{\fgp/o_{\Cp}}(\cG \times_{o_L} o_{\Cp},\widehat{\bG}_m \times_{o_L} o_{\Cp})$ is free of rank $1$ as an $o_L$-module.
\begin{definition}\label{def:dualLTgenerator} \hsp
\begin{enumerate}
\item We fix a generator $t'_o$ for $T(\cG(p)^\vee)$ as an $o_L$-module. 
\item We let $F_{t'_o}$ be the generator for the $o_L$-module $ \Hom_{\fgp/o_{\Cp}}(\cG \times_{o_L} o_{\Cp},\widehat{\bG}_m \times_{o_L} o_{\Cp})$, which corresponds to $t'_o$ along the isomorphism 
\[T(\cG(p)^\vee) \stackrel{\cong}{\to} \Hom_{\fgp/o_{\Cp}}(\cG \times_{o_L} o_{\Cp},\widehat{\bG}_m \times_{o_L} o_{\Cp})\]
obtained by combining (\ref{eq:Tate177}) and (\ref{eq:genFF}).
\item We let $\tau : G_L \to o_L^\times$ be the character afforded by the free rank $1$ $o_L$-module $T(\cG(p)^\vee)$:
\[ \sigma( t'_o ) = \tau(\sigma) t'_o \quad\mbox{for all}\quad \sigma \in G_L.\]
\end{enumerate}
\end{definition}
The morphism of formal groups $F_{t'_o} : \cG \times_{o_L} o_{\Cp} \to \widehat{\bG}_m \times_{o_L} o_{\Cp}$ is an element of
\[ F_{t'_o}(Z) \in \cO(\cG \times_{o_L} o_{\Cp}) = o_{\Cp}\dcroc{Z}.\]
Then $1 + F_{t'_o}(Z)$ is ``grouplike" in the topological Hopf algebra $o_{\Cp}\dcroc{Z}$: it satisfies the relation
\[1 + F_{t'_o}(Z_1 +_{\cG} Z_2) = (1 + F_{t'_o}(Z_1))(1 + F_{t'_o}(Z_2)).\]
When we further base change the formal group $\cG \times_{o_L} o_{\Cp}$ to $\Cp$, it becomes isomorphic to the additive formal group. It follows from this that $\log F_{t'_o}(Z)$ is necessarily ``primitive'' in the topological Hopf algebra $\Cp\dcroc{Z}$: it satisfies the relation
\[\log (1 + F_{t'_o}(Z_1 +_{\cG} Z_2)) = \log (1 + F_{t'_o}(Z_1)) + \log (1 + F_{t'_o}(Z_2)).\]
Since the \emph{logarithm} $\log_{\LT}(Z)$ of the formal group $\cG$ spans the space of primitive elements in $\Cp\dcroc{Z}$, it follows that there exists a unique element $\Omega \in \Cp$ such that
\[ 1 + F_{t'_o}(Z) = \exp(\Omega \log_{\LT}(Z)).\]
\begin{definition} The element $\Omega$ is called the \emph{period} of the dual $p$-divisible group $\cG(p)^\vee$. \end{definition}

Let $I_L \subseteq G_L$ denote the inertia subgroup.

\begin{lemma}\label{opentau}
If $L \neq \Qp$, then the character $\tau : I_L \to o_L^\times$ has an open image.
\end{lemma}
\begin{proof}
Let $\chi_\pi$ be the character describing the $G_L$-action on the Tate module $T$ of $\mathcal{G}$. By local class field theory we know that on $I_L$, $\Norm_{L/\mathbb{Q}_p} \circ \chi_\pi = \chi_{\cyc}$, the cyclotomic character. From Definition \ref{def:dualLTgenerator}(2), we have $\tau = \chi_\pi^{-1} \cdot \chi_{\cyc}$. Hence $\tau : I_L \to o_L^\times$ is the composition of the surjective map $\chi_\pi : I_L \to o_L^\times$ and of the map given by $x \mapsto \prod_{\sigma : L \to \Qpbar, \ \sigma \neq \Id} \sigma(x)$.

On the Lie algebra $L$ of $o_L^\times$, the derivative of the above map is given by $U=
\Tr_{L/\Qp} - \Id$. We prove that $U: L \to L$ is injective, hence surjective, which implies the lemma. If $U(x)=0$, then $x = (U+\Id)x = \Tr_{L/\Qp}(x) \in \Qp$ and hence $U(x)=([L:\Qp]-1)x$ so that $x=0$.
\end{proof}

For future use, we record here the more precise result due to B. Xie which gives a sufficient criterion for $\tau$ to be surjective.  

\begin{lemma}\label{lem:tau-sur}
  If $d-1$ and $(p-1)p$ are coprime, then $\tau : I_L \to o_L^\times$ is surjective.
\end{lemma}
\begin{proof}
Since $\tau = \chi_\pi^{-1} \cdot \chi_{\cyc}$ and $\chi_{\cyc} = \Norm_{L/\mathbb{Q}_p} \circ \chi_\pi$, we have
\begin{equation*}
  \tau(g) = \chi_\pi(g)^{-1} \Norm_{L/\mathbb{Q}_p}(\chi_\pi(g))  \quad\text{for any $g \in I_L$}.
\end{equation*}
Note also that the restriction to $I_L$ of the totally ramified surjective character $\chi_\pi \twoheadrightarrow o_L^\times$ is still surjective. Let now $u \in o_L^\times$ be any fixed element.

We first show that there is an $a \in \mathbb{Z}_p^\times$ such that $a^{d-1} = \Norm_{L/\mathbb{Q}_p}(u)$. Let $v := \Norm_{L/\mathbb{Q}_p}(u)$ and let $\bar{v}$ denote its image in $\mathbb{F}_p^\times$. By our assumption the polynomial $Z^{d-1} - \bar{v}$ is separable over $\mathbb{F}_p$ and has a root in $\mathbb{F}_p^\times$. Hence Hensel's lemma implies that the polynomial $Z^{d-1} - v$ has a root $a \in \mathbb{Z}_p^\times$.

Choosing now a $g \in I_L$ such that $\chi_\pi(g) = a u^{-1}$ we deduce that
\begin{equation*}
  \tau(g) = (au^{-1})^{-1} \Norm_{L/\mathbb{Q}_p}(au^{-1}) = ua^{-1} a^d \Norm_{L/\mathbb{Q}_p}(u^{-1}) = u \ . \qedhere
\end{equation*}
\end{proof}

\subsection{The Amice-Katz transform}\label{sec:AKtrans}
With the period $\Omega \in \Cp$ in hand, now we recall some constructions from $p$-adic Fourier Theory \cite{ST}. For each $a \in o_L$, define
\[\Delta_a := 1 + F_{a t'_o}(Z) = \exp( a \Omega \log_{\LT}(Z)) \in \Cp\dcroc{Z}^\times.\]
The map $(o_L,+) \to (\Cp\dcroc{Z}^\times, \times)$ which sends $a \in o_L$ to $\Delta_a$ is  a group homomorphism. The fundamental property of these power series is that their coefficients all lie in $o_{\Cp}$:
\[\Delta_a \in o_{\Cp}\dcroc{Z}^\times \qmb{for all} a \in o_L.\]
This follows from the fact that for each $a \in o_L$, $F_{at'_o} : \cG \times_{o_L} o_{\Cp} \to \widehat{\bG}_m \times_{o_L} o_{\Cp}$ is a homomorphism of formal groups defined over $o_{\Cp}$; see also \cite[Lemma 4.2(5)]{ST}.
\begin{definition}\label{DualTowerDef}\hsp \begin{enumerate}
\item Let $L_\infty$ be the closure in $\Cp$ of the subfield $L(\Omega)$ of $\Cp$ generated by $L$ and $\Omega$. 
\item Let $L_\tau := L_\infty \cap \overline{L}$.
\item Let $o_\infty:= L_\infty \cap o_{\Cp}$.
\item Let $o_\tau := L_\tau \cap o_{\Cp}$.
\end{enumerate}
\end{definition}

\begin{lemma}\label{lem:LOmegaDense} We have $L_\infty = \Cp^{\ker \tau}$ and $o_\infty = o_{\Cp}^{\ker \tau}$.
\end{lemma}
\begin{proof}  From the relation appearing in Definition \ref{def:dualLTgenerator}(3), we deduce 
\[ \sigma(\Omega) = \tau(\sigma)\Omega \qmb{for all} \sigma \in G_L.\]
This immediately implies that $L_\infty \subseteq \Cp^{\ker \tau}$. Let $H := \Gal(\overline{L} / L_\tau)$, a closed subgroup of $G_L$, and let $g \in H$. Then $g$ extends to a unique continuous $L_\tau$-linear automorphism $g$ of $\Cp$. Now $L_\infty$ is the closure of $L_\tau$ in $\Cp$, so $g$ fixes $\Omega \in L_\infty$. Hence $\tau(g) = 1$ by the above relation. Hence $H \leq \ker \tau$ which implies that $\Cp^{\ker \tau} \leq \Cp^H$. But $\overline{L}^H$ is dense in $\Cp^H$ by the Ax-Sen-Tate theorem, \cite[Proposition 2.1.2]{BriCon}, and $\overline{L}^H = L_\tau$ by infinite Galois theory. Hence $L_\tau$ is dense in $\Cp^H$, so $\Cp^H$ is contained in the closure of $L_\tau$ in $\Cp$, namely $L_\infty$. Hence $\Cp^{\ker\tau} \leq L_\infty$. 

The second statement follows from the first by intersecting $L_\infty = \Cp^{\ker \tau}$ with $o_{\Cp}$.\end{proof}
It is clear from the definition of $\Delta_a$ that in fact 
\[\Delta_a \in o_\infty\dcroc{Z}^\times \qmb{for all} a \in o_L.\] 
\begin{definition} We write $o_L\dcroc{o_L}$ for the completed group ring of the abelian group $o_L$ with coefficients in $o_L$.  The \emph{Amice-Katz transform} is the unique extension to a continuous $o_L$-algebra homomorphism
\[ \mu : o_L\dcroc{o_L} \to \cO(\cG \times_{o_L} o_\infty) = o_\infty\dcroc{Z}\]
of the group homomorphism $o_L \to o_{\Cp}\dcroc{Z}^\times$ which sends $a \in o_L$ to $\Delta_a \in o_\infty\dcroc{Z}^\times$. \end{definition} 

\subsection{The Schneider-Teitelbaum uniformisation} 
At this point, rigid analytic geometry enters the picture. Let $\bfB$ be the rigid $L_\infty$-analytic open disc of radius one, with local coordinate $Z$. By definition, $\bfB$ is the colimit of the rigid $L_\infty$-analytic closed discs $\bfB(r)$ of radius $r < 1$, as $r \in |L_\infty^\times|$ approaches $1$ from below:
\[ \bfB = \colim\limits_{r < 1} \bfB(r), \quad \bfB(r) = \Sp L_\infty \langle Z / \dot{r} \rangle\]
where $\dot{r}$ is any choice of an element of $L_\infty^\times$ such that $|\dot{r}| = r$. Choosing, for convenience, any strictly increasing sequence $r_1 < r_2 < r_3 < \cdots $ of real numbers in $|L_\infty| \cap (0,1)$ approaching $1$ from below, we have a descending chain of $L_\infty$-algebras, each one containing $o_\infty\dcroc{Z}$:
\[ L_\infty \langle Z / \dot{r}_1 \rangle \supsetneq L_\infty \langle Z / \dot{r}_2 \rangle \supsetneq L_\infty \langle Z / \dot{r}_3 \rangle \supsetneq \cdots \supsetneq \bigcap\limits_{n=1}^\infty L_\infty \langle Z / \dot{r}_n \rangle = \cO(\bfB) \supseteq o_\infty\dcroc{Z} \otimes_{o_L}L.\]
With this notation in place, it follows from one of Schneider-Teitelbaum's main results, \cite[Theorem 3.6]{ST}, that the $o_L$-algebra homomorphism $\mu : o_L\dcroc{o_L} \to o_\infty\dcroc{Z}$ extends to a continuous \emph{isomorphism} of $L$-Fr\'echet algebras
\[ \mu_{\rig} : D^{L-\an}(o_L, L_\infty) \stackrel{\cong}{\longrightarrow} \cO(\bfB)\]
which makes the following diagram commutative:
\[ \xymatrix{ o_L\dcroc{o_L} \otimes_{o_L} L \ar[rr]^\mu\ar[d] &&\ar[d] o_\infty\dcroc{Z} \otimes_{o_L} L\ar[d] \\ D^{L-\an}(o_L,L_\infty) \ar[rr]^{\cong}_{\mu_{\rig}} &&  \cO(\bfB)} \]
The vertical arrow on the left is the natural restriction map  $o_L\dcroc{o_L} \otimes_{o_L} L$ into $D^{L-\an}(o_L,L)$, witnessing the fact that every locally $L$-analytic function on $o_L$ is continuous, and hence that every continuous distribution on $o_L$ restricts to a locally $L$-analytic distribution on $o_L$; see \cite{ST2} for more details. The vertical arrow on the right is the inclusion $o_\infty\dcroc{Z} \otimes_{o_L}L \subset \cO(\cB)$ from the above discussion. Combining the isomorphism $\mu_{\rig}$ with the Fourier transform ${\cF : D^{L-\an}(o_L, L_\infty) \to \cO(\frX \times_L L_\infty)}$, we obtain an isomorphism of $L_\infty$-Fr\'echet algebras
\[ \mu_{\rig} \circ \cF : \cO(\frX \times_L L_\infty) \stackrel{\cong}{\longrightarrow} \cO(\bfB).\]
Since $\frX \times_L L_\infty$ and $\bfB$ are both Stein rigid analytic varieties over $L_\infty$, this isomorphism determines, and is completely determined by, an isomorphism
\[ \kappa := \Sp(\mu_{\rig} \circ \cF) : \bfB \stackrel{\cong}{\longrightarrow} \frX \times_L L_\infty.\]
This is a version of \cite[Theorem 3.6]{ST}: the base-change of the character variety $\frX$ to $L_\infty$ is isomorphic to the rigid $L_\infty$-analytic open disc of radius one, so $\kappa$ can be viewed as giving a \emph{uniformisation} of $\frX \times_L L_\infty$ by $\bfB$. Schneider and Teitelbaum also show that the morphism $\kappa$ is given on $\Cp$-points by the following rule: for each $z \in \bfB(\Cp)$ we can evaluate the power series $\Delta_a \in o_\infty\dcroc{Z}$ at $Z = z$ to obtain an element $\Delta_a(z) \in o_{\Cp}^\times$, and the locally $L$-analytic character $\kappa(z) : o_L \to \Cp$ is given by
\[ \kappa(z)(a) = \Delta_a(z) \qmb{for all} a \in o_L.\]

\subsection{$\Lambda_L(\frX)$ and the twisted $G_L$-action on $\Cp\dcroc{Z}$} It is natural to enquire, in the light of the Schneider-Teitelbaum isomorphism
\[ \kappa : \bfB \stackrel{\cong}{\longrightarrow} \frX \times_L L_\infty\]
how far the character variety $\frX$ is itself from being isomorphic to an open rigid $L$-analytic unit disc. For general reasons, $\frX \times_L L_\infty$ carries a natural action of the Galois group $G_L$, acting on the second factor, giving an isomorphism of $L$-Fr\'echet algebras
\[ \cO(\frX) \cong \cO(\frX \times_L L_\infty)^{G_L}.\]
\begin{definition} The \emph{twisted $G_L$}-action on $\cO(\bfB)$ is given as follows:
\[ \sigma \ast F(Z) := ({}^\sigma F)([\tau(\sigma)^{-1}](Z)) \qmb{for all} F(Z) \in \cO(\bfB), \sigma \in G_L.\]
\end{definition}
Here $F \mapsto {}^\sigma F$ is the ''coefficient-wise" $G_L$-action on $\Cp\dcroc{Z} \supset \cO(\bfB)$, given explicitly by ${}^\sigma(\sum\limits_{n=0}^\infty a_n Z^n) = \sum\limits_{n=0}^\infty \sigma(a_n) Z^n$ for all $\sigma \in G_L$. 

Schneider and Teitelbaum showed that this twisted $G_L$-action on $\cO(\bfB)$ in fact comes from the following twisted $G_L$-action on the set of $\Cp$-points $\bfB(\Cp)$:
\[ \sigma \ast z = \kappa^{-1}(\sigma \circ \kappa(z)) \qmb{for all} z \in \bfB(\Cp), \sigma \in G_L.\]
From the proof of \cite[Corollary 3.8]{ST}, we can also deduce the following 
\begin{proposition}\label{prop:TwistedAction} The algebra isomorphism $\kappa^\ast = \mu_{\rig} \circ \cF : \cO(\frX \times_L L_\infty) \stackrel{\cong}{\longrightarrow} \cO(\bfB)$ is equivariant with respect to the natural $G_L$-action on the source, and the twisted $G_L$-action on the target.
\end{proposition}

\begin{corollary} The map $\mu_{\rig}$ restricts to give an isomorphism of $o_L$-algebras
\[ (\mu_{\rig} \circ \cF)^\circ : \cO^\circ(\frX) \stackrel{\cong}{\longrightarrow} o_\infty\dcroc{Z}^{G_L, \ast}.\]
\end{corollary}
\begin{proof} Applying the functor $\cO^\circ$ to the isomorphism of rigid $L_\infty$-analytic varieties $\kappa : \bfB \to \frX \times_L L_\infty$, we see that $\mu_{\rig} \circ \cF$ restricts to an $o_\infty$-algebra isomorphism
\[ \cO(\frX \times_L L_\infty)^\circ \stackrel{\cong}{\longrightarrow} \cO(\bfB)^\circ.\]
It is well known that $\cO(\bfB)^\circ = o_\infty\dcroc{Z}$ and that $\Lambda_L(\frX) = \cO(\frX)^\circ = (\cO(\frX \times_L L_\infty)^\circ)^{G_L}$. The result follows by passing to $G_L$-invariants and applying Proposition \ref{prop:TwistedAction}.
\end{proof}
Consequently, the image of the Amice-Katz transform $\mu : o_L\dcroc{o_L} \to o_\infty\dcroc{Z}$ lands in the subring of twisted $G_L$-invariants. Our main goal in this paper is to study the following 
\begin{question} Is the Amice-Katz transform $\mu : o_L\dcroc{o_L} \to o_\infty\dcroc{Z}^{G_L, \ast}$ an isomorphism? \end{question}

\subsection{Some properties of $\Lambda_L(\frX)$}
\label{subextraprop}

Recall that $\Lambda_L(\frX)$ is the ring $\mathcal{O}_L^{\leq 1}(\mathfrak{X}) = o_\infty\dcroc{Z}^{G_L, \ast}$. From \cite{BSX} we know (through the LT-isomorphism) that $\Lambda_L(\frX)$ is an integral domain and that the norm $\|\ \|_{\mathfrak{X}} = \|\ \|_1$ on $\Lambda_L(\frX)$ is multiplicative.

\begin{lemma}\label{residue}
If $L \neq \Qp$ and if $K$ is a finite extension of $L$, then $\overline{k}\dcroc{Z}^{G_K,*} = k_K$.
\end{lemma}

\begin{proof}
If $g \in I_K$, then $g$ acts trivially on $\overline{k}$, so that the $G_{L,*}$ action of $g \in I_K$ on $\overline{k}\dcroc{Z}$ is given by $g : \sum_{n \geq 0} a_n Z^n \mapsto \sum_{n \geq 0} a_n ([\tau(g)^{-1}]Z)^n$. The character $\tau : I_K \to o_L^\times$ has an open image by lemma \ref{opentau}. This image therefore contains $\chi_\pi(I_M)$ where $M \subset L_\infty$ is some finite extension of $L$, and $\overline{k}\dcroc{Z}^{I_K,*} = \overline{k}\dcroc{Z}^{I_M}$ where $I_M$ acts on $\overline{k}\dcroc{Z}$ via $g : \sum_{n \geq 0} a_n Z^n \mapsto \sum_{n \geq 0} a_n ([\chi_\pi(g)]Z)^n$. We know from the theory of the field of norms that $\overline{k}\dcroc{Z}$ with that action of $I_M$ embeds into $\tilde{\mathbf{E}}^+ \simeq \varprojlim_{(-)^q} o_{\Cp}$ in an $I_M$-equivariant way. Let $P := \Cp^{I_M}$. We have $(\tilde{\mathbf{E}}^+)^{I_M} \simeq \varprojlim_{(-)^q} o_P = \overline{k}$ since $P/\Qp$ is finitely ramified. Hence $\overline{k}\dcroc{Z}^{I_M} = \overline{k}$ and $\overline{k}\dcroc{Z}^{I_K,*} = \overline{k}$. The lemma then follows from the fact that on $\overline{k}$, the twisted $G_L$-action coincides with the usual $G_L$-action, so that $\overline{k}^{G_K,*} = k_K$.
\end{proof}

We have a surjective map $\Lambda_L(\frX) \to k$ given by $f \mapsto f(\chi_{\mathrm{\triv}}) \bmod \mathfrak{m}_L$. Its kernel $\mathfrak{m}(\mathfrak{X}) := \{f \in \Lambda_L(\frX) : f(\chi_{\triv}) \in \mathfrak{m}_L\}$ is a maximal ideal of $\Lambda_L(\frX)$, with residue field $k$. Lemma \ref{residue} above implies that $\mathfrak{m}(\mathfrak{X}) = \mathcal{O}_{\mathfrak{m}_{\Cp}}(\mathfrak{X})^{G_L,*}$.

\begin{lemma}\label{local}
The ring $\Lambda_L(\frX)$ is a local ring.
\end{lemma}
\begin{proof}
We have to show that $\mathfrak{m}(\mathfrak{X})$ is the unique maximal ideal, i.e., that $f$ is a unit in $\Lambda_L(\frX)$ if and only if $f(\chi_{\triv}) \in o_L^\times$. The direct implication is obvious. We therefore assume that $f(\chi_{\triv}) \in o_L^\times$. The image $F(Z) \in o_{\mathbf{C}_p}\dcroc{Z}$ of $f$ under the LT-isomorphism then satisfies $F(0) \in o_L^\times$ and hence is a unit in $o_{\mathbf{C}_p}\dcroc{Z}$. We deduce that $f$ is a unit in $\mathcal{O}_{\Cp}(\mathfrak{X})$. Since the twisted $G_L$-action must fix with $f$ also its inverse we obtain that $f$ is a unit in $\mathcal{O}_L(\mathfrak{X})$ and hence in $\mathcal{O}^b_L(\mathfrak{X})$ by \cite{BSX} Cor.\ 1.24. The multiplicativity of the norm $\|\ \|_{\mathfrak{X}}$ finally implies that $1 = \| f \|_{\mathfrak{X}} = \| f^{-1} \|_{\mathfrak{X}}$.
\end{proof}

The $o_L$-algebra $\Lambda_L(\frX)$ carries two natural topologies. One is the $p$-adic topology which is induced by the norm $\|\ \|_{\mathfrak{X}}$. The other is the topology induced by the Frechet topology of $\mathcal{O}_L(\mathfrak{X})$. We will call the latter the weak topology on $\Lambda_L(\frX)$.

\begin{remark}\label{padic-weak}
The weak topology on $\Lambda_L(\frX)$ is coarser than the $p$-adic topology.
\end{remark}
\begin{proof}
Let $\mathfrak{X} = \bigcup_{n \geq 1} \mathfrak{X}_n$ be a Stein covering by affinoid subdomains $\mathfrak{X}_n$ (cf.\ \cite{BSX} \S1.3). The Frechet topology of $\mathcal{O}_L(\mathfrak{X})$ is the projective limit of the Banach topologies on the affinoid algebras $\mathcal{O}_L(\mathfrak{X}_n)$. Since $\mathfrak{X}$ is reduced these Banach topologies are defined by the respective supremum norm (cf.\ \cite{BGR} Thm.\ 6.2.4/1). Therefore the Banach topology on $\mathcal{O}_L(\mathfrak{X}_n)$ induces on its unit ball with respect to the supremum norm the $p$-adic topology. It follows that the natural maps $\Lambda_L(\frX) \rightarrow \mathcal{O}_L(\mathfrak{X}_n)$ are continuous for the $p$-adic topology on the source and the Banach topology on the target. Therefore the inclusion $\Lambda_L(\frX) \subseteq \mathcal{O}_L(\mathfrak{X})$ is continuous for the $p$-adic topology on the source and the Frechet topology on the target.
\end{proof}

\begin{lemma}\label{padic-complete}
$\Lambda_L(\frX)$ is $p$-adically separated and complete.
\end{lemma}
\begin{proof}
We show that, for any reduced rigid analytic variety $\mathfrak{Y}$ over $L$, the ring $\mathcal{O}_L^{\leq 1}(\mathfrak{Y})$ of holomorphic functions bounded by $1$ is $p$-adically separated and complete. Let $\mathfrak{Y} = \bigcup_{i \in I} \mathfrak{Y}_i$ be an admissible covering by affinoid subdomains. Since $\mathfrak{Y}$ is assumed to be reduced, the supremum seminorm on each $\mathcal{O}_L(\mathfrak{Y}_i)$ is a norm and defines its affinoid Banach topology (cf.\ \cite{BSX} \S1.3). Hence $\| \ \|_{\mathfrak{Y}}$ is a norm on $\mathcal{O}_L^b(\mathfrak{Y})$ and defines the $p$-adic topology on $\mathcal{O}_L^{\leq 1}(\mathfrak{Y})$. In particular, the $p$-adic topology on $\mathcal{O}_L^{\leq 1}(\mathfrak{Y})$ is separated. Now let $(f_n)_n$ be a Cauchy sequence for $\| \ \|_{\mathfrak{Y}}$ in $\mathcal{O}_L^{\leq 1}(\mathfrak{Y})$. It restricts to a Cauchy sequence in $\mathcal{O}_L^{\leq 1}(\mathfrak{Y}_i)$ for each $i \in I$ which converges to a function $g_i \in \mathcal{O}_L^{\leq 1}(\mathfrak{Y}_i)$. Obviously the $g_i$ glue to a function $g \in \mathcal{O}_L^{\leq 1}(\mathfrak{Y})$. We have to show that the sequence $(f_n)_n$ converges to $g$ with respect to $\| \ \|_{\mathfrak{Y}}$. Let $\epsilon > 0$ be arbitrary. First we find an integer $N > 0$ such that $\| f_m - f_n\|_{\mathfrak{Y}} < \epsilon$ for all $m,n > N$. Secondly, for any $i \in I$, we have $\|g - f_m\|_{\mathfrak{Y}_i} < \epsilon$ for all sufficiently large (depending on $i$) $m$. It follows that $\|g - f_n\|_{\mathfrak{Y}_i} \leq \max(\|g - f_m\|_{\mathfrak{Y}_i}, \| f_m - f_n\|_{\mathfrak{Y}_i}) \leq \max(\|g - f_m\|_{\mathfrak{Y}_i}, \| f_m - f_n\|_{\mathfrak{Y}}) < \epsilon$ for any $n > N$ and any $i \in I$. Hence $\|g - f_n\|_{\mathfrak{Y}} \leq \epsilon$ for any $n > N$.
\end{proof}

\begin{proposition}\label{weak-compact}
$\Lambda_L(\frX)$ is compact in the weak topology.
\end{proposition}
\begin{proof}
According to \cite{Eme} Prop.\ 6.4.5 the space $\mathfrak{X}$ is strictly quasi-Stein. This means that a Stein covering $\mathfrak{X} = \bigcup_{n \geq 1} \mathfrak{X}_n$ can be chosen such that the inclusion maps $\mathfrak{X}_n \subseteq \mathfrak{X}_{n+1}$ are relatively compact. By loc.\ cit.\ Prop.\ 2.1.16 this implies that the restriction maps $\mathcal{O}_L(\mathfrak{X}_{n+1}) \rightarrow \mathcal{O}_L(\mathfrak{X}_n)$, which we simply view as inclusions, are compact maps between Banach spaces. Working over a locally compact field we deduce (cf.\ \cite{SchNFA} Remark 16.3 and \cite{PGS} Cor.\ 6.1.14) that the closure $C_n$ of $\mathcal{O}^{\leq 1}_L(\mathfrak{X}_{n+1})$ in $\mathcal{O}_L(\mathfrak{X}_n)$ is compact. We, of course, have $\Lambda_L(\frX) \subseteq \mathcal{O}^{\leq 1}_L(\mathfrak{X}_{n+1}) \subseteq C_n$. Therefore, if $L_n \subseteq \mathcal{O}_L(\mathfrak{X}_n)$ is any open lattice, then the $o_L$-modules $\Lambda_L(\frX) / \Lambda_L(\frX) \cap L_n \subseteq C_n / C_n \cap L_n$ are finite. It is straightforward to see that then $\Lambda_L(\frX) / \Lambda_L(\frX) \cap L$ must be finite for any open lattice $L \subseteq \mathcal{O}_L(\mathfrak{X})$. On the other hand $\Lambda_L(\frX)$ is weakly closed in $\mathcal{O}_L(\mathfrak{X})$ and hence is weakly complete. It follows (cf.\ \cite{SchNFA} Cor.\ 7.6) that $\Lambda_L(\frX)$ with its weak topology is the projective limit of the finite groups $\Lambda_L(\frX) / \Lambda_L(\frX) \cap L$ and hence is compact.
\end{proof}

\begin{lemma}
\label{weak-madic}
\hspace{10pt}
\begin{enumerate}
  \item Any open neighbourhood of zero for the weak topology on $\Lambda_L(\frX)$ contains a power of the maximal ideal $\mathfrak{m}(\mathfrak{X})$.
  \item If the ideal $\mathfrak{m}(\mathfrak{X})$ is finitely generated then the weak topology on $\Lambda_L(\frX)$ coincides with the $\mathfrak{m}(\mathfrak{X})$-topology.
\end{enumerate}
\end{lemma}

\begin{proof}
We have $\mathfrak{m}(\mathfrak{X}) = \pi_L \Lambda_L(\frX) + \mathfrak{n}$, where $\mathfrak{n}$ denotes the ideal of all functions in $\Lambda_L(\frX)$ which vanish in $\chi_{\triv}$. We consider the divisor $\Delta$ on $\mathfrak{X}$ which maps $\chi_{\triv}$ to $1$ and all other points to zero. For any integer $m \geq 1$ we have the ideal $I_{m\Delta} \subseteq \mathcal{O}_L(\mathfrak{X})$ corresponding to the divisor $m\Delta$. As a consequence of \cite{BSX} Prop.\ 1.4 these ideals are closed in $\mathcal{O}_L(\mathfrak{X})$ and satisfy $\bigcap_m I_m = \{0\}$. Hence the ideals $I_m \cap \Lambda_L(\frX)$ are closed in $\Lambda_L(\frX)$ with zero intersection. Let now $U \subseteq \Lambda_L(\frX)$ be any fixed open neighbourhood of zero for the weak topology. Suppose that $I_m \cap \Lambda_L(\frX) \nsubseteqq U$ for any $m \geq 1$. We then may pick, for any $m \geq 1$, a function $f_m \in (I_m \cap \Lambda_L(\frX)) \setminus U$. According to Prop.\ \ref{weak-compact} the weak topology on $\Lambda_L(\frX)$ is compact. Hence the sequence $(f_m)_m$ has a convergent subsequence with a limit $f \in \Lambda_L(\frX)$. On the one hand we have $f_n \in I_m \cap \Lambda_L(\frX)$ for any $n \geq m$. Since $I_m \cap \Lambda_L(\frX)$ is closed it follows that $f \in I_m \cap \Lambda_L(\frX)$ for any $m \geq 1$. Therefore $f = 0$. But on the other hand all the $f_m$ and hence $f$ lie in the closed complement of the open subset $U$. This is a contradiction. We conclude that $\mathfrak{n}^m \subseteq I_m \cap \Lambda_L(\frX) \subseteq U$ for any sufficiently large $m$. As a consequence of Remark \ref{padic-weak} we also have $\pi_L^m \Lambda_L(\frX) \subseteq U$ for any sufficiently large $m$. Hence $\mathfrak{m}(\mathfrak{X})^{2m} \subseteq \pi_L^m \Lambda_L(\frX) + \mathfrak{n}^m \subseteq U$ for large $m$. This proves (1).

We have to show that the ideals $\mathfrak{m}(\mathfrak{X})^m$ are open for the weak topology. Under our assumption all ideals $\mathfrak{m}(\mathfrak{X})^m$, for $m \geq 1$, are finitely generated. Hence all $\mathfrak{m}(\mathfrak{X})^{m+1} / \mathfrak{m}(\mathfrak{X})^m$ are finite dimensional $k$-vector spaces. We see that each quotient $\Lambda_L(\frX) / \mathfrak{m}(\mathfrak{X})^m$, for $m \geq 1$, is a finite $o_L$-module. Hence it suffices to show that the ideal $\mathfrak{m}(\mathfrak{X})^m$ is closed for the weak topology. Let $f_1, \ldots, f_r$ be generators of $\mathfrak{m}(\mathfrak{X})^m$. Then $\mathfrak{m}(\mathfrak{X})^m$ is the image of the map  $\Lambda_L(\frX)^r  \rightarrow \Lambda_L(\frX)$ sending $(h_1, \ldots,h_r)$ to $\sum_i h_i f_i$, which is a continuous map between compact spaces by Prop.\ \ref{weak-compact}. This proves (2).
\end{proof}

\begin{remark}
Any $f \in \mathfrak{m}(\mathfrak{X})$ satisfies $\|f \|_{\mathfrak{X}_n} < 1$ for any $n$.
\end{remark}

\begin{proof}
If $\|f \|_{\mathfrak{X}_n} = 1$ then the maximum modulus principle for the affinoid $\mathfrak{X}_n$ implies that there is a point $z \in \mathfrak{X}_n$ such that $|f(z)|=1$. By considering $f$ as an element of $o_{\Cp}\dcroc{T}$, we see that $f(0)$ is a unit so that $f$ is not in $\mathfrak{m}(\mathfrak{X})$.
\end{proof}

Next we consider the injective map
\begin{equation*}
  \Lambda(o_L) = o_L\dcroc{o_L} \longrightarrow \Lambda_L(\frX) \ ,
\end{equation*}
which we treat as an inclusion. More explicitly, let $a_1, \ldots a_d$ be a basis of $o_L$ as a $\Zp$-module. Then the image of the above map is the ring of formal power series $o_L\dcroc{\delta_{a_1} - \delta_0, \ldots, \delta_{a_d} - \delta_0}$ inside $\Lambda_L(\frX)$. We immediately conclude from Lemma \ref{residue} that
\begin{equation*}
  \mathfrak{m}(\mathfrak{X}) \cap o_L\dcroc{o_L} = \langle \pi_L, \delta_{a_1} - \delta_0, \ldots, \delta_{a_d} - \delta_0 \rangle \subseteq o_L\dcroc{o_L} \ .
\end{equation*}

\begin{lemma}\label{prime-ideal}
  $\mathcal{O}_L^{< 1}(\mathfrak{X}) \cap o_L\dcroc{o_L} = \pi_L o_L\dcroc{o_L}$.
\end{lemma}
\begin{proof}  We  have $\pi_L o_L\dcroc{o_L} \subseteq P := \mathcal{O}_L^{< 1}(\mathfrak{X}) \cap o_L\dcroc{o_L}$. It follows that $\overline{P} := P/\pi_L o_L\dcroc{o_L}$ is a ``canonical'' prime ideal in the formal power series ring $k\dcroc{o_L}$: in particular, it is invariant for the $o_L^\times$ action on the mod-$p$ Iwasawa algebra $k\dcroc{o_L}$. It certainly is not the unique maximal ideal. In this situation, \cite[Corollary 8.1(b)]{Ard12} implies that $\overline{P}$ must be the zero ideal, provided we can show that the open subgroup $1 + po_L \subset o_L^\times$ acts \emph{rationally irreducibly} on $o_L$. 

We have to show that every non-trivial $1 + p o_L$-stable subgroup of $o_L$ is open in $o_L$. But such a subgroup contains $(1 + p o_L) a - a = p a o_L$ for some $0 \neq a \in o_L$, and is therefore open in $o_L$. \end{proof}

\begin{corollary}\label{1stNormCor} The restriction of the norm $\|\cdot \|$ on $\Lambda_L(\frX)$ to $o_L\dcroc{o_L}$ coincides with the $\pi$-adic norm on $o_L\dcroc{o_L}$: for any $x \in \pi^n o_L\dcroc{o_L} \backslash \pi^{n+1} o_L\dcroc{o_L}$ we have
\[ \| x \| = |\pi^n|.\]
\end{corollary}

\begin{proof}  Since $\| \pi^n y \| = |\pi^n| \|y\|$ for any $y \in o_L\dcroc{o_L}$, we may assume that $n = 0$. But now since $x \notin \pi o_L\dcroc{o_L}$, Lemma \ref{prime-ideal} tells us that $\|x \| = 1$.
\end{proof} 

\begin{corollary} The $o_L$-module $\Lambda_L(\frX) / o_L\dcroc{o_L}$ is torsionfree.
\end{corollary}
\begin{proof} Suppose that $f \in \Lambda_L(\frX)$ is such that $\pi^n f \in o_L\dcroc{o_L}$ for some $n \geq 0$. Choose $n$ least possible and suppose for a contradiction that $n \geq 1$. Then $\pi^n f \in o_L\dcroc{o_L} \backslash \pi o_L\dcroc{o_L}$, else otherwise we would be able to deduce that $\pi^{n-1} f \in o_L\dcroc{o_L}$. Hence $\| \pi^n f \| = 1$ by Corollary \ref{1stNormCor}, which implies that $|\pi|^{-n} = \|f\| \leq 1$. Hence $n = 0$.
\end{proof}

\begin{corollary}
\label{tensint}
We have $\Lambda_L(\frX) \cap (L \otimes_{o_L} o_L\dcroc{o_L}) =  o_L\dcroc{o_L}$.
\end{corollary}

\section{The Katz isomorphism}
\subsection{The $\psi_q$-operator}\label{PsiSection}
We denote by $\oplus$ the formal group law of $\cG$. Furthermore let $\cG_1$ denote the group of $\pi$-torsion points of $\cG$. Its cardinality is $q$. It coincides with the set of zeros of the Frobenius power series $[\pi](Z) = \varphi(Z)$. 

We fix a $\pi$-adically complete and flat $o_L$-algebra $S$ in what follows and define an injective $S$-algebra endomorphism $\varphi : S\dcroc{Z} \to S\dcroc{Z}$ by setting
\[ \varphi(F)(Z) := F([\pi](Z)) \qmb{for all} F(Z) \in S\dcroc{Z}.\]

\begin{lemma}
\label{Weierstrass} \hsp
\be 
\item For any $F \in S\dcroc{Z}$ there is a unique $F_0 \in S\dcroc{Z}$ and a unique polynomial $F_1 \in S[Z]$ of degree $< q$ such that $F = \varphi(Z)F_0  + F_1$.
\item  $\{F \in S\dcroc{Z} : F(\zeta) = 0 \ \text{for any $\zeta \in \cG_1$}\} = \varphi(Z) S\dcroc{Z}$.
\ee
\end{lemma}
\begin{proof}
(1). This is a form of Weierstrass division. Since $\varphi(Z) \equiv Z^q \bmod \pi o_L\dcroc{Z}$, the proof of \cite[VII.3.8 Prop.\ 5]{B-CA}  goes through by replacing the maximal ideal of $S$ in the argument with the ideal $\pi S$.

(2). Since $\varphi(Z)$ vanishes on $\cG_1$, the inclusion $\supseteq$ is clear. If $F \in S\dcroc{Z}$ vanishes on $\cG_1$ then using (1) we may assume that $F\in S[Z]$ with $\deg F < q$. But then $F = 0$, which gives the other inclusion.
\end{proof}

Using the above lemma the proof of \cite[Lemma 3]{Col} remains valid for $S$ and gives
\begin{equation*}
  \varphi(S\dcroc{Z}) = \{ F \in S\dcroc{Z} : F(Z) = F(\zeta \oplus Z) \ \text{for any $\zeta \in \cG_1$}\}.
\end{equation*}
Since the map $\varphi$ is injective, Lemma \ref{Weierstrass}(2) implies the existence of a unique $S$-linear endomorphism $\psi_{\col}$ of $S\dcroc{Z}$ such that
\begin{equation*}
  \varphi (\psi_{\col} (F)(Z)) = \sum_{\zeta \in \cG_1} F(\zeta \oplus Z)  \qmb{for any $F \in S\dcroc{Z}$}.
\end{equation*}
\begin{definition} Let $S\dcroc{Z}_L := S\dcroc{Z} \otimes_{o_L}L$. The \emph{$\psi_q$-operator} is defined by
\[ \psi_q := \frac{1}{q}\psi_{\col} : S\dcroc{Z}_L \to S\dcroc{Z}_L.\]
\end{definition}

Note that $\psi_{\col}$ (respectively, $\psi_q$) preserves $S'\dcroc{Z}$ (respectively, $S'\dcroc{Z}_L$) for any intermediate $\pi$-adically complete and flat $o_L$-subalgebra $S'$ of $S$. These operators satisfy the following useful Projection Formula.
\begin{lemma}\label{lem:projection-formula}
  For any $F, G \in S\dcroc{Z}$ we have $\psi_q(F \varphi(G)) = \psi_q(F) G$.
\end{lemma}
\begin{proof}
We may instead establish the analogous formula for $\psi_{\col}$. Note that $[\pi](\zeta \oplus Z) = [\pi](\zeta) \oplus [\pi](Z) = [\pi](Z)$ for any $\zeta \in \cG_1$, since $[\pi](\zeta) = \varphi(\zeta) = 0$. Therefore
\begin{align*}
  \varphi (\psi_{\col}(F \varphi(G))) & = \sum_{\zeta \in \cG_1} (F \varphi(G))(\zeta \oplus Z) = \sum_\zeta F(\zeta \oplus Z) G([\pi](\zeta \oplus Z)) \\
     & = \sum_\zeta F(\zeta \oplus Z) G([\pi](Z)) = \sum_\zeta F(\zeta \oplus Z) \varphi(G)  \\
     & = \varphi(\psi_{\col}(F)) \varphi(G) = \varphi(\psi_{\col}(F)G) \ .
\end{align*}
The result follows because $\varphi$ is injective.
\end{proof}

\begin{corollary}\label{cor:FundPsiPhi} We have the fundamental equation $\psi_q \circ \varphi = 1_{S\dcroc{Z}_L}$. \end{corollary}
\begin{proof} Note that $\varphi(\psi_{\col}(1)) = q1$, so $\varphi(\psi_q(1)) = 1$ and hence $\psi_q(1) = 1$. Now set $F = 1$ in Lemma \ref{lem:projection-formula}.
\end{proof}
Next, we remind the reader what the operators $\varphi$ and $\psi_q$ do to the special power series $\Delta_a = \exp(a \Omega \log_{\LT}(Z))$ from $\S \ref{sec:AKtrans}$.
\begin{lemma}\label{lem:PhiPsiDelta} Let $a \in o_L$.
\be \item $\varphi(\Delta_a) = \Delta_{\pi a}$.
\item $\psi_q(\Delta_a) = \delta_{a \in \pi o_L} \Delta_{a/\pi}$.
\ee
\end{lemma}
\begin{proof} 
(1) More generally, whenever $a,b \in o_L$ we have
\[ \Delta_a([b](Z)) = \exp(a \Omega \log_{\LT}([b](Z))) = \exp(ab \Omega \log_{\LT}(Z)) = \Delta_{ab}(Z).\]
Hence $\varphi(\Delta_a) = \Delta_a([\pi](Z)) =\Delta_{\pi a}$ as claimed.

(2) Using the fact that $\log_{\LT}$ is a formal homomorphism from $\mathcal{G}$ to the formal additive group we compute
\begin{align*}
  \varphi(\psi_{\col}(\Delta_a)) & = \sum_{\zeta \in \mathcal{G}_1} \Delta_a (\zeta \oplus Z) = \sum_\zeta \exp(a\Omega \log_{\LT}(\zeta \oplus Z)) \\
                    & = \sum_\zeta \exp\left(a\Omega (\log_{\LT}(\zeta) + \log_{\LT}(Z))\right)  \\
                    & = \big( \sum_{\zeta \in \mathcal{G}_1} \Delta_a(\zeta) \big) \Delta_a \ .
\end{align*}
Under the Schneider-Teitelbaum isomorphism $\kappa$, the group $\mathcal{G}_1$ corresponds to the group of characters $\chi$ of the finite group $o_L/\pi_L o_L$, and, if $\zeta$ corresponds to $\chi$, then $\Delta_a(\zeta) = \ev_{\overline{a}}(\chi) = \chi(\overline{a})$, where $\overline{a} := a + \pi o_L$. Hence
\[\varphi(\psi_{\col}(\Delta_a)) = \left(\sum_{\chi} \chi(\overline{a}) \right) \Delta_a.\]
By column orthogonality of characters of the finite group $o_L / \pi _L$, we have $\sum_{\chi} \chi(\overline{a}) = q\delta_{\overline{a}, \overline{0}} = q\delta_{a \in \pi o_L}$. Hence $q \varphi(\psi_q(\Delta_a)) = q \delta_{a \in \pi o_L} \Delta_a = q \delta_{a \in \pi o_L} \varphi(\Delta_{a/\pi})$, using part (1). Since $\varphi$ is injective, we deduce that $\psi_q(\Delta_a) = \delta_{a \in \pi o_L} \Delta_{a/\pi}$ as required. \end{proof}
Write $\mathfrak{m} := \langle \pi, Z \rangle$ and $A := S\dcroc{Z}$. 
\begin{lemma}\label{lem:PsiCts} The operators $\varphi$ and $\psi_{\col}$ on $A$ are $\frm$-adically continuous.
\end{lemma}
\begin{proof}  Since $\varphi(Z) \in \langle Z \rangle$, we see that $\varphi( \frm^n ) \subseteq \langle \pi, \varphi(Z) \rangle^n \subseteq \frm^n$ for all $n \geq 0$. This implies the $\frm$-adic continuity of $\varphi$. 

Suppose first that $\cG_1$ is contained in $S$. Then the $S$-linear maps $A \to A$ sending $F(Z)$ to $F(Z +_{\cG} \zeta)$ are continuous with respect to $\frm$-adic topology for each $\zeta \in \cG_1$; hence $\varphi \circ \psi_{\col}$ is also $\frm$-adically continuous in this case. Let $L_1 = L(\cG_1)$, a finite extension of $L$ and let $S_1 := o_{L_1} \otimes_{o_L} S$. Since $o_{L_1}$ is a free $o_L$-module of finite rank, $S_1$ is still a $\pi$-adically complete and flat $o_L$-algebra, so letting $A_1 = S_1\dcroc{Z}$, we see that $\varphi \circ \psi_{\col} : A_1 \to A_1$ is $\frm A_1$-adically continuous. It follows that $\varphi \circ \psi_{\col} : A \to A$ is also $\frm$-adically continuous.

Let $n \geq 0$ be given. Since $\varphi(Z) \equiv Z^q \mod \pi A$, we have $\frm^{qn} = \langle \pi, Z \rangle^{qn} \subseteq \langle \pi, Z^q\rangle^n = \langle \pi, \varphi(Z)\rangle^n = A\varphi(\frm^n)$. Therefore $\frm^{qn} \cap \varphi(A) \subseteq A \varphi(\frm^n) \cap \varphi(A) = \varphi(\frm^n)$ where this last equation follows from the fact that $\varphi(A)$ admits a direct complement in $A$ as a $\varphi(A)$-module. However since $\varphi \circ \psi_{\col}$ is continuous, $\varphi\psi_{\col}(\frm^m) \subseteq \frm^{qn}$ for some $m \geq 0$. Hence
\[ \varphi\psi_{\col}(\frm^m) \subseteq \frm^{qn} \cap \varphi(A) \subseteq \varphi(\frm^n).\]
 The $\frm$-adic continuity of $\psi_{\col}$ now follows from the injectivity of $\varphi$. \end{proof}

\begin{lemma}\label{lem:phizeroseq} We have $\varphi^n(a_n) \to 0$ in the $\frm$-adic topology on $A$, for any sequence of elements $(a_n)$ contained in $ZA$.
\end{lemma}
\begin{proof} Since $\varphi(Z) \in \cG$ we see that $\varphi(Z) \in Z \frm$. Assume inductively that $\varphi^n(Z) \in Z \frm^n$; then $\varphi^{n+1}(Z) \in \varphi(Z \frm^n) \subseteq \varphi(Z) \frm^n \subseteq Z \frm^{n+1}$, completing the induction. Write $a_n = Z b_n$ for some $b_n \in A$; then $\varphi^n(a_n) = \varphi^n(Z) \varphi(b_n) \in Z \frm^n \subseteq \frm^{n+1}$ for all $n \geq 0$, so $\varphi^n(a_n) \to 0$.\end{proof}

\begin{proposition}
\label{critolpsi}
If $f \in \Lambda_L(\frX)$ is such that $\psi_q^n(\mu(f)  \Delta_a) \in \Lambda_L(\frX)$ for all $a \in o_L$ and $n \geq 0$, then $f \in o_L\dcroc{o_L}$.
\end{proposition}

\begin{proof}
We will show that $|f(\mathbf{1}_{a+\pi^n o_L})| \leq 1$ for all $a \in o_L$ and $n \geq 0$. By \cite[Lemma 4.6(4)]{ST}, we have 
\[f(\mathbf{1}_{a+\pi^n o_L}) = (f \delta_{-a})(\mathbf{1}_{\pi^n o_L}).\] 
The orthogonality of columns in the character table of the finite group $o_L / \pi^n o_L$ implies that
\[\mathbf{1}_{\pi^n o_L} = \frac{1}{q^n} \sum_{[\pi^n](z)=0} \kappa_z.\] 
Hence by ibid., $(f  \delta_{-a})(\mathbf{1}_{\pi^n o_L}) = \frac{1}{q^n} \sum\limits_{[\pi^n] (z)=0} f (z) \Delta_{-a}(z)$. We now observe that 
\[\frac{1}{q^n} \sum_{[\pi^n](z)=0} f (z)  \Delta_{-a}(z) = \psi_q^n(\mu(f)  \Delta_{-a})(0).\] 
Since $\psi_q^n(\mu(f)  \Delta_{-a}) \in \Lambda_L(\frX)$ by assumption, we have $|f(\mathbf{1}_{a+\pi^n o_L})| \leq 1$ for all $a \in o_L$ and $n \geq 0$, as claimed. Therefore $f \in o_L\dcroc{o_L}$.\end{proof}

\begin{corollary}
\label{ololpsi}
If $R$ is a sub $o_L\dcroc{o_L}$-algebra of $\Lambda_L(\frX)$ such that $\psi(R) \subset R$, then $R=o_L\dcroc{o_L}$.
\end{corollary}
We can now prove Theorem \ref{intropsicrit} from the introduction.
\begin{theorem}
We have $\Lambda_L(\frX) = o_L \dcroc{o_L}$ if and only if $\psi_q(\Lambda_L(\frX)) \subset \Lambda_L(\frX)$.
\end{theorem}
\begin{proof} The forward implication is clear in view of Lemma \ref{lem:PhiPsiDelta}(1). The reverse implication follows from Corollary \ref{ololpsi} applied with $R = \Lambda_L(\frX)$.\end{proof}

\subsection{The covariant bialgebra of \ts{\cG}}
Katz \cite[\S 1]{Ka2} talks about the ``algebra $\Diff(\cG)$ of all $\cG$-invariant $o_L$-linear differential operators from $\cO(\cG)$ into itself". Because we are not aware of any place in the literature which adequately deals with invariant differential operators on formal groups, we will instead use the \emph{covariant bialgebra} of $\cG$ which will turn out to be isomorphic to Katz's $\Diff(\cG)$.

\begin{definition} \label{HypDef}\hsp
\begin{enumerate}
\item Let $Z_1 +_{\cG} Z_2 \in o_L\dcroc{Z_1,Z_2}$ denote the formal group law defining the formal group $\cG$.
\item Let $U(\cG)$ denote the set of all $o_L$-linear maps from $\cO(\cG) = o_L\dcroc{Z}$ to $o_L$ that vanish on some power of the augmentation ideal $Z o_L\dcroc{Z}$. In other words,
\[U(\cG) = \lim\limits_{\longrightarrow} \Hom_{o_L}(\cO(\cG) / Z^n \cO(\cG), o_L).\]
\item For each $f, g \in U(\cG)$, define the product $f \cdot g$ by the formula
\[ (f \cdot g)(F(Z)) = (f \widehat{\otimes} g)(F(Z_1 +_{\cG} Z_2)) \quad\mbox{for all}\quad F(Z) \in o_L\dcroc{Z}.\]
\item With this product, $U(\cG)$ is the \emph{covariant bialgebra} of $\cG$, defined at \cite[36.1.8]{Haz}.
\item For each $m \geq 0$, let $u_m \in U(\cG)$ be the unique $o_L$-linear map that satisfies
\[u_m(Z^n) = \delta_{mn} \quad\mbox{ for all } \quad n \geq 0.\]
\item Let $\langle -,- \rangle : U(\cG) \times \cO(\cG) \to o_L$ be the evaluation pairing:
\[ \langle f, F \rangle := f(F).\]
\end{enumerate}
\end{definition}

This covariant bialgebra is also known as the \emph{hyperalgebra} or the \emph{distribution algebra} of $\cG$.  We will now explain the link with Katz's work, using his notation.

\begin{lemma} \label{StrConsts}\hsp
\begin{enumerate} 
\item $\{u_n : n \geq 0\}$ is an $o_L$-module basis for $U(\cG)$.
\item Let $i \geq 0$ and write $(Z_1 +_{\cG} Z_2)^i = \sum\limits_{\stackrel{n,m \geq 0}{n + m \geq i}}^\infty \lambda(n,m;i) Z_1^n Z_2^m$ for some $\lambda(n,m;i) \in o_L$. Then for all $n,m \geq 0$ we have
\[u_n \cdot u_m = \sum\limits_{k=0}^{n+m} \lambda(n,m;k) u_k.\]
\item Let $s$ be a variable. The map $L[s] \to U(\cG) \otimes_{o_L} L$ which sends $s$ to $u_1 \otimes 1$ is an isomorphism of positively filtered $L$-algebras.
\end{enumerate}
\end{lemma}
\begin{proof} (1) This is clear because $Z^no_L\dcroc{Z} = o_LZ^n \oplus Z^{n+1}o_L\dcroc{Z}$ for any $n \geq 0$. 

(2) We compute that for every $n,m,i \geq 0$ we have
\[ (u_n\cdot u_m)(Z^i) = (u_n \widehat{\otimes} u_m)( (Z_1 +_{\cG} Z_2)^i ) = (u_n \widehat{\otimes} u_m)\left(  \sum\limits_{\stackrel{a,b \geq 0}{a + b \geq i}}^\infty \lambda(a,b;i) Z_1^a Z_2^b \right) = \lambda(n,m;i).\]
Because $\sum\limits_{k=0}^{n+m} \lambda(n,m;k) u_k$ also sends $Z^i$ to $\lambda(n,m;i)$, it must be equal to $u_n \cdot u_m$.

(3) From (2) we see that the $o_L$-submodule $U(\cG)_n$ of $U(\cG)$ generated by $\{u_i : 0 \leq i \leq n\}$ defines an algebra filtration on $U(\cG)$:
\[ U(\cG)_n \cdot U(\cG)_m \subseteq U(\cG)_{n+m} \quad\mbox{for all}\quad n,m\geq 0.\]
The associated graded ring is the free $o_L$-module with basis $\{\gr u_n : n \geq 0\}$. Since $Z_1 +_{\cG} Z_2 \equiv Z_1 + Z_2 \mod (Z_1,Z_2)^2$, we see that $\lambda(n,m;n+m) = \binom{n+m}{n}$ for any $n,m \geq 0$. Hence from (2) we see that the multiplication in $\gr U(\cG)$ is given by 
\[ (\gr u_n) \cdot (\gr u_m) = \binom{n+m}{n} \gr u_{n+m}.\]
The same formulas hold in $\gr (U(\cG) \otimes_{o_L} L)$. Induction on $n$ shows that $(\gr u_1)^n = n! \gr u_n$ for all $n \geq 0$. Since $L$ has characteristic zero, we see that $\gr (U(\cG) \otimes_{o_L} L)$ is generated by $\gr u_1$ as an $L$-algebra. The result follows.
\end{proof}

We will henceforth identify $U(\cG) \otimes_{o_L} L$ with the polynomial ring $L[s]$.  Recall the polynomials $P_n(Y) \in L[Y]$ from \cite[Definition 4.1]{ST}, which are defined by the following formal expansion:
\begin{equation*}
  \exp(Y \log_{\LT}(Z)) = \sum_{m=0}^\infty P_m(Y) Z^m.
\end{equation*}

\begin{lemma}\label{PnU} For every $n \geq 0$, we have $u_n = P_n(u_1)$ inside  $U(\cG) \otimes_{o_L}L$. \end{lemma}
\begin{proof} The structure constants of Katz's algebra $\Diff(\cG)$ are the same as the ones in $U(\cG)$ by \cite[(1.2)]{Ka2} and Lemma \ref{StrConsts}(2). So the $o_L$-linear map that sends $D(n) \in \Diff(\cG)$ to $u_n\in U(\cG)$ is an $o_L$-algebra isomorphism. Comparing \cite[Corollary 1.8]{Ka2} with \cite[Definition 4.1]{ST} shows that $D(n) = P_n(D(1))$ in $\Diff(\cG) \otimes_{o_L} L$ for all $n \geq 0$. The result follows by applying the algebra isomorphism $\Diff(\cG) \to U(\cG)$ established above. \end{proof}
Of course in the context of affine group schemes, this isomorphism between the algebra of left-invariant differential operators on the group scheme and the distribution algebra of the group scheme is the well known `Invariance Theorem', \cite[Chapter II, \S 4, Theorem 6.6]{DeGa}.

Next, we consider the action of the monoid $o_L$ on the formal group $\cG$. The covariant bialgebra construction is functorial in $\cG$: if $\varphi : \cG \to \cH$ is a morphism of formal groups, then $U(\varphi) : U(\cG) \to U(\cH)$ is the morphism of $o_L$-bialgebras which is the transpose to the $o_L$-algebra homomorphism $\varphi^\ast : \cO(\cH) \to \cO(\cG)$ induced by $\varphi$. Using the evaluation pairing, we have the following formula which defines this action:
\begin{equation}\label{TransOact} \langle U(\varphi)(f), F \rangle = \langle f, \varphi^\ast(F)\rangle \quad\mbox{for all}\quad f \in U(\cG), F \in \cO(\cG).\end{equation}
\begin{definition}\label{MonAct1} Let $a \in o_L$.
\begin{enumerate}
\item Let $[a] : \cG \to \cG$ be the action of $a$ on $\cG$.
\item Write $a \cdot f := U([a])(f)$ for all $f \in U(\cG)$.
\end{enumerate}
\end{definition}
The $o_L$-algebra endomorphism $U([a])$ of $U(\cG)$ extends to an $L$-algebra endomorphism $U([a]) \otimes 1$ of $U(\cG) \otimes_{o_L} L = L[s]$. What does this action do to the generator $s$ of $L[s]$? 

\begin{lemma}\label{ActionOnTgt} We have $a \cdot s = as$ for all $a \in o_L$.
\end{lemma}
\begin{proof} We know that $[a](Z) \equiv aZ \mod Z^2 o_L\dcroc{Z}$. Hence
\[ \langle U([a])(u_1), Z^n \rangle = \langle u_1, [a](Z)^n \rangle = a \delta_{n,1} = \langle a u_1, Z^n \rangle \quad\mbox{for all}\quad n \geq 0\]
using Definition \ref{HypDef}(5). Hence $a \cdot u_1 = au_1$ and so $a \cdot s = as$. \end{proof}

\begin{corollary}\label{MatrixCoeffs} For each $j \geq i \geq 0$ and $a \in o_L$ there exists $\sigma_{ij}(a) \in o_L$ such that
\[ a \cdot u_j = P_j(a s) = \sum\limits_{i=0}^j \sigma_{ij}(a) P_i(s) = \sum\limits_{i=0}^j \sigma_{ij}(a) u_i.\]
\end{corollary}
\begin{proof} It follows from Lemma \ref{ActionOnTgt} that the $L$-algebra endomorphisms of $L[s]$ given by $s \mapsto as$ preserve the $o_L$-subalgebra $U(\cG) \subset L[s]$. Hence $a \cdot u_j = P_j(as)$ lies in $U(\cG)$ for all $a \in o_L$ and all $j \geq 0$. But $U(\cG)$ has $\{u_i : i \geq 0\}$ as an $o_L$-module basis by Lemma \ref{StrConsts}(a), so $P_j(as)$ must be an $o_L$-linear combination of these $u_i$'s. On the other hand, $P_j(s)$ is a polynomial of degree $j$ in $s$, therefore so is $P_j(as)$; because $\deg P_i = i$ for each $i$ it follows that $P_j(as)$ is an $L$-linear combination of $P_0(s),\cdots,P_j(s)$ only. 
\end{proof}

We now introduce a coefficient ring $S$, which we assume to be a $\pi$-adically complete $o_L$-algebra. For every $S$-module $M$, let $M^\ast := \Hom_{S}(M,S)$ be the $S$-module of $S$-linear functionals on $M$. We will need to work with a larger class of $S$-linear functionals on $S\dcroc{Z}$ than those arising from $U(\cG)$, namely the \emph{continuous} ones.

\begin{definition}\label{def:CtsLinFun} We say that $\lambda \in S\dcroc{Z}^\ast$ is \emph{continuous} if it is continuous with respect to the $\langle \pi,Z\rangle$-adic topology on $S\dcroc{Z}$, and the $\pi$-adic topology on $S$. Let $S\dcroc{Z}^\ast_{\cts}$ denote the set of these continuous $S$-linear functionals on $S\dcroc{Z}$.
\end{definition}
Explicitly $\lambda \in S\dcroc{Z}^\ast$ is continuous if and only if for all $n \geq 0$ there exists $m \geq 0$ such that $\lambda( \langle \pi, Z \rangle^m ) \subseteq \pi^n S$.  

Consider now the base change $U(\cG_S) := U(\cG) \otimes_{o_L} S$, and its $\pi$-adic completion
\[\h{U(\cG_S)} = \varprojlim U(\cG) \otimes_{o_L} (S  / \pi^n S).\] 
Since $\{u_m : m \geq 0\}$ is an $o_L$-module basis for $U(\cG)$ by Lemma \ref{StrConsts}(1), we see that $\h{U(\cG_S)}$ has the following description:
\begin{equation} \label{eq:HUGS} \h{U(\cG_S)} = \left\{ \sum\limits_{m=0}^\infty a_m u_m :  \quad a_m \in S, \lim\limits_{m\to\infty} a_m = 0\right\}.\end{equation}
Here we equip $S$ with the $\pi$-adic topology.
\begin{lemma}\label{lem:BigPairing} \hsp
\be \item The pairing $\langle-,-\rangle: U(\cG) \times o_L\dcroc{Z} \to o_L$ extends to an $S$-bilinear pairing
\[\langle-,-\rangle: \h{U(\cG_S)} \times S\dcroc{Z} \to S.\]
\item For each $u \in \h{U(\cG_S)}$, the $S$-linear map $\langle u,- \rangle : S\dcroc{Z} \to S$ is continuous.
\item The map $\h{U(\cG_S)} \to S\dcroc{Z}^\ast_{\cts}$, $u \mapsto \langle u,-\rangle$, is an $S$-linear bijection.
\item The map $S\dcroc{Z} \to \h{U(\cG_S)}^\ast$, $F \mapsto \langle -, F \rangle$, is an $S$-linear bijection.
\ee\end{lemma}
\begin{proof} (1) Let $u = \sum\limits_{m=0}^\infty a_m u_m \in \h{U(\cG_S)}, F = \sum\limits_{n=0}^\infty F_nZ^n \in S\dcroc{Z}$ and define $\langle u, F\rangle = \sum\limits_{m=0}^\infty a_m F_m$. This series converges in $S$ because $a_m \to 0$ as $m \to \infty$ and because $S$ is assumed to be $\pi$-adically complete.

(2) Let $n \geq 0$ and write $u = \sum\limits_{m=0}^\infty a_m u_m$ with $a_m \to 0$. Then for some $r \geq 0$, $a_m \in \pi^n S$ for all $m \geq r$. Hence $\langle u,-\rangle$ sends the ideal $\langle \pi^n, Z^r\rangle$ of $S\dcroc{Z}$ into $\pi^n S$. Since $\langle \pi, Z\rangle^{n+r} \subseteq \langle \pi^n, Z^r \rangle$, we conclude that $\langle u,-\rangle$ is $\langle \pi, Z\rangle$-adically continuous.

(3) The injectivity of $u \mapsto \langle u,-\rangle$ follows by evaluating on each $Z^n$. Now let $\lambda \in S\dcroc{Z}^\ast_{\cts}$ and define $a_m := \lambda(Z^m) \in S$ for each $m \geq 0$. Since $\lambda$ is $\langle \pi,Z \rangle$-adically continuous, for each $n \geq 0$ we can find some $r \geq 0$ such that $\lambda( \langle \pi , Z \rangle^r ) \subseteq \pi^n S$. Then $a_m \in \pi^n S$ for all $m \geq r$ which implies that $a_m \to 0$ as $m \to \infty$. Hence $u := \sum\limits_{m=0}^\infty a_m u_m$ is an element of $\h{U(\cG_S)}$ and $\langle u, - \rangle - \lambda$ vanishes on $S[Z]$ by construction. Since this difference is continuous and since $S[Z]$ is dense in $S\dcroc{Z}$ with respect to the $\langle \pi, Z\rangle$-adic topology, we conclude that $\lambda = \langle u,- \rangle$.

(4) Again, the injectivity of $F \mapsto \langle -, F \rangle$ follows from $\langle u_m, F \rangle = F_m$. Given an $S$-linear map $\lambda : U(\cG_S) \to S$, let $F := \sum\limits_{n=0}^\infty \lambda(u_n) Z^n$. Then $\langle u_m, F\rangle = \lambda(u_m)$ for all $m \geq 0$. Since the $u_m$ span $U(\cG_S)$ as an $S$-module, $\lambda = \langle -, F \rangle$.
\end{proof}
As an immediate consequence of Lemma \ref{lem:BigPairing}, we have the following
\begin{corollary}\hsp \label{cor:adjoints}
\be \item For every continuous $S$-linear $\alpha : S\dcroc{Z} \to S\dcroc{Z}$ there exists a unique $S$-linear map $\alpha^\ast : \h{U(\cG_S)} \to \h{U(\cG_S)} $ such that 
\[ \langle \alpha^\ast u, F \rangle = \langle u, \alpha F \rangle \qmb{for all} u \in \h{U(\cG_S)} , F \in S\dcroc{Z}.\]
\item For every $S$-linear $\beta: \h{U(\cG_S)} \to \h{U(\cG_S)}$ there exists a unique $S$-linear map \newline $\beta^\ast : S\dcroc{Z} \to S\dcroc{Z}$ such that
\[ \langle u, \beta F \rangle = \langle \beta^\ast u, F \rangle  \qmb{for all} u \in \h{U(\cG_S)} , F \in S\dcroc{Z}.\]
\ee\end{corollary}

We also extend this $S$-linear pairing to an $S_L := S \otimes_{o_L} L$-linear pairing
\[\langle-,-\rangle: \h{U(\cG_S)}_L \times S\dcroc{Z}_L \to S_L\]
which we will use without further mention.

\begin{lemma}\label{lem:extendToS} The restriction map $\h{U(\cG_S)}^\ast \to \Hom_{o_L}(\h{U}, S)$ is an $S$-linear isomorphism.
\end{lemma}
\begin{proof} Let $\lambda : \h{U(\cG_S)} \to S$ be an $S$-linear map whose restriction to $\h{U}$ is zero. Then in particular $\lambda(u_m) = 0$ for all $m \geq 0$, so $\lambda$ vanishes on all finite sums of the form $\sum\limits_{m=0}^n a_m u_m \in \h{U(\cG_S)}$ with $a_m \in S$. These sums are $\pi$-adically dense in $\h{U(\cG_S)}$ in view of $(\ref{eq:HUGS})$, so for any $x \in \h{U(\cG_S)}$, $\lambda(x) \in \bigcap\limits_{n=0}^\infty \pi^n S$. Since we're assuming that $S$ is $\pi$-adically complete, this intersection is zero, so $\lambda = 0$ and the restriction map in question is injective.

Suppose now $\lambda : \h{U} \to S$ is an $o_L$-linear map. Using the description of $\h{U(\cG_S)}$ given in $(\ref{eq:HUGS})$, we extend it to an $S$-linear map $\tilde{\lambda} : \h{U(\cG_S)} \to S$ by setting for every zero-sequence $(a_m)$ in $S$
\[ \tilde{\lambda}\left(\sum\limits_{m=0}^\infty a_m u_m \right) := \sum\limits_{m=0}^\infty a_m \lambda(u_m).\]
Since $\lim\limits_{m \to \infty} a_m= 0$ in $S$, the series on the right hand side converges in $S$ because $S$ is assumed to be $\pi$-adically complete. So, $\tilde{\lambda}$ is a well-defined $S$-linear map extending $\lambda$. \end{proof}

\subsection{$\Gal$-continuous functions}

Let $\cC^0(o_L, \Cp)$ be the $\Cp$-Banach space of all continuous $\Cp$-valued functions on $o_L$, equipped with the supremum norm. The unit ball of this $\Cp$-Banach space is the $o_{\Cp}$-submodule $\cC^0(o_L, o_{\Cp})$ of continuous $o_{\Cp}$-valued functions. 

\begin{definition}\label{def:GalCts} A function $f \in \cC^0(o_L, \Cp)$ is said to be \emph{$\Gal$-continuous} if 
\[ \sigma(f(a)) = f(a \tau(\sigma)) \qmb{for all} a \in o_L, \sigma\in G_L.\]
We write $C := \cC^0_{\Gal}(o_L, \Cp)$ for the set of all $\Gal$-continuous $\Cp$-valued functions.
\end{definition}
Evidently $\cC := \cC^0_{\Gal}(o_L, o_{\Cp}) = C \cap \cC^0(o_L, o_{\Cp})$ forms an $o_L$-lattice in $C$. 

\begin{lemma}\label{lem:valuesInOinf} Let $f \in C$. Then $\im f \subseteq L_\infty$, and $\im f \subseteq o_\infty$ if $f \in \cC$.
\end{lemma}
\begin{proof} By Definition \ref{def:GalCts}, we have $\im f \subseteq \Cp^{\ker \tau}$ for all $f \in C$, and $\im f \subseteq o_{\Cp}^{\ker \tau}$ for all $f \in \cC$. But $\Cp^{\ker \tau} = L_\infty$ and $o_{\Cp}^{\ker \tau} = o_\infty$ by Lemma \ref{lem:LOmegaDense}.
\end{proof}

\begin{lemma}\label{lem:KatzGalCts} For each $u \in \h{U}$, the function $a \mapsto \cK(u)(a) := \langle u, \Delta_a\rangle$ on $o_L$ is $\Gal$-continuous.
\end{lemma}
\begin{proof} By definition, $\cK(u)$ is the composition of $\mu_{|o_L} : o_L \to o_\infty\dcroc{Z}^\times$ with the restriction of the linear functional $\langle u, - \rangle : o_\infty\dcroc{Z} \to o_\infty$ to $o_\infty\dcroc{Z}^\times$. This linear functional is continuous by Lemma \ref{lem:BigPairing}(3),  so to establish the continuity of $\cK(u)$ it remains to show that $\mu_{|o_L}$ is continuous. Since $\mu_{|o_L}$ is a group homomorphism, it is enough to show that it is continuous at the identity element $0$ of $o_L$. Let $n > 0$ and consider the basic open neighbourhood $1 + \langle \pi, Z \rangle^n$ of $1 \in o_\infty\dcroc{Z}^\times$. Since $\varphi^n(Z) \to 0$ as $n \to \infty$ in $o_\infty\dcroc{Z}$ by Lemma \ref{lem:phizeroseq}, we can find $m \geq 0$ such that $\varphi^m(Z) \in \langle \pi, Z \rangle^n$. Hence for any $a \in o_L$,  using Lemma \ref{lem:PhiPsiDelta} we calculate 
\[\Delta_{\pi^m a} - \Delta_0 = \varphi^m(\Delta_a - 1) \in \varphi^m(Z o_\infty\dcroc{Z}) \subseteq \varphi^m(Z) o_\infty\dcroc{Z} \subseteq \langle \pi, Z \rangle^n.\]
Hence $\mu_{|o_L}$ is continuous as required. 

Now let $\sigma \in G_L$; since $\Delta_a \in o_\infty\dcroc{Z}$ is invariant for the $\ast$-action of $G_L$ on $o_\infty\dcroc{Z}$, we know that $\sigma(\Delta_a) = \Delta_a([\tau(\sigma)](Z)) = \Delta_{a \tau(\sigma)}$ for any $a \in o_L$. Since $u \in \h{U}$, we have for any $a \in o_L$
\[ \sigma(\cK(u)(a)) = \sigma(\langle u, \Delta_a\rangle) = \langle u, \sigma(\Delta_a) \rangle = \langle u, \Delta_{a \tau(\sigma)}\rangle = \cK(u)(a \tau(\sigma)).\]
Hence $\cK(u)$ is indeed $\Gal$-continuous. \end{proof}
\begin{definition}\label{def:KatzPsiPhi}  \hsp
\be \item Define \emph{the Katz map} $\cK : \widehat{U} \to \cC$ as follows:
\[\cK(u)(a) = \langle u, \Delta_a \rangle \quad \mbox{for any} \quad u \in \widehat{U}, a \in o_L.\]
\item Define $\cK_1 : \h{U} \to o_\infty$ by $\cK_1 = \ev_1 \circ \cK$.
\item Define $\psi_C : C \to C$ by the rule
\[ \psi_C(f)(a) = \delta_{a \in \pi o_L} f(a/\pi) \qmb{for all} a\in o_L.\]
The operator $\psi_{\cC} : \cC \to \cC$ is by definition the restriction of $\psi_C$ to $\cC$.
\item Define $\varphi_C : C \to C$ by the rule
\[ \varphi_C(f)(a) = f(\pi a) \qmb{for all} a \in o_L.\]
The operator $\varphi_{\cC} : \cC \to \cC$ is by definition the restriction of $\varphi_C$ to $\cC$.
\ee\end{definition}

Now we recall the coefficient ring $S$ that was introduced before Definition \ref{def:CtsLinFun}. Applying the $S$-linear duality functor 
\[(-)^\ast := \Hom_{o_L}(-,S)\] 
to the Katz map $\cK : \h{U} \to \cC$ gives us the \emph{dual Katz map} 
\[\cK^\ast : \cC^\ast \to \h{U}^\ast\] 
defined on the space of \emph{$S$-valued Galois measures} $\cC^\ast = \Hom_{o_L}(\cC,S)$. We identify $\h{U}^\ast = \Hom_{o_L}(\h{U}, S)$ with $S\dcroc{Z}$ using Lemma \ref{lem:extendToS} and Lemma \ref{lem:BigPairing}(4); then $\cK^\ast : \cC^\ast \to S\dcroc{Z}$ is given explicitly by 
\begin{equation}\label{eq:ExplicitKatz} \langle u_m, \cK^\ast(\lambda)\rangle = \lambda(P_m(-\Omega)))  \qmb{for all} \lambda \in \cC^\ast, m \geq 0.\end{equation}
After Lemma \ref{lem:PsiCts} and Corollary \ref{cor:adjoints} applied with $S = o_L$, we have at our disposal the dual $o_L$-linear endomorphisms $\psi_{\col}^\ast$ and $\varphi^\ast$ of $\h{U}$. 
\begin{lemma}\label{lem:KatzPhiPsi}
We have $\cK \varphi^\ast = \varphi_{\cC} \cK$ and $\cK \psi_{\col}^\ast = q \psi_{\cC} \cK$.
\end{lemma}
\begin{proof} Let $u \in \h{U}_L$ and $a\in o_L$. Then using Lemma \ref{lem:PhiPsiDelta}, we have
\[\begin{array}{lllllll} \cK(\psi_{\col}^\ast(u))(a) &=& \langle \psi_{\col}^\ast(u), \Delta_a\rangle &=& \langle u, \psi_{\col}(\Delta_a)\rangle  &=& \langle u, q \psi_q(\Delta_a)\rangle \\
&=& q\langle u, \delta_{a \in \pi o_L} \Delta_{a/\pi}\rangle  &=& q\delta_{a \in \pi o_L} \cK(u)(a/\pi)  &=& q\psi_{\cC}(\cK(u))(a)  \end{array}\]
which gives the second equation. The first equation is proved in a similar manner.
\end{proof}
\begin{corollary}\label{cor:DualKatzPhiPsi}
We have $\cK^\ast \varphi_{\cC}^\ast = \varphi \cK^\ast$ and $\cK^\ast \psi_{\cC}^\ast = \psi_q \cK^\ast$.
\end{corollary}
\begin{proof}  We apply the $S$-linear duality functor $(-)^\ast = \Hom_{o_L}(-,S)$ to the equations from Lemma \ref{lem:KatzPhiPsi}. Using Lemma \ref{lem:BigPairing}, we see that
\[ \cK^\ast \varphi_{\cC}^\ast = (\varphi_{\cC} \cK)^\ast = (\cK \varphi^\ast)^\ast = \varphi^{\ast\ast} \cK^\ast = \varphi \cK^\ast,\]
and similarly, 
\[q\cK^\ast \psi_{\cC}^\ast = (q\psi_{\cC} \cK)^\ast  = (\cK \psi_{\col}^\ast)^\ast = \psi_{\col} \cK^\ast = q \psi_q \cK^\ast.\]
Now divide both sides by $q$. \end{proof}

\begin{lemma}\label{lem:psiK1} We have $\psi_q \circ \cK_1^\ast = 0$.
\end{lemma}
\begin{proof} Corollary \ref{cor:DualKatzPhiPsi} gives $\psi_q \cK_1^\ast = \psi_q \cK^\ast \ev_1^\ast = \cK^\ast \psi_{\cC}^\ast \ev_1^\ast = (\ev_1\psi_{\cC}\cK)^\ast$. But $\ev_1\psi_{\cC}(f) = \psi_{\cC}(f)(1) = 0$ for any $f \in \cC$ by Definition \ref{def:KatzPsiPhi}(3), because $\delta_{1 \in \pi o_L} = 0$.
\end{proof}

\begin{proposition}\label{prop:KerK1ImPsi} Suppose $\cK$ is injective and $\tau$ is surjective. Then $q \ker \cK_1 \subseteq \psi_{\col}^\ast( \h{U} )$.
\end{proposition}

\begin{proof} In this proof we may assume $S = o_L$. Suppose that $\ev_1 \circ \cK(u) = 0$ for some $u \in \h{U}$. Then $\cK(u)$ is zero on $o_L^\times$ because $\tau$ is surjective and because $\cK(u)$ is $\Gal$-continuous by Lemma \ref{lem:KatzGalCts}. Hence $\cK(u) = \psi_{\cC} \varphi_{\cC} \cK(u)$. But $q \psi_{\cC} \varphi_{\cC} \cK(u) = q \psi_{\cC} \cK \varphi^\ast(u) = \cK \psi_{\col}^\ast\varphi^\ast(u)$ by Lemma \ref{lem:KatzPhiPsi}, so $\cK(q u - \psi_{\col}^\ast \varphi^\ast(u)) = 0$. Since $\cK$ is injective by assumption, $qu = \psi_{\col}^\ast (\varphi^\ast(u)) \in \psi_{\col}^\ast(\h{U})$.
\end{proof}

\begin{proposition}\label{prop:PartialConditionalKatz} Suppose that $\tau : G_L \to o_L^\times$ and $\cK_1 : \h{U} \to o_\infty$ are surjective, and that $\cK : \h{U} \to \cC$ is injective. Then 
\[\cK_1^\ast : o_\infty^\ast \to S\dcroc{Z}^{\psi_q=0}\] 
is an $S$-linear bijection.
\end{proposition}
\begin{proof} The image of $\cK^\ast : o_\infty^\ast \to S\dcroc{Z}$ is contained in $S\dcroc{Z}^{\psi_q=0}$ by Lemma \ref{lem:psiK1}. If $\cK_1^\ast(\ell) = 0$ for some $\ell \in o_\infty^\ast$, then $\ell \circ \cK_1 = 0$ so $\ell(\cK_1(\h{U})) = 0$. But $\cK_1(\h{U}) = o_\infty$ by assumption, so $\ell=0$. Hence $\cK_1^\ast$ is injective and it remains to prove it is also surjective.

Take some $F \in S\dcroc{Z}^{\psi_q = 0}$ and let $\ell := \langle -, F \rangle \in \h{U(\cG_S)}^\ast \cong \h{U}^\ast$ be the $S$-valued  $o_L$-linear functional on $\h{U}$ given by Lemma \ref{lem:extendToS} and Lemma \ref{lem:BigPairing}(4). Then since $\psi_{\col}(F) = q\psi_q(F) = 0$, 
\[0 = \langle u, \psi_{\col}(F)\rangle = \langle\psi_{\col}^\ast(u),F\rangle = \ell( \psi_{\col}^\ast(u)) \quad\mbox{for all} \quad  u \in \h{U}.\] 
So, $\ell$ vanishes on $\psi_{\col}^\ast(\h{U})$ and hence also on $q \ker \cK_1$ by Proposition \ref{prop:KerK1ImPsi}. Since $o_L$ has no $q$-torsion, we see that $\ell$ is zero on $\ker \cK_1$. Hence $\ell$ descends to an $S$-valued $o_L$-linear functional on $\h{U} / \ker \cK_1$. But this quotient is isomorphic to $o_\infty$ by assumption. So, we get a well-defined $o_L$-linear form $\overline{\ell} : o_\infty \to S$ such that $\overline{\ell}(\cK_1(u)) = \ell(u)$ for all $u \in \h{U}$. Then 
\[ \langle u, \cK_1^\ast(\overline{\ell})\rangle = \overline{\ell}(\cK_1(u)) = \ell(u) = \langle u, F\rangle \qmb{for all} u \in \h{U}\]
which implies that $F = \cK_1^\ast(\overline{\ell})$ by Lemma \ref{lem:BigPairing}(4). Hence $\cK_1^\ast$ is surjective.
\end{proof}

We make the following tentative

\begin{conjecture}\label{conj:NiceKatz}$\cK_1 : \h{U} \to o_\infty$ is surjective and $\cK : \h{U} \to \cC$ is injective whenever $\tau$ is surjective. \end{conjecture}

\subsection{The largest $\psi_q$-stable $o_L$-submodule of $o_L\dcroc{Z}$} For brevity, we will write 
\[A := S\dcroc{Z}\] 
in this subsection. The $\psi_q$-operator is only defined on $A_L$ and it does \emph{not} preserve $A$, in general.

\begin{definition} Let $A^{\psiqint}$ be the largest $S$-submodule of $A$ stable under $\psi_q$. 
\end{definition}
\begin{remark}\label{rem:explicitpsiint} We have $A^{\psiqint} = \{F \in A : \psi_q^n(F) \in A$ for all $n \geq 0\}$. \end{remark}

\begin{lemma}\label{lem:ImDualKatz} The image of $\cK^\ast : \cC^\ast \to A$ is contained in $A^{\psiqint}$.\end{lemma}
\begin{proof} Let $\lambda \in \cC^\ast$. By Corollary \ref{cor:DualKatzPhiPsi}, $\psi_q^n( \cK^\ast(\lambda) )= \cK^\ast((\psi_{\cC}^\ast)^n(\lambda))$ lies in $A$ for all $n \geq 0$. Now use Remark \ref{rem:explicitpsiint}.
\end{proof}

Clearly, $A^{\psi_q=0}$ is contained in $A^{\psiqint}$; moreover this last is $\varphi$-stable in view of Remark \ref{rem:explicitpsiint} and the fact that $\psi_q \circ \varphi = 1_A$ by Corollary \ref{cor:FundPsiPhi}. Therefore 
\[ S + \sum\limits_{n =0}^\infty \varphi^n\left(A^{\psi_q=0}\right) \subseteq A^{\psiqint}.\]
Our next result makes this relation more precise; first we need some more notation.
\begin{definition} We have the following \emph{truncation operators}:
\be \item $s : \cC \to \cC$, given by $s(f) = f - f(0)1$, and
\item $t : A \to A$, given by $t(a) = a - a(0)1$.
\ee
\end{definition}
It will be helpful to observe that $t \varphi = \varphi t$ as $S$-linear endomorphisms of $A$.
\begin{proposition}\label{prop:phipsiphi} There is a well-defined $o_L$-linear bijection
\[1 \oplus \sum\limits_{n=0}^\infty \varphi^n t : o_L \oplus \prod\limits_{n=0}^\infty A^{\psi_q=0} \stackrel{\cong}{\longrightarrow} A^{\psiqint}.\]
\end{proposition}
\begin{proof} Given any $(a_n)_n \in \prod_{n=0}^\infty A^{\psi_q=0}$, Lemma \ref{lem:phizeroseq} implies that $\varphi^n(t(a_n)) \to 0$ as $n \to \infty$, because $t(a_n) \in ZA$ for all $n \geq 0$.  Hence
\[ (z, (a_n)_n) \mapsto z + \sum\limits_{n=0}^\infty \varphi^n( t(a_n))\]
is a well-defined $S$-linear map $\gamma:S \oplus \prod\limits_{n=0}^\infty A^{\psi_q=0} \to A$. Now $A^{\psiqint}$ is a $t$-stable $S$-submodule of $A$ since $\psi_q(1) = 1$.  Because $a_n \in A^{\psi_q=0}$, this implies that $\varphi^n(t(a_n)) = t \varphi^n(a_n) \in t(A^{\psiqint}) \subseteq A^{\psiqint}$ for any $n \geq 0$. Since $\psi_{\col} : A \to A$ is continuous by Lemma \ref{lem:PsiCts} and since $A^{\psiqint} = \{a \in A : \psi_{\col}^n(a) \in q^n A$  for all $n \geq 0\}$ by Remark \ref{rem:explicitpsiint}, we see that $A^{\psiqint}$ is a closed $S$-submodule of $A$ with respect to the $\langle \pi, Z \rangle$-adic topology on $A = S\dcroc{Z}$. Hence the image of $\gamma$ is contained in $A^{\psiqint}$, and it remains to show that $\gamma$ is bijective.

Suppose that $\gamma(z, (a_n)_n) = 0$ so that $z = -\sum\limits_{n=0}^\infty \varphi^n(t(a_n))$. Since $Z A$ is closed in $A$, this infinite sum lies in $ZA$. Since $S \cap Z A = 0$, we conclude that $z = 0 $. Hence $a_0 = -\sum\limits_{n=1}^\infty \varphi^n(t(a_n)) \in \varphi(A)$. But $a_0 \in A^{\psi_q=0}$ by definition, and
\[ A^{\psi_q=0} \cap \varphi(A) = 0\]
because $\psi_q \circ \varphi = 1_A$ by Corollary \ref{cor:FundPsiPhi}. Hence $a_0 = 0$. Proceeding inductively on $n$, we quickly deduce that $a_n = 0$ for all $n \geq 0$ in a similar manner. Hence $\gamma$ is injective.

Now let $a \in A^{\psiqint}$; then by definition, $\psi_q^n(a) \in A$ for all $n \geq 0$, so we can define 
\[a_n := \psi_q^n(a) - \varphi \psi_q^{n+1}(a) \in A.\]
Since $\psi_q \circ \varphi = 1_A$ by Corollary \ref{cor:FundPsiPhi}, we see that $a_n \in A^{\psi_q = 0}$ for all $n \geq 0$. Since $t \varphi = \varphi t$,
\[\sum\limits_{n=0}^m \varphi^n(t(a_n)) = t\left(\sum\limits_{n=0}^m \varphi^n(\psi_q^n(a) - \varphi\psi_q^{n+1}(a))\right) = t(a - \varphi^{m+1}\psi_q^{m+1}(a))\]
for any $m \geq 0$. Since $t \varphi^{m+1} \psi_q^{m+1}(a) = \varphi^{m+1}(t \psi_q^{m+1}(a)) \to 0$ as $m \to \infty$ by Lemma \ref{lem:phizeroseq}, 
\[\gamma( a(0), (a_n)_n) = a(0) + t(a) - \lim\limits_{m \to 0} \varphi^{m+1} (t \psi_q^{m+1}(a)) = a.\]
Hence $\gamma$ is surjective. \end{proof}
\begin{lemma}\label{lem:phiK1K} For each $n \geq 0$, there is a commutative diagram
\[ \xymatrix{ o_\infty^\ast \ar[rr]^{\ev_{\pi^n}^\ast}\ar[d]_{\cK_1^\ast} && \cC^\ast \ar[d]^{\cK^\ast} \ar[rr]^{s^\ast} && \cC^\ast\ar[d]_{\cK^\ast} \\ 
A^{\psi_q=0} \ar[rr]_{\varphi^n} && A^{\psiqint}\ar[rr]_t && A^{\psiqint}.}\]
\end{lemma}
\begin{proof} To see that the square on the left commutes, we use Corollary \ref{cor:DualKatzPhiPsi}:
\[ \varphi^n \cK_1^\ast = \varphi^n \cK^\ast \ev_1^\ast = \cK^\ast \varphi_{\cC}^\ast \ev_1^\ast = \cK^\ast (\ev_1 \varphi_{\cC})^\ast = \cK^\ast \ev_{\pi^n}^\ast.\]
Hence in view of Lemma \ref{lem:BigPairing}(4), it remains to show that
\[\langle u_m, \cK^\ast( s^\ast (\lambda))\rangle = \langle u_m, t (\cK^\ast_1(\lambda))\rangle \qmb{for all} m \geq 0, \lambda \in \cC^\ast.\]
Since $t$ kills the constant term of a power series in $A$, we have
\[\langle u_m, t(a) \rangle = \delta_{m \geq 1} \langle u_m, a \rangle \qmb{for all} a \in A.\]
Now $\cK(u_m)(0) = P_m(0) = \delta_{m,0}$ by \cite[Lemma 4.2]{ST} and $\cK(u_0) = \cK(1) = 1$, so
\[\langle u_m, \cK^\ast(s^\ast(\lambda)) \rangle=\lambda(s(\cK(u_m)) = \lambda(\cK(u_m) - \cK(u_m)(0) 1) = \delta_{m \geq 1} \lambda(\cK(u_m)) = \langle u_m, t (\cK^\ast(\lambda))\rangle.\]
The result follows.\end{proof}

Let $c_0(o_\infty) := \{ (x_n)_n \in \prod\limits_{n=0}^\infty o_\infty : \lim\limits_{n \to \infty} x_n = 0 \}$.

\begin{lemma}\label{lem:explicitGalCont} Suppose that $\tau$ is surjective. Then the map
\[ \eta : \cC \to o_L \oplus c_0(o_\infty)\]
given by $\eta(f) = (f(0), (f(\pi^n) - f(0))_n)$ is an $o_L$-linear bijection.
\end{lemma}
\begin{proof}  Recall that any $f \in \cC$ takes values in $o_\infty$ by Lemma \ref{lem:valuesInOinf}. Since $\pi^n \to 0$ as $n \to \infty$ in $o_L$ and since $f$ is continuous, $f(\pi^n) - f(0) \to 0$ as $n \to \infty$ in $o_\infty$. Thus $\eta$ is well-defined. 

Suppose $\eta(f) = 0$ for some $f \in \cC$. Then $f(0) = 0$ and $f(\pi^n) = 0$ for all $n \geq 0$. Hence $f(\pi^n \tau(\sigma)) = \sigma(f(\pi^n)) = 0$ for all $\sigma \in G_L$, so $f$ also vanishes on $\pi^n \tau(G_L)$ for each $n \geq 0$. Since $\tau$ is surjective, $f$ vanishes on $\bigcup\limits_{n=0}^\infty \pi^n o_L^\times \cup \{0\} = o_L$, so $f = 0$. Hence $\eta$ is injective.

To show $\eta$ is surjective, let $(z,(z_n)_n) \in o_L \oplus c_0(o_\infty)$ and define $f : o_L \to o_\infty$ by setting $f(0) = z$ and $f(\pi^n \tau(\sigma)) := z + \sigma(z_n)$ for all $n \geq 0$ and all $\sigma \in G_L$. This makes sense because $\tau$ is surjective, and if $\tau(\sigma) = \tau(\sigma')$ for some $\sigma,\sigma' \in G_L$ then $\sigma^{-1}\sigma' \in \ker \tau$ fixes $o_\infty$ by Lemma \ref{lem:LOmegaDense}, so $\sigma'(z_n) = \sigma(\sigma^{-1}\sigma'(z_n)) = \sigma(z_n)$ for any $n \geq 0$. It is easy to see that $f : o_L \to o_\infty$ is $\Gal$-continuous and that $\eta(f) = (z,(z_n)_n)$. Hence $\eta$ is surjective.
\end{proof}
Lemma \ref{lem:explicitGalCont} allows us to give an explicit description of the space of Galois measures $\cC^\ast$.
\begin{corollary}\label{cor:dualeta} Suppose $\tau$ is surjective. Then 
\[\eta^\ast : o_L \oplus \prod\limits_{n=0}^\infty o_\infty^\ast \to \cC^\ast\]
is an $o_L$-linear bijection.
\end{corollary}
\begin{proof} The functor $(-)^\ast = \Hom_{o_L}(-,S)$ from $o_L$-modules to $S$-modules commutes with finite direct sums and sends $c_0(o_\infty)$ to $\prod\limits_{n=0}^\infty o_\infty^\ast$. Now apply this functor to the isomorphism $ \eta : \cC \stackrel{\cong}{\longrightarrow} o_L \oplus c_0(o_\infty)$ from Lemma \ref{lem:explicitGalCont}. \end{proof}
\begin{theorem}\label{thm:generalKatz} Suppose that $\tau$ is surjective and that $\cK_1^\ast  : o_\infty^\ast \to A^{\psi_q = 0}$ is an isomorphism. Then $\cK^\ast : \cC^\ast \to A^{\psiqint}$ is an isomorphism as well. 
\end{theorem}
\begin{proof} Using Corollary \ref{cor:dualeta} and Proposition \ref{prop:phipsiphi}, we can build the following diagram:
\[ \xymatrix{ S \oplus \prod\limits_{n=0}^\infty o_\infty^\ast \ar[rrrr]^{\eta^\ast}\ar[d]_{1 \oplus \prod\limits_{n=0}^\infty\cK_1^\ast} &&&& \cC^\ast \ar[d]^{\cK^\ast} \\ 
S \oplus \prod\limits_{n=0}^\infty A^{\psi_q=0} \ar[rrrr]_{1 \oplus \sum\limits_{n=0}^\infty \varphi^n t} &&&& A^{\psiqint}.}\]
Note that we can write $\eta = \ev_0 \oplus (\ev_{\pi^n} \circ s)_n$. Lemma \ref{lem:phiK1K} implies that 
\[\cK^\ast (\ev_{\pi^n} \circ s)^\ast = \cK^\ast s^\ast \ev_{\pi^n}^\ast = t \varphi^n \cK_1^\ast = \varphi^n t \cK_1^\ast \qmb{for any} n \geq 0.\]
Using $P_m(0) = \delta_{m,0}$ again together with $(\ref{eq:ExplicitKatz})$, we also have
\[\cK^\ast(\eta^\ast( 1, (0)_n )) = \cK^\ast (\ev_0^\ast(1)) = \sum\limits_{m=0}^\infty \ev_0^\ast(1)(P_m(-\Omega))Z^m = \sum\limits_{m =0}^\infty P_m(0) Z^m = 1.\] 
Therefore the diagram is commutative. Now $\eta^\ast$ is an isomorphism by Corollary \ref{cor:dualeta}, and bottom map is an isomorphism by Proposition \ref{prop:phipsiphi}. Since $\cK_1^\ast$ is an isomorphism by assumption, the vertical map on the left is an isomorphism as well. Hence $\cK^\ast$ is also an isomorphism by the commutativity of the diagram.
\end{proof}

\begin{corollary} Let $S$ be any $\pi$-adically complete $o_L$-algebra. The dual Katz map 
\[\cK^\ast : \cC^\ast \to S\dcroc{Z}^{\psiqint}\]
is an isomorphism if $\tau : G_L \to o_L^\times$ and $\cK_1 : \h{U} \to o_\infty$ are surjective, and $\cK : \h{U} \to \cC$ is injective. \end{corollary}
\begin{proof} Apply Theorem \ref{thm:generalKatz} together with Proposition \ref{prop:PartialConditionalKatz}.
\end{proof}

\subsection{The Newton polygon of $\Delta_1(Z)-1$}
\label{npsec} In this section, we obtain some estimates on $v_\pi(P_k(\Omega))$, $k \geq 1$. Recall that $d$ and $e$ and $f$ denote the degree and ramification and inertia indices of $L/\Qp$, respectively. 
\begin{lemma}
\label{olmodpim}
If $k \geq 0$ and $1 \leq r \leq e$, then we have an isomorphism of abelian groups
 \[o_L / \pi^{ek+r} o_ L \cong (\bZ/p^k \bZ)^{f(e-r)} \oplus (\bZ / p^{k+1} \bZ)^{fr}.\]
\end{lemma}
\begin{proof} Note that $p o_L = \pi^e o_L \subseteq \pi^r o_L$ since $e \geq r$ by assumption, so $o_L / \pi^r o_L$ is an elementary abelian $p$-group of order $|o_L / \pi^r o_L|  = p^{fr}$. Hence, using the elementary divisors theorem, we can find $v_1,\cdots,v_d \in o_L$ such that 
\[o_L = \Zp v_1 \oplus \cdots \oplus \Zp v_d \qmb{and} \pi^r o_L = \bigoplus\limits_{i=1}^s \Zp v_i \oplus \bigoplus_{i=s+1}^d \Zp p v_i\]
for some integer $s$ with $1 \leq s \leq d$. We deduce that $fr = d-s$, so $s = f(e-r)$. Since $\pi^{ek + r} o_L = p^k \pi^r o_L$, the result now follows easily. 
\end{proof}

\begin{lemma}
\label{ordof1}
In $o_L / \pi^{ek+r} o_ L$, the image of $1$ has order $p^{k+1}$.
\end{lemma}
\begin{proof}
This can be proved directly as $p^k \cdot 1 \in \pi^{ek} \cdot o_L^\times \neq 0$ in $o_L / \pi^{ek+r} o_ L$.
\end{proof}

\begin{definition} Let $m \geq 0$.
\begin{enumerate}
\item Let $k_m = \lfloor (m-1)/e \rfloor$, so that $m=ek_m+r$ with $1 \leq r \leq e$.
\item Define $x_m := q^m / p^{k_m+1}$.
\item Define
\[ y_0 = \frac{e}{p-1} - \frac{1}{q-1}  \ \ \text{and}\ \  y_m  = \frac{e}{p-1} - \sum_{j=1}^{m-1} \frac{1}{p^{k_j+1}} - \frac{q}{p^{k_m+1}(q-1)}. \]
\end{enumerate}
\end{definition}
For example, $x_0=1$ and $x_1=q/p$. Note that if $m=en+r$ with $1 \leq r \leq e$, then
\[ y_{en+r} = \frac{e}{p^n(p-1)} - \frac{r}{p^{n+1}} - \frac{1}{(q-1)p^{n+1}}. \]

\begin{theorem}
\label{npvert}
The vertices of the Newton polygon of $\Delta_1(Z)-1$ (using the valuation $v_\pi$, and excluding the point $(0,+\infty)$) are the points $(x_m,y_m)$ for $m \geq 0$.
\end{theorem}

\begin{proof}
Via the Schneider-Teitelbaum isomorphism, the zeroes of the power series 
\[\Delta_1(Z)-1 = \sum\limits_{m=1}^\infty P_m(\Omega)Z^m \in o_{\Cp} \dcroc{Z}\]
are the $z \in \frm_{\Cp}$ such that $\kappa_z$ is an $L$-analytic character satisfying $\kappa_z(1)=1$. These characters are torsion \footnote{Suppose that $\kappa(1) = 1$. Then $\kappa(a) = 1$ for all $ a \in \Zp$. Hence $\kappa'(1) = 0$. Since $\kappa$ is locally $L$-analytic, $\kappa'$ is $L$-linear, and hence $\kappa' = 0$ so that $\kappa$ is locally constant, and hence torsion.}, and correspond to some of the torsion points of the Lubin-Tate group $\cG$. There are precisely $q^m$ points in $\cG[\pi^m]$, and the common valuation of each point $z \in \cG[\pi^m] \setminus \cG[\pi^{m-1}]$ is  $v_\pi(z) = 1/q^{m-1}(q-1)$.

If we write $m=ek+r$ as above, then in view of Lemma \ref{olmodpim} and Lemma \ref{ordof1} there are $x_m = q^m/p^{k_m+1}$ elements $z \in \cG[\pi^m]$ such that $\kappa_z(1)=1$.  

Let $((x_m',y_m'))_{m = 0}^\infty$ be the vertices of the Newton polygon, so that the first vertex is $(x_0',y_0') = (1, v_\pi(\Omega)) = (x_0,y_0)$. The slope of the line segment between $(x'_{m-1},y'_{m-1})$ and $(x'_m,y'_m)$ is minus the common valuation of the elements of $z \in \cG[\pi^m] \setminus \cG[\pi^{m-1}]$ satisfying $\kappa(z) = 1$, that is $1/q^{m-1}(q-1)$. Hence $x'_m = x_m$ for all $m \geq 0$. Using the definitions of $x_m$ and $y_m$, we have the formula
\[ y_m = y_0 - \frac{x_1-x_0}{q^{1-1}(q-1)} - \cdots - \frac{x_m-x_{m-1}}{q^{m-1}(q-1)}\]
which implies that $y'_m = y_m$ for all $m \geq 0$.
\end{proof}

\begin{remark}
\label{realchcknp}
As $m \to +\infty$, $y_m \to  0$, consistent with the fact that $\| \Delta_1(Z)-1\| = 1$.
\end{remark}

\begin{corollary}
\label{valpkgen}
We have the following formulas for $v_\pi(P_k(\Omega))$.
\begin{enumerate}
\item For all $m \geq 0$, we have $v_\pi(P_{x_m}(\Omega)) = y_m$.
\item For all $n \geq 0$, we have $v_\pi(P_{p^{n(d-1)}}) = 1/p^n \cdot v_\pi(\Omega)$.
\end{enumerate}
\end{corollary}

\begin{proof}
Item (1) follows immediately from Theorem \ref{npvert}. Item (2) follows from item (1) with $m=en$. Indeed, $x_{en}  = q^{en}/p^n = p^{n(d-1)}$ and \[ y_{en} = \frac{e}{p-1} - \frac{e}{p} - \frac{e}{p^2} - \cdots - \frac{e}{p^{n-1}} - \frac{e-1}{p^n} - \frac{q}{p^n(q-1)} = \frac{1}{p^n} \cdot \left( \frac{e}{p-1} - \frac{1}{q-1}\right). \qedhere \] 
\end{proof}

\begin{remark}
\label{vpkunr}
If $L/\Qp$ is unramified, then item (2) of Corollary \ref{valpkgen} gives all the valuations of the $P_k(\Omega)$ that can be computed using the Newton polygon. For $n \geq 0$, we get \[ \vp(P_{p^{n(d-1)}}) = 1/p^n \cdot v_\pi(\Omega) = \frac{1}{p^{n-1}(p-1)} \cdot \frac{q/p-1}{q-1} .\]
\end{remark}

\begin{corollary}
\label{vpkpk}
Suppose that $L = \bQ_{p^2}$ and $\pi = p$. Then we have 
\[\vp(P_{p^k}(\Omega)) = \frac{1}{p^{k-1}(q-1)} \qmb{for all} k \geq 1,\]
and if $k \geq 1$ and $p^{k-1} \leq m \leq p^k$, then
\[ \vp(P_m(\Omega)) \geq \frac{1}{p^{k-1}(q-1)} + \frac{p^k-m}{q^{k-1}(q-1)} = \frac{1}{p^{k-2}(q-1)} - \frac{m-p^{k-1}}{q^{k-1}(q-1)}. \]
\end{corollary}

\subsection{Verifying Conjecture \ref{conj:NiceKatz} in a special case} \label{Qp^2KatzSection}
\begin{definition} Fix $m \geq 1$.
\begin{enumerate}
\item Let $\cG_m = \cG[\pi^m]$ be the finite flat $o_L$-group scheme of $\pi^m$-torsion points in the Lubin-Tate formal group $\cG$.
\item Let $\cG_m'$ be the Cartier dual of $\cG_m$.
\item Let $U(m) := \cO(\cG_m') = \Hom_{o_L}(o_L\dcroc{Z} / \langle \varphi^m(Z) \rangle, o_L).$
\item Let $\cG' := \colim \cG'_m$ be the dual $p$-divisible group to the $p$-divisible group defined by the formal group $\cG$. 
\end{enumerate}
\end{definition}

Recall that by Cartier duality --- see \cite[p. 177]{Tate66} --- the period $\Omega \in \Cp$ corresponds to a choice of generator $t' \in T_p \cG' = T_{\pi}\cG'$ as an $o_L$-module. We recall how this correspondence works. First, the element 
\[\Delta_1 = \sum\limits_{n = 0}^\infty P_n(\Omega)Z^n \in o_{\Cp}\dcroc{Z}\]
gives a compatible system of group-like elements $(\Delta_1(m))_{m=1}^\infty \in \prod\limits_{m=1}^\infty \cO(\cG_m)$, where $\Delta_1(m)$ is the image of $\Delta_1$ in $\cO(\cG_m \times_{o_L} o_{\Cp}) = o_{\Cp}\dcroc{Z} / \langle \varphi^m(Z) \rangle$ under the natural surjective homomorphism of $o_{\Cp}$-algebras $o_{\Cp}\dcroc{Z} \twoheadrightarrow \cO(\cG_m \times_{o_L} o_{\Cp})$. Since $\cO(\cG_m\times_{o_L} o_{\Cp})$ can be identified with $\Hom_{o_{\Cp}}(\cO(\cG'_m\times_{o_L} o_{\Cp}), o_{\Cp})$, $\Delta_1(m)$ can be viewed as an $o_{\Cp}$-linear map $U(m) \otimes_{o_L} o_{\Cp} \to o_{\Cp}$ which is in fact an $o_{\Cp}$-algebra homomorphism because $\Delta_1(m)$ is group-like. This map is determined by its restriction to $U(m)$; this restriction is an $o_L$-algebra homomorphism $t'_m : U(m) \to o_{\Cp}$ and is therefore an element of $\cG'_m(\Cp)$. Finally, the multiplication-by-$\pi$-maps $\cG'_{m+1}(\Cp) \to \cG'_m(\Cp)$ in the inverse system defining the Tate module $T_{\pi}\cG'$ are induced by the inclusions of $o_L$-algebras $U(m) \hookrightarrow U(m+1)$, so $t'_{m+1|U(m)} = t'_m$ for all $m \geq 1$, and the generator $t' \in T_{\pi} \cG'$ is given by $t' = (t'_m)_{m=1}^\infty \in \prod\limits_{m=1}^\infty \cG'_m(\Cp)$.

\begin{lemma}\label{K1rest} Let $m \geq 1$. The restriction of $\cK_1$ to $U(m) \subset \h{U}$ is equal to $t'_m$.
\end{lemma}
\begin{proof} Recall that we have identified $\h{U}$ with $o_L\dcroc{Z}^\ast_{\cts}$ using Lemma \ref{lem:BigPairing}(3). Let $u \in U(m)$ and let $\tilde{u} \in \h{U}$ be the corresponding $o_L$-linear map $o_L\dcroc{Z} \to o_L$ which kills $\langle \varphi^m(Z) \rangle$. Then 
\[t'_m(u) = \Delta_1(m)(u) =  \langle \tilde{u}, \Delta_1 \rangle  = \cK(\tilde{u})(1) = \cK_1(\tilde{u})\]
and the result follows. \end{proof}

For each $m \geq 1$, let $L_m$ be the finite Galois extension of $L$ contained in $L_\infty = \Cp^{\ker \tau}$ defined by $\Gal(L_\infty/L_m) = \tau^{-1}(1 + \pi^m o_L)$. 

\begin{lemma}\label{ImTm} Let $m \geq 1$. Then $t'_m(U(m)) \subseteq o_{L_m}$.
\end{lemma}
\begin{proof} Let $\sigma \in \Gal(L_\infty/L_m)$ so that $\tau(\sigma) \in 1 + \pi^m o_L$. Then by definition of the character $\tau$, $\sigma$ acts trivially on $\cG'_m(\Cp)$. In other words, $\sigma(t'_m(u)) = t'_m(u)$ for all $u \in U(m)$ and hence $t'_m(U(m)) \subseteq L_\infty^{\Gal(L_\infty/L_m)} = L_m$. But $U(m)$ is a finitely generated $o_L$-module so $t'_m(U(m))$ is integral over $o_L$ and is therefore contained in $o_{L_m}$.\end{proof}

\begin{definition} For each $m \geq 1$, let $U(m)_k := \im(U(m) \to \h{U} / \pi \h{U})$. \end{definition}
We will identify $U_k := U/\pi U$ with $\h{U} / \pi \h{U}$ via the natural map $U / \pi U \to \h{U} / \pi \h{U}$ and we regard $U(m)_k$ as being naturally embedded into $U(m+1)_k$.

\begin{proposition}\label{prop:K1onto} Suppose that $t'_m( U(m) ) = o_{L_m}$ for all $m \geq 1$. Then $\cK_1 : \h{U} \to o_\infty$ is surjective.
\end{proposition}
\begin{proof} Consider $o_\tau := \overline{L} \cap o_\infty$. Since $o_\tau$ is $\pi$-adically dense in $o_\infty$, to prove that $\cK_1(\h{U})$ contains $o_\infty$, it is enough to prove that it contains $o_\tau$. Fix $m \geq 1$. By Lemma \ref{K1rest}, the restriction of $\cK_1 : \h{U} \to o_{\Cp}$ to $U(m)$ is equal to $t'_m$. Hence by assumption $o_{L_m} = t'_m(U(m)) = \cK_1(U(m))$, so $o_\tau = \bigcup\limits_{m \geq 1} o_{L_m}$ is also contained in $\cK_1(\widehat{U})$. \end{proof}

\begin{lemma}\label{U(m)modpi} For each $m \geq 1$, we have $U(m) + \pi \h{U} = \sum\limits_{r=0}^{q^m-1} o_L u_r + \pi \h{U}$.
\end{lemma}
\begin{proof} Let $u \in U(m)$ and let $\tilde{u} : o_L\dcroc{Z} \to o_L$ be the corresponding $o_L$-linear form which vanishes on $\langle \varphi^m(Z) \rangle$. Consider $v := \tilde{u} - \sum\limits_{r=0}^{q^m - 1} \tilde{u}(Z^r) u_r \in \h{U}$. For each $r < q^m$, $u_r$ sends $\langle \varphi^m(Z)\rangle$ into $\pi o_L$ because $\varphi^m(Z) \equiv Z^{q^m} \mod \pi o_L\dcroc{Z}$. Since $\tilde{u}$ kills $\langle \varphi^m(Z) \rangle$, we see that $v$ also sends $\langle \varphi^m(Z)\rangle$ into $\pi o_L$. By construction, $v$ is zero on $1, Z, \cdots, Z^{q^m - 1}$. Since 
\begin{equation}\label{oLZdecomp}o_L 1 \oplus o_L Z \oplus \cdots \oplus o_L Z^{q^m - 1} \oplus \langle \varphi^m(Z)\rangle = o_L\dcroc{Z},\end{equation}
we conclude that $v\left( o_L\dcroc{Z} \right) \subseteq \pi o_L$ and hence $v = \pi w$ for some $o_L$-linear form $w : o_L\dcroc{Z} \to o_L$. Since $v : o_L\dcroc{Z} \to o_L$ is continuous for the weak topology on $o_L\dcroc{Z}$, so is $w$. Hence $w \in \h{U}$ and hence $\tilde{u} \in \sum\limits_{r=0}^{q^m-1} o_L u_r + \pi \h{U}$. This shows that $\subseteq$ holds.

For the reverse containment, it is enough to show that $u_r \in U(m) + \pi \h{U}$ for each $r = 0,\ldots, q^m -1$. Using  (\ref{oLZdecomp}), define an $o_L$-linear form $w_r : o_L\dcroc{Z} \to o_L$ which is zero on $\langle \varphi^m (Z) \rangle$ and which sends $Z^i$ to $\delta_{i,r}$ for each $0 \leq i < q^m$. Since $u_r$ sends $\langle \varphi^m(Z) \rangle$ into $\pi o_L$, the same is true of $u_r - w_r$. Since $u_r - w_r$ is zero on $1, Z, \cdots, Z^{q^m - 1}$ by construction, we see that $u_r - w_r$ sends all of $o_L\dcroc{Z}$ into $\pi o_L$. Hence $u_r - w_r = \pi v_r$ for some $o_L$-linear form $v_r : o_L\dcroc{Z} \to o_L$. Since $u_r - w_r$ is continuous for the weak topology on $o_L\dcroc{Z}$, so is $v_r$. Because $w_r$ is zero on $\langle \varphi^m(Z) \rangle$,  it lies in $U(m)$ and hence $u_r = w_r + \pi v_r \in U(m) + \pi \h{U}$. \end{proof}

\begin{proposition}\label{prop:Qp2KatztmUm} If $L = \mathbb{Q}_{p^2}$, then $t'_m( U(m) ) = o_{L_m}$ for all $m \geq 1$.\end{proposition}
\begin{proof} Fix $m \geq 1$. By Lemma \ref{U(m)modpi}, for each $0 \leq r < q^m$ we can find $w_r \in U(m)$ such that $w_r - u_r \in \pi \h{U}$. Set $r := p^{2m-1} = p q^{m-1} < q^m$. Note that $\cK_1(u_r) = \cK(u_r)(1) = \langle u_r, \Delta_1 \rangle = P_r(\Omega)$. Since $L = \mathbb{Q}_{p^2}$, Corollary \ref{vpkpk} applied with $k = 2m-1$ tells us that
\[ \vp(\cK_1(u_r)) = \vp( P_r(\Omega) ) = \frac{1}{p^{2m-2}(q-1)} = \frac{1}{q^{m-1}(q-1)} = [L_m : L]^{-1} < 1.\]
Now $\pi o_L = p o_L$ since $L = \mathbb{Q}_{p^2}$, so $\cK_1(u_r - w_r) \in \cK_1(\pi \h{U}) \subseteq p o_{\Cp}$ since $\cK_1$ takes values in $o_{\Cp}$. Hence $\vp(\cK_1(u_r) - \cK_1(w_r)) \geq 1$ and $\vp(\cK_1(w_r)) = \vp(\cK_1(u_r)) = [L_m : L]^{-1}$. Therefore $\cK_1(w_r)$ is a uniformiser in $L_m$ and the result follows. \end{proof}

Now we start to explore the injectivity of $\cK : \h{U} \to \cC$.

\begin{lemma} For each $m \geq 1$, we have $U(m) \cap \pi \h{U} = \pi U(m)$.
\end{lemma}
\begin{proof} Let $g = \pi h \in U(m)$ for some $h \in \h{U}$. Then $\pi \langle h, F \rangle = \langle \pi h, F \rangle = 0$ for any $F \in \langle \varphi^m(Z) \rangle$. Hence $\langle h, F \rangle = 0$ for all such $F$ as well, so $h \in U(m)$ and $g \in \pi U(m)$. \end{proof}

\begin{corollary}\label{Umk} The map $\cO(\cG_m' \times_{o_L} k) = U(m) / \pi U(m) \to U(m)_k$ is an isomorphism.
\end{corollary} 

Since $\cG'$ forms a $p$-divisible group, we have a closed immersion $\cG_m' \to \cG_{m+1}'$ for each $m \geq 1$. The comorphism of this map $\cO(\cG_{m+1}') \to \cO(\cG_m')$ is the dual of the $o_L$-Hopf algebra map $\cO(\cG_m) \to \cO(\cG_{m+1})$ induced by $\varphi : \cO(\cG) \to \cO(\cG)$. Using Corollary \ref{Umk}, we obtain connecting maps $\varphi_k^\ast : U(m+1)_k \to U(m)_k$.

\begin{lemma}\label{PhiSurj} The comorphisms $\varphi_k^\ast : U(m+1)_k \to U(m)_k$ are surjective for all $m \geq 1$.
\end{lemma}
\begin{proof} By Corollary \ref{Umk}, $U(m)_k$ is isomorphic to $\cO(\cG_m' \times_{o_L} k) = \Hom_k(\cO(\cG_m \times_{o_L} k),k)$ as a $k$-vector space. Since $\varphi(Z) \equiv Z^q \mod \pi o_L\dcroc{Z}$, we have $\cO(\cG_m \times_{o_L} k) = k\dcroc{Z} / \langle Z^{qm} \rangle$ and the $k$-algebra homomorphism $\varphi_k : k\dcroc{Z} / \langle Z^{qm} \rangle \to k\dcroc{Z} / \langle Z^{q(m+1)} \rangle$ which sends $Z$ to $Z^q$ is \emph{injective}. Hence the dual map 
\[\varphi^\ast_k : \Hom_k( k\dcroc{Z} / \langle Z^{q(m+1)} \rangle, k) \to \Hom_k( k\dcroc{Z} / \langle Z^{qm} \rangle, k) \]
is surjective and the result follows. \end{proof}

Next we consider an ideal $I$ of $U_k$ and we set $I(m) := I \cap U(m)$ for all $m \geq 1$. We assume that $I$ is $\varphi^\ast$-stable, in the sense that $\varphi^\ast(I) \subseteq I$.

\begin{proposition}\label{ProjLimToCoLim} Suppose that $I$ is a $\varphi^\ast$-stable ideal of $U_k$ such that $\varprojlim \frac{U(m)_k}{I(m)}$ is finite dimensional over $k$. Then $U_k / I = \colim \frac{U(m)}{I(m)}$ is also finite dimensional over $k$.
\end{proposition}
\begin{proof} Let $m \geq 1$ and consider the short exact sequence 
\[ 0 \to I(m) \to U(m)_k \to U(m)_k / I(m) \to 0.\]
Since $I$ is $\varphi^\ast$-stable by assumption, we get a short exact sequence of towers of finite-dimensional $k$-vector spaces. Passing to the inverse limit therefore gives an exact sequence
\[0 \to I(\infty) := \varprojlim I(m) \to \varprojlim U(m)_k \to \varprojlim \frac{U(m)_k}{ I(m)} \to 0.\]
By assumption, the term on the right is a finite dimensional $k$-vector space. We see from Lemma \ref{PhiSurj} that the connecting maps $U(m+1)_k / I(m+1) \to U(m)_k / I(m)$ induced by $\varphi^\ast$ are \emph{surjective}. Therefore, for large $m$, all of these maps are necessarily isomorphisms, and therefore there exists $m_0 \geq 1$ such that
\[ \dim \frac{U(m+1)_k}{I(m+1)} = \dim \frac{U(m)_k}{I(m)} \quad\mbox{for all} \quad m \geq m_0.\]
Now the definition of $I(m)$ shows that the natural connecting maps in the opposite direction $U(m)_k / I(m) \to U(m+1)_k / I(m+1)$ is \emph{injective} for any $m \geq 1$. They are therefore isomorphisms whenever $m \geq m_0$. The result follows. \end{proof}

\begin{proposition}\label{KerKmodpi} Let $J = \ker \cK$ and let $I := (J + \pi \h{U}) / \pi \h{U}$ be its image in $U_k$. Then $I$ is a $\varphi^\ast$-stable ideal in $U_k$ such that $\dim U_k / I = \infty$.
\end{proposition}
\begin{proof} Since $\cK \varphi^\ast = \varphi_{\cC} \cK$ by Lemma \ref{lem:KatzPhiPsi}, we see that $J$ is a $\varphi^\ast$-stable ideal in $\h{U}$. Hence its image $I$ in $U_k$ is also $\varphi^\ast$-stable.

Suppose that $h \in \h{U}$ and $r \geq 1$ are such that $\pi^r h \in J$. Then $\cK(\pi^r h) = 0$ in $\cC$, so $\cK(h) = 0$ as well. So $J \cap \pi^r \h{U} = \pi^r J$ for all $r \geq 1$. Now consider the short exact sequence
\[ 0 \to J \to \h{U} \to \cK(U) \to 0.\]
Equip both $\h{U}$ and $\cK(U)$ with the $\pi$-adic filtrations. Then the above shows that the subspace filtration on $J$ induced by the $\pi$-adic filtration on $\h{U}$ coincides with the $\pi$-adic filtration on $J$. Therefore we get a short exact sequence of $\gr o_L$-modules
\[ 0 \to \gr J \to \gr \h{U} \to \gr \cK(U) \to 0.\]
So, if $\dim U_k / I < \infty$, then $\gr \h{U} / \gr J \cong (U_k / I) [\gr \pi]$ is a finitely generated module over $\gr o_L$, so $\gr \cK(U) $ is a finitely generated $\gr o_L$-module. The $\pi$-adic filtration on $\cC$ is separated, hence the $\pi$-adic filtration on $\cK(U)$ is also separated. Therefore $\cK(U)$ is a finitely generated $o_L$-module by \cite[Chapter I, Theorem 5.7]{LVO}. Hence $\cK(U[1/\pi])$ is a finite dimensional $L$-vector space. But this contradicts \cite[Theorem 4.7]{ST}: the space of locally $L$-analytic Gal-continuous functions is not finite dimensional over $L$ since it contains the subspace of locally constant Gal-continuous functions, which is infinite dimensional over $L$.
\end{proof}

\begin{corollary}\label{cor: Qp2Katz} Suppose that $d := [L:\Qp] = 2$. Then $\cK : \h{U} \to \cC$ is injective.
\end{corollary}
\begin{proof} By Proposition \ref{KerKmodpi}, $I = (\ker \cK + \pi \h{U})/ \pi \h{U}$ is a $\varphi^\ast$-stable ideal in $U_k$ of infinite codiension in $U_k$. Hence $I(\infty) := \varprojlim (I \cap U(m)_k)$ is an ideal of infinite codimension in $\varprojlim U(m)_k$ by Proposition \ref{ProjLimToCoLim}. By \cite[Example 2.5.3]{Hopkins2019}, the Dieudonn\'e module $M(\cG_k)$ associated with the Lubin-Tate formal group $\cG_k = \cG \times_{o_L} k$ over the perfect field $k$ has basis $\{\gamma, V\gamma, \cdots, V^{d-1}\gamma\}$ over $\mathbb{W}(k)$ and satisfies $V^d = p$. Hence the Verschiebung operator $V$ on $M(\cG_k)$ is topologically nilpotent. Therefore the Cartier dual $\cG'_k$ is connected. Hence $\varprojlim U(m)_k \cong \cO(\cG' \times_{o_L} k)$ is isomorphic to $k\dcroc{X_1,\cdots,X_{d-1}}$ by \cite[Propositions 1 and 3]{Tate66}. Since $d = 2$, we conclude that $I(\infty) = 0$. Hence $I(m) = 0$ for all $m \geq 1$ and hence $I = 0$. So $\ker \cK = 0$ as well.
\end{proof}

\begin{theorem}\label{thm:partialKatzIsoForQp2} Suppose that $L = \bQ_{p^2}$. Then 
\[\cK_1^\ast : o_\infty^\ast \to o_L\dcroc{Z}^{\psi_q=0}\] 
is an $o_L$-linear bijection.
\end{theorem}
\begin{proof} Since $d = 2$, we know that $\tau$ is surjective by Lemma \ref{lem:tau-sur}. Then $\cK : \h{U} \to \cC$ is injective by Corollary \ref{cor: Qp2Katz} and $\cK_1 : \h{U} \to o_\infty$ is surjective by Proposition \ref{prop:K1onto} and Proposition \ref{prop:Qp2KatztmUm}. Now apply Proposition \ref{prop:PartialConditionalKatz}.
\end{proof}

We can now prove Theorem \ref{introkatzmap} from the Introduction. In fact, we prove the following more general version, from which Theorem \ref{introkatzmap} follows as a special case by setting $S = o_K$.

\begin{theorem}
\label{thm:katzisom}
Let $L = \bQ_{p^2}$ and let $S$ be a $\pi$-adically complete $o_L$-algebra.
\begin{enumerate}
\item The map $\cK^\ast : \Hom_{o_L}(\cC^0_{\Gal}(o_L,o_{\Cp}), S) \to S \dcroc{Z}$ is injective.
\item Its image is equal to $S \dcroc{Z}^{\psiqint}$. 
\end{enumerate}
\end{theorem}
\begin{proof} Since $d = 2$, we know that $\tau$ is surjective by Lemma \ref{lem:tau-sur}.  By Theorem \ref{thm:partialKatzIsoForQp2}, the map $\cK_1^\ast : o_\infty^\ast \to o_L\dcroc{Z}^{\psi_q=0}$ is an isomorphism. Now apply Theorem \ref{thm:generalKatz}.\end{proof}

\section{Integer-valued polynomials}

\subsection{The algebraic dual of $\cO^\circ(\frX_K)$}
Pick a basis $\{v_1,\cdots,v_d\}$ for $o_L$ as a $\Zp$-module with $v_1 = 1$. We view $o_L$ as a $p$-valued group with $p$-valuation $\omega$ given by 
\[ \omega\left( \sum_{i=1}^d \lambda_i v_i \right) = 1 + \min_{1\leq i \leq d} \vp(\lambda_i).\]
Let $r$ be a real number in the range $1/p \leq r < 1$. Recall from \cite[\S 4]{ST2} that $D^{\Qp-\an}(o_L,K)$ carries a norm $||\cdot ||_r$ given by
\begin{equation}\label{NormDr} || \sum\limits_{\alpha \in \mathbb{N}^d} d_\alpha \mathbf{b}^\alpha ||_r = \sup\limits_{\alpha \in \mathbb{N}^d} |d_\alpha| r^{|\alpha|}.\end{equation}
where $b_i := \delta_{v_i} - 1 \in D^{\Qp-\an}(o_L,K)$ for $i=1,\cdots, d$, $\mathbf{b}^\alpha = b_1^{\alpha_1} \cdots b_d^{\alpha_d} \in D^{\Qp-\an}(o_L,K)$ and $|\alpha| = \tau\alpha = \alpha_1 + \cdots + \alpha_d$ for all $\alpha \in \mathbb{N}^d$. 
\begin{definition} Let $1/p \leq r < 1$.
\be \item Let $D^{\Qp-\an}_r(o_L,K)$ denote the completion of $D^{\Qp-\an}(o_L,K)$ with respect to $||\cdot ||_r$.
\item Let $\frX_0(r)_K := \Sp D^{\Qp-\an}_r(o_L,K)$.
\item Let $\frX(r)_K := \frX_K \cap \frX_0(r)_K = \Sp D^{L-\an}_r(o_L,K)$, where $D^{L-\an}_r(o_L,K)$ is the factor algebra of $D^{\Qp-\an}_r(o_L,K)$ by the ideal generated by the elements 
\[u_2 - v_2 u_1, \quad u_3 - v_3 u_1, \quad \cdots \quad, u_d - v_d u_1\]
where $u_i := \log(1 + b_i) \in D^{\Qp-\an}(o_L,K)$.
\ee
\end{definition}
As $r$ approaches $1$ from below, the $K$-affinoid varieties $\frX(r)_K$ form an increasing family of $K$-affinoid subvarieties of $\frX_K$: whenever $1/p \leq r < r' < 1$ we have
\begin{equation}\label{Chain} \mathbbm{1} \in \frX(1/p)_K \subset \cdots \subset \frX(r)_K \subset \frX(r')_K \subset \cdots \subset \frX_K  = \bigcup\limits_{1/p \leq r < 1} \frX(r)_K.\end{equation}
Here $\mathbbm{1} \in \frX_K$ is the \emph{trivial character}: the ideal generated by $b_1,\cdots,b_d$.
\begin{lemma}\label{Stalk} The completed local ring $\widehat{\cO_{\frX_K, \mathbbm{1}}}$ of $\frX$ at $\mathbbm{1}$ is isomorphic to a power series ring in one variable $b := b_1$ over $K$:
\[ \widehat{\cO_{\frX, \mathbbm{1}}} \cong K\dcroc{b}.\]
\end{lemma}
\begin{proof} We have $\cO(\frX_0(1/p)_K) = K \langle b_1/p, \cdots, b_d /p \rangle = K \langle u_1/p, \cdots, u_d/p \rangle$. Quotienting out by the  ideal generated by the elements $u_i - v_i u_1$ shows that $\cO(\frX_0(1/p)_K) = K \langle u_1/p \rangle = K \langle b/p\rangle$. So $\frX_0(1/p)_K$ is isomorphic to the closed disc of radius $|p| = 1/p$ with local coordinate $b$; it is well known that the completed local ring at $b = 0$ of such a disc is $K\dcroc{b}$. The result follows since $\mathbbm{1} \in \frX(1/p)_K$ implies that $\widehat{\cO_{\frX_K, \mathbbm{1}}} = \widehat{\cO_{\frX(1/p)_K, \mathbbm{1}}} = K\dcroc{b}.$ \end{proof}
Applying the functor $\cO^\circ$ to the increasing chain of rigid $K$-varieties $(\ref{Chain})$ and using Lemma \ref{Stalk} yields a decreasing chain of $o_K$-algebras
\begin{equation}\label{AlgChain} K\dcroc{b} \supset \cO^\circ(\frX(1/p)_K) \supset \cdots \supset \cO^\circ(\frX(r)_K) \supset \cO^\circ(\frX(r')_K) \supset \cdots \supset \cO^\circ(\frX_K) \supseteq o_K\dcroc{o_L}.\end{equation}
\begin{definition} Let $A$ be an $o_K$-subalgebra of $K\dcroc{b}$ and let $m \geq 0$. The \emph{$m$-th infinitesimal neighbourhood of $\mathbbm{1}$ in $A$} is the image $A_m$ of $A$ in $K\dcroc{b}/b^{m+1}K\dcroc{b}$:
\[ A_m := \frac{A + b^{m+1} K\dcroc{b}}{b^{m+1} K\dcroc{b}} \quad \subset \quad \frac{K\dcroc{b}}{b^{m+1}K\dcroc{b}}.\]
\end{definition}
\begin{remark} This construction respects inclusions and compatible with variation in $m$. More precisely, whenever $A \subseteq B$ are two $o_K$-subalgebras of $K\dcroc{b}$, for every $n \geq m$ there is a commutative diagram of $o_K$-algebras
\[ \xymatrix{ A_n \ar[r] \ar[d] & B_n \ar[d] \\ A_m \ar[r] & B_m }\]
with injective horizontal arrows and surjective vertical arrows.
\end{remark}

\begin{definition}\label{AlgDual} Let $A$ be an $o_K$-subalgebra of $K\dcroc{b}$ and let $A_m^\ast := \Hom_{o_K}(A_m, o_K)$ for each $m \geq 0$. The \emph{algebraic dual} of $A$ is
\[A_\infty^\ast := \underset{m \geq 0}{\colim}{} \hsp A_m^\ast.\]
\end{definition}
\begin{lemma}\label{Lex} Let $o_K\dcroc{o_L} \subseteq A \subseteq B$ be two $o_K$-subalgebras of $K\dcroc{b}$ and let $n \geq m \geq 0$. 
\be \item In the commutative square
\[ \xymatrix{ A_n^\ast & B_n^\ast \ar[l] \\ A_m^\ast \ar[u] & B_m^\ast\ar[l]\ar[u] }\]
all arrows are injective.
\item The map $B_\infty^\ast \to A_\infty^\ast$ is injective.
\ee\end{lemma}
\begin{proof} (1) The vertical maps $A_m^\ast \to A_n^\ast$ are injective because $A_n \to A_m$ is surjective. Let $C$ be the cokernel of the map $A_n \to B_n$. Since $A_n$ contains $o_K\dcroc{o_L}_n$ which is an $o_K$-lattice in $K\dcroc{b}_n$, we see that $C$ is a torsion $o_K$-module. The dual functor $(-)^\ast$ is left exact, so we have the exact sequence $0 \to C^\ast \to B_n^\ast \to A_n^\ast$. Since $C$ is torsion, $C^\ast = 0$ which shows the injectivity of the horizontal arrows in our diagram.

(2) This follows by taking the colimit over all of the horizontal maps in part (1) above.
\end{proof}
Thus we see that the connecting maps appearing in the colimit in Definition \ref{AlgDual} are injective. Applying the contravariant algebraic dual functor $(-)^\ast_\infty$ to the chain $(\ref{AlgChain})$ and using Lemma \ref{Lex}(2) gives us a chain of algebraic duals
\[ \cO^\circ(\frX(1/p)_K)^\ast_\infty \subset \cdots \subset \cO^\circ(\frX(r)_K)^\ast_\infty \subset \cO^\circ(\frX(r')_K)^\ast_\infty \subset \cdots \subset \cO^\circ(\frX_K)^ \ast_\infty \subseteq o_K\dcroc{o_L}^\ast_\infty.\]
We can now calculate the largest one of these, namely the algebraic dual of the Iwasawa algebra $o_K\dcroc{o_L}$, but first we must introduce integer-valued polynomials. Recall the following notion from \cite{Bhargava}.

\begin{definition}\label{PO} A \emph{$\pi$-ordering} for $o_L$ is a subset $\{\alpha_0,\alpha_1,\alpha_2,\ldots\}$ of $o_L$ such that
\begin{equation}\label{POrd} v_\pi\left( \prod_{i=0}^{k-1}(\alpha_k - \alpha_i) \right) = \inf\limits_{s \in o} v_\pi \left(\prod_{i=0}^{k-1}(s - \alpha_i) \right) \quad\mbox{for all} \quad k \geq 1.\end{equation}
\end{definition}

Starting from an arbitrary element $\alpha_0 \in o_L$, it is possible to construct a $\pi$-ordering $\{\alpha_0, \alpha_1, \ldots\}$ of $o_L$ by induction on $k$, choosing at each stage $\alpha_k$ to minimise the expression appearing on the right hand side of $(\ref{POrd})$. In particular, $\pi$-orderings always exist, but are far from unique. 

\begin{definition}\label{Lagrange}  Let$\{\alpha_0,\alpha_1,\ldots\}$ be a $\pi$-ordering for $o_L$.
\begin{enumerate}
\item Define the \emph{Lagrange polynomials} as follows: $f_0(X) := 1$ and
\[ f_k(X) := \frac{ (X - \alpha_0)(X - \alpha_1)\cdots(X - \alpha_{k-1})}{ (\alpha_k - \alpha_0)(\alpha_k - \alpha_1)\cdots(\alpha_k - \alpha_{k-1})} \in L[X] \quad\mbox{for each}\quad k \geq 1. \]
\item Suppose that $R$ is an $o_L$-algebra which embeds into $R_L := R \otimes_{o_L} L$. Then we define the ring of \emph{$R$-valued polynomials on $o_L$} as follows:
\[ \Int(o_L, R) := \{ g(X) \in R_L[X] : g(o_L) \subset R \}\]
\item For each $m \geq 0$, let $\Int(o_L,R)_m$ denote the $R$-submodule of $\Int(o_L,R)$ consisting of all $R$-valued polynomials on $o_L$ of degree at most $m$.
\end{enumerate}\end{definition}

The following result, closely related to de Shalit's work on Mahler bases \cite{dSh}, explains why we are interested in these Lagrange polynomials.

\begin{lemma}\label{IntBasis} $\{f_0, f_1, f_2, \ldots\}$ is an $R$-module basis for $\Int(o_L,R)$.
\end{lemma}
\begin{proof} It follows directly from Definition \ref{PO} that $v_\pi(f_k(s)) \geq 0$ for all $s \in o_L$ and all $k \geq 0$. Hence $f_k(o_L) \subset o_L \subset R$ for all $k \geq 0$ which implies that \begin{equation}\label{Rfi} R f_0 + R f_1 + R f_2 + \cdots + R f_n + \cdots \quad \subseteq \quad \Int(o_L,R).\end{equation}
If $g \in R_L[X]$ has degree $n$ and leading coefficient $\lambda$, then $g - \lambda (\alpha_n - \alpha_0)\cdots(\alpha_n - \alpha_{n-1}) f_n$ has degree strictly less than $n$. This implies that $\{f_0,f_1,f_2,\ldots\}$ generates $R_L[X]$ as an $R_L$-module. Now let $g \in \Int(o_L,R)$ and write $g = \lambda_0 f_0 + \cdots + \lambda_n f_n$ for some $\lambda_0,\cdots,\lambda_n \in R_L$ as above. Setting $X = \alpha_0$ shows that $\lambda_0 = g(\alpha_0) \in R$ since $g \in \Int(o_L,R)$. Assume inductively that $\lambda_0,\ldots,\lambda_{t-1} \in R$ for some $1 \leq t \leq n$. Setting $X = \alpha_t$ shows that
\[\lambda_t = g(\alpha_t) - \lambda_0 f_0(\alpha_t) - \lambda_1 f_1(\alpha_t) - \cdots - \lambda_{t-1} f_{t-1}(\alpha_t)\]
and this lies in $R$ because $g(\alpha_t) \in R$ and $f_i(\alpha_t) \in R$ for all $i$. This completes the induction and shows that we have equality in $(\ref{Rfi})$. Taking $g = 0$ in the above argument also shows that the sum on the left hand side of $(\ref{Rfi})$ is direct.
\end{proof}
Using Lemma \ref{IntBasis}, we obtain the following
\begin{corollary}\label{RegularBasis}\hsp \be \item The multiplication map
\[ \Int(o_L, o_L) \otimes_{o_L} o_K \to \Int(o_L, o_K)\]
is an isomorphism, which sends $\Int(o_L,o_L)_m \otimes_{o_L} o_K$ onto $\Int(o_L,o_K)_m$ for any $m \geq 0$.
\item The Lagrange polynomials $\{f_0(Y), \cdots, f_m(Y)\}$ associated with a choice of $\pi$-ordering for $o_L$ form an $o_K$-module basis for $\Int(o_L,o_K)_m$.
\ee
\end{corollary}

\begin{proposition}\label{DualIwasawa} The evaluation map $\ev : \Int(o_L,o_K)_m   \longrightarrow o_K\dcroc{o_L}^\ast_m$ defined by
\[ \ev(f(Y))(\lambda) :=  \lambda(f(Y))\]
for all $f(Y) \in \Int(o_L,o_K)_m, \lambda \in o_K\dcroc{o_L}$ is an $o_K$-module isomorphism.
\end{proposition}
\begin{proof} This is essentially a complicated-looking tautology, but we try to give the details. 

Note that $o_K\dcroc{o_L}_m$ is an $o_K$-lattice in $K\dcroc{b}_m$. We can therefore identify $o_K\dcroc{o_L}_m^\ast$ with an $o_K$-submodule of $V := \Hom_K( K\dcroc{b}_m, K)$, a $K$-vector space of dimension $m+1$. The linear functionals $\ev(1), \ev(Y), \cdots, \ev(Y^m)$  are linearly independent in $V$ because if $\sum_{i=0}^m c_i \ev(Y^i) = 0$ then $\ev(\sum_{i=0}^m c_iY^i)(\delta_a) = \sum_{i=0}^m c_i a^i = 0$ for all $a \in o_L$ and this forces $c_0 = \cdots = c_m = 0$. It follows that $\ev : K[Y]_m \to V$ is injective and is therefore an isomorphism by the rank-nullity theorem. 

Hence $\ev : \Int(o_L,o_K)_m  \to o_K\dcroc{o_L}^\ast_m$ is injective. However if $g \in o_K\dcroc{o_L}^\ast_m$ then by the above we can find some $f(Y) \in K[Y]_m$ such that $\ev(f(Y)) = g$. Since $\delta_a \in o_K\dcroc{o_L}$ for all $a \in o_L$, we see that $f(a) = \ev(f(Y))(\delta_a) = g(\delta_a)$ must lie in $o_K$ for all $a \in o_L$.  \end{proof}

\begin{corollary}\label{IntDualIwasawa} The map $\ev: \Int(o_L,o_K) \to o_K\dcroc{o_L}^\ast_\infty$ is an isomorphism.
\end{corollary}
\begin{proof} This follows immediately from Proposition \ref{DualIwasawa}.\end{proof}

\begin{proposition}\label{Extends} Suppose that $K$ is discretely valued. Then 
\[\cO^\circ(\frX_K)^ \ast_\infty  = \underset{r < 1}{\colim}{} \hsp \cO^\circ(\frX(r)_K)^\ast_\infty.\]
\end{proposition}
\begin{proof} Since colimits commute with colimits, it is enough to show that for every $m \geq 0$, 
\[\cO^\circ(\frX_K)^ \ast_m  = \underset{r < 1}{\colim}{} \hsp \cO^\circ(\frX(r)_K)^\ast_m.\]
Fix $m \geq 0$. Then $\cO^\circ(\frX(r)_K)_m$ form a decreasing chain of $o_K$-submodules of the $m+1$-dimensional $K$-vector space $K\dcroc{b}_m$, and all of them contain the $o_K$-lattice $o_K\dcroc{o_L}_m$. Since $K$ is discretely valued, the $o_K$-module $(K/o_K)^{m+1}$ satisfies the descending chain condition. Hence there exists $r_0 < 1$ such that
\begin{equation}\label{Stabilised} \cO^\circ(\frX(r)_K)_m = \cO^\circ(\frX(r_0)_K)_m \qmb{whenever} r_0 \leq r < 1.\end{equation}
Following an argument of Schmidt \cite[proof of Proposition 4.9]{Schm1}, we will now show that 
\[\cO^\circ(\frX_K)_m  = \cO^\circ(\frX(r_0)_K)_m.\]
The forward inclusion is clear, so fix some $\xi \in \cO^\circ(\frX(r_0)_K)_m$, choose a sequence of real numbers $r_0 < r_1 < r_2 < \cdots$ approaching $1$ and consider the $K$-Banach space
\[A_j := \cO(\frX(r_j)_K).\]
Let $\varphi_j : A_j^\circ \to K\dcroc{b}_m$ be the obvious $o_K$-linear map. Using $(\ref{Stabilised})$ we see that the convex subset 
\[ \varphi_j^{-1}(\xi) \subset A_j\]
is non-empty. It was recorded in the proof of \cite[Lemma 6.1]{ST2} that the restriction maps $A_{j+1} \to A_j$ are compact. We may therefore argue as in \cite[Proposition V.3.2]{Gru} that 
\[\bigcap\limits_{j=0}^\infty \varphi_j^{-1}(\xi) \subseteq \cO^\circ(\frX_K)\]
is non-empty. Then any element $\lambda$ in this intersection satisfies $\lambda_m = \xi$, so $\xi \in \cO^\circ(\frX_K)_m$ as required. Hence $\cO^\circ(\frX_K)^\ast_m = \cO^\circ(\frX(r)_K)^\ast_m$ whenever $r_0 \leq r < 1$, and the result follows.\end{proof}

\subsection{The matrix coefficients $\rho_{i,j}(Y)$}
\label{RhoijSect}
Let $\mathbf{B}_{\Cp}$ be the rigid analytic open unit disc of radius $1$ defined over $\Cp$, with global coordinate function $Z$. There is a twisted $G_L = \Gal(\Cp/L)$-action on $\cO(\mathbf{B}_{\Cp})$ given by $F \mapsto F^\sigma \circ [\tau(\sigma^{-1})]$, which induces an $L$-algebra isomorphism
\[ \mu : \cO(\frX_L) \stackrel{\cong}{\longrightarrow} \cO(\mathbf{B}_{\Cp})^{G_L,\ast},\]
see \cite[Corollary 3.8]{ST}. Inspecting the proof of this result, we see that it extends naturally to give a description of $\cO(\frX_K)$ for more general closed coefficient fields $L \subseteq K \subseteq \Cp$ as well:

\begin{lemma}\label{KTwistedInv} There is a $K$-algebra isomorphism
\[ \mu_K : \cO(\frX_K) \stackrel{\cong}{\longrightarrow} \cO(\mathbf{B}_{\Cp})^{G_K,\ast}.\]
\end{lemma}
Since $\cO^\circ(\mathbf{B}_{\Cp}) = o_{\Cp}\dcroc{Z}$, we deduce the following
\begin{corollary}\label{muK} There is an isomorphism of $o_K$-algebras
\[ \mu_K : \cO^\circ(\frX_K) \stackrel{\cong}{\longrightarrow} o_{\Cp}\dcroc{Z}^{G_K, \ast}.\]
\end{corollary}

Until the end of $\S \ref{RhoijSect}$, we assume that $\Omega$ is transcendental over $K$.
\begin{definition}\label{AdmRalg} We call an $o_K$-subalgebra $R$ of $K[\Omega] \cap o_{\Cp}$ \emph{admissible} if $P_n(\Omega) \in R$ for all $n \geq 0$, and if $R$ is stable under the natural $G_L$-action on $K[\Omega] \cap o_{\Cp}$.
\end{definition}

\begin{example} $K[\Omega] \cap o_{\Cp}$ is itself an admissible $o_K$-subalgebra of $K[\Omega]$.
\end{example}
\begin{proof} This follows from Corollary \ref{muK} together with \cite[Lemma 4.2(5)]{ST}.\end{proof}

\begin{definition} \label{BnBasis} Let $R \subset K[\Omega]$ be an admissible $o_K$-subalgebra.
\be
\item Let $K[\Omega]_n := \{f(\Omega) \in K[\Omega]: \deg(f) \leq n\}$ for each $n \geq 0$.
\item Let $R_n := R \cap K[\Omega]_n$ for each $n \geq 0$.
\item $\{b_n(\Omega) : n \geq 0\} \subset R$ is a \emph{regular basis} if 
\[ b_0(\Omega) = 1, \qmb{and} R_n = R_{n-1} \oplus o_K b_n(\Omega) \qmb{for all} n \geq 1.\]
\ee
\end{definition}

\begin{lemma} Suppose that $K$ is discretely valued. Then a regular basis exists for every admissible $o_K$-subalgebra $R$ of $K[\Omega] \cap o_{\Cp}$.
\end{lemma}
\begin{proof} Since $\Omega$ is assumed to be transcendental over $K$, the $K$-vector space $K[\Omega]_n$ has dimension $n+1$. The restriction of the norm $|\cdot |$ on $\Cp$ to $K[\Omega]_n$ turns it into a normed vector space over $K$ and by Definition \ref{AdmRalg}(1), $R_n$ is contained in the unit ball with respect to this norm. Since any two norms on a finite dimensional $K$-vector space are equivalent --- see \cite[Proposition 4.13]{SchNFA} --- it follows that $R_n \subseteq \pi^{-m} o_K[\Omega]_n$ for sufficiently large $m$. 

Since $K$ is discretely valued, its valuation ring $o_K$ is Noetherian and this forces $R_n$ to be a free $o_K$-module of rank $n+1$. Because the $R_n$'s form a nested chain, we can now construct the desired $o_K$-module basis for $R$ by induction on $n$.
\end{proof}
\begin{example}\label{PnRegBas} Let $U := \sum\limits_{n = 0}^\infty o_K P_n(\Omega)$. Then $U$ is an admissible subalgebra of $K[\Omega]$, and $\{P_n(\Omega) : n \geq 0\}$ is a regular basis for $R$: since $\deg P_j(Y) = j$, an element $f(\Omega)$ of $U_n$ is a $K$-linear combination of $P_0(\Omega),\cdots,P_n(\Omega)$ lying in $U$, but $\{P_m(\Omega) : m \geq 0\}$ is an $o_L$-module basis for $U$ so all coefficients of $f(\Omega)$ must in fact lie in $o_L$.
\end{example}

Until the end of $\S \ref{RhoijSect}$, we assume that
\begin{itemize}
\item $K$ is a discretely valued intermediate subfield $L \subseteq K \subseteq \Cp$,
\item $\Omega$ is transcendental over $K$,
\item $R \subseteq K[\Omega] \cap o_{\Cp}$ is an admissible $o_K$-subalgebra, and
\item $\{b_n(\Omega) : n \geq 0\}$ is a regular basis for $R$.
\end{itemize}
\begin{lemma}\label{RhoDegJ} Let $j \geq 0$. \be \item There are unique $\rho_{0,j}(Y), \rho_{1,j}(Y),\cdots, \rho_{j,j}(Y) \in K[Y]$ such that
\[\label{PjYO} P_j(Y \Omega) = \sum\limits_{i=0}^j \rho_{i,j}(Y) b_i(\Omega).\]
\item $\deg \rho_{i,j}(Y) \leq j$ whenever $0 \leq i \leq j$.
\item $\deg \rho_{j,j}(Y) = j$.
\item $\rho_{i,j}(a) \in o_K$ whenever $a \in o_L$ and $0 \leq i \leq j$.
\ee \end{lemma}
\begin{proof} (1) $\Omega$ is transcendental over $K$, and $\{b_i(\Omega) : i \geq 0\}$ is a $K$-vector space basis for $K[\Omega]$ with $\deg b_i(\Omega) = i$ for each $i$. Hence it is also a $K[Y]$-module basis for the two-variable polynomial algebra $K[\Omega, Y]$, so we can find unique $\rho_{i,j}(Y) \in K[Y]$ such that 
\[P_j(Y \Omega) = \sum\limits_{i \geq 0} \rho_{i,j}(Y) b_i(\Omega)\]
where $\rho_{i,j}(Y) = 0$ for sufficiently large $i$. Now $P_j(s)$ is a polynomial in $s$ of degree $j$ by \cite[Lemma 4.2(3)]{ST}, so $\Omega^j$ is the highest degree monomial in $\Omega$ appearing in $P_j(Y \Omega)$. Since $\deg b_i(\Omega) = i$, this means $\rho_{i,j}(Y) = 0$ for $i > j$. 

(2) Since the highest degree monomial in $Y$ appearing in $P_j(Y\Omega)$ is $Y^j$, this means that $\deg \rho_{i,j}(Y) \leq j$ for each $i \leq j$. 

(3) The monomial $Y^j\Omega^j$ appears in $P_j(Y \Omega)$ with a non-zero coefficient. This monomial does not appear in $\rho_{i,j}(Y) b_i(\Omega)$ for any $i < j$ because $\deg b_i(\Omega) = i$ for all $i$. So it must appear in $\rho_{j,j}(Y) b_j(\Omega)$, and because of (2), this can only happen if $\deg \rho_{j,j}(Y) = j$.

(4) Let $a \in o_L$. We know that $P_j(a \Omega) \in o_{\Cp}$ by \cite[Lemma 4.2(5)]{ST}; in fact, $P_j(a \Omega)$ is an $o_L$-linear combination of the $P_i(\Omega)$ for $0 \leq i \leq j$ by Corollary \ref{MatrixCoeffs}, so $P_j(a \Omega) \in R$. Setting $Y = a$ in $(\ref{PjYO})$ shows that $\rho_{i,j}(a) \in o_K$, since $\{b_i(\Omega) : i \geq 0\}$ is a regular basis for $R$.
\end{proof}

\begin{theorem}\label{bkZj} For each $\lambda \in D^{L-\an}(o_L,K)$ we have
\[ \mu_K(\lambda) = \sum\limits_{j=0}^\infty \sum\limits_{k=0}^j \lambda(\rho_{k,j}(Y)) b_k(\Omega) Z^j.\]
\end{theorem}
In the case when $\lambda = \delta_a$ for some $a \in o_L$, Lemma \ref{RhoDegJ} implies that
\[ \mu(\delta_a) = \sum\limits_{j=0}^\infty P_j(a \Omega) Z^j = \sum\limits_{j=0}^\infty \left(\sum\limits_{i=0}^j \rho_{i,j}(a) b_i(\Omega)\right) Z^j = \sum\limits_{j=0}^\infty \sum\limits_{k=0}^j \delta_a(\rho_{k,j}(Y)) b_k(\Omega) Z^j\]
which explains where the formula comes from. We will now give a rigorous argument to show that the formula is valid for any $\lambda \in D^{L-\an}(o_L,K)$. 
\begin{lemma}\label{MuFiT}  Let $t := \log_{LT}(Z)$ be the Lubin-Tate logarithm. Then
\[\mu_K(\lambda) = \sum\limits_{k=0}^\infty \lambda( Y^k/k! ) \Omega^k t^k \qmb{for all} \lambda \in D^{L-\an}(o_L,K).\]
\end{lemma}
\begin{proof} Since we may identify $\Cp\dcroc{t}$ with $\Cp\dcroc{Z}$, we can write $\mu_K(\lambda) = \sum\limits_{m=0}^\infty c_{i,m} t^m$ for some $c_{i,m} \in \Cp$. Then applying \cite[Lemma 4.6(8)]{ST}, we have 
\[ \lambda(Y^k/k!) = \{\mu_K(\lambda), Y^k/k!\} = \frac{(\Omega^{-1} \partial_t)^k}{k!}(\mu_K(\lambda))(0) =  \Omega^{-k} c_{i,k}\quad\mbox{for all}\quad k \geq 0. \qedhere\]
\end{proof}
\begin{proposition}\label{tZformula} Let $\lambda \in \Hom_L(L[Y],K)$. Then in $\Cp\dcroc{t} = \Cp\dcroc{Z}$ we have
\[ \sum\limits_{k=0}^\infty \lambda(Y^k/k!) \Omega^k t^k = \sum\limits_{j=0}^\infty \sum\limits_{k=0}^j \lambda(\rho_{k,j}(Y)) b_k(\Omega) Z^j.\]
\end{proposition}
\begin{proof}For each $k \geq 0$, write $t^k = \sum\limits_{j=k}^\infty d_j^{(k)} Z^j \in L\dcroc{Z}$. Substituting this into Lemma \ref{MuFiT} gives
\begin{equation}\label{MuFiZ}\sum\limits_{k=0}^\infty \lambda(Y^k/k!) \Omega^k t^k = \sum\limits_{k=0}^\infty \lambda(Y^k / k!) \Omega^k \sum\limits_{j=k}^\infty d_j^{(k)} Z^j =
\sum\limits_{j=0}^\infty\left( \sum\limits_{k=0}^j \frac{1}{k!} d_j^{(k)}\Omega^k \lambda(Y^k) \right) Z^j. \end{equation}
On the other hand, the identity
\[\sum\limits_{j=0}^\infty P_j(Y\Omega )Z^j = \exp(Y \Omega t ) = \sum\limits_{k=0}^\infty \frac{1}{k!} t^k \Omega^k Y^k = \sum\limits_{k=0}^\infty \frac{Y^k\Omega^k}{k!}  \sum\limits_{j=k}^\infty d_j^{(k)} Z^j\]
together with Lemma \ref{RhoDegJ} shows that for all $j \geq 0$ we have
\begin{equation}\label{DjKBkO} \sum\limits_{k=0}^j \frac{1}{k!} d^{(k)}_j \Omega^k Y^k = P_j(\Omega Y) = \sum\limits_{k=0}^j \rho_{k,j}(Y)b_k(\Omega).\end{equation}
Now, the $L$-linear form $\lambda : L[Y] \to K$ extends to a $K[\Omega]$-linear form $K[\Omega,Y] \to K[\Omega]$. Applying this extension to $(\ref{DjKBkO})$ gives
\[\sum\limits_{k=0}^j \frac{1}{k!} d_j^{(k)}\Omega^k \lambda(Y^k) = \sum\limits_{k=0}^j \lambda(\rho_{k,j}(Y)) b_k(\Omega).\]
Substituting this equation into $(\ref{MuFiZ})$ gives the result. \end{proof}

\begin{proof}[Proof of Theorem \ref{bkZj}] Follows immediately from Lemma \ref{MuFiT} and Proposition \ref{tZformula}.
\end{proof}
\begin{definition}\label{CheckR} Let $\check{R}$ be the $o_K$-linear span of $\{\rho_{k,j}(Y) : j \geq k \geq 0\}$ in the space $I := \Int(o_L,o_K)$ of $o_K$-valued polynomials on $o_L$.
\end{definition}
We will see shortly that $\check{R}$ does not depend on the choice of regular basis for $R$.
\begin{corollary}\label{LambdaT} Let $\lambda \in D^{L-\an}(o_L,K)$. Then $\mu_K(\lambda) \in R\dcroc{Z}$ if and only if $\lambda(\check{R}) \subseteq o_K$.
\end{corollary}
\begin{proof} Theorem \ref{bkZj} tells us that $\mu_K(\lambda) \in R\dcroc{Z}$ if and only if $\sum\limits_{k=0}^j \lambda(\rho_{k,j}(Y)) b_k(\Omega) \in R$ for all $j \geq 0$. Since $\{b_k(\Omega) : k \geq 0\}$ is a regular basis, this is equivalent to $\lambda(\rho_{k,j}(Y)) \in o_K$ for all $j \geq k \geq 0$.
\end{proof}

\begin{proposition}\label{AlgToLAnDists} Let $\lambda \in \Hom_K(K[Y],K)$ be such that $\lambda(\check{R}) \subseteq o_K$. Then there exists $\tilde{\lambda} \in \mu_K^{-1}(R\dcroc{Z}) \subseteq \cO^\circ(\frX_K)$ such that $\tilde{\lambda}_{|K[Y]} = \lambda$.
\end{proposition}
\begin{proof} The twisted $G_L$-action on $\Cp\dcroc{Z}$ preserves $R\dcroc{Z}$ since we assumed that $R \subseteq K[\Omega] \cap o_{\Cp}$ is $G_L$-stable in Definition \ref{AdmRalg}. Therefore $R\dcroc{Z}^{G_L,\ast}$ makes sense.

Define $F_\lambda := \sum\limits_{j=0}^\infty \sum\limits_{k=0}^j \lambda(\rho_{k,j}(Y)) b_k(\Omega) Z^j \in \Cp\dcroc{Z}$.  Then $F_\lambda \in K\dcroc{\Omega t} = \Cp\dcroc{Z}^{G_K,\ast}$ by Proposition \ref{tZformula} and $F_\lambda \in R\dcroc{Z}$  because $\lambda(\check{R}) \subseteq o_K$. Hence $F_\lambda \in R\dcroc{Z}^{G_K,\ast} \subseteq o_{\Cp}\dcroc{Z}^{G_K,\ast}$, so $F_\lambda = \mu_K(\tilde{\lambda})$ for some $\tilde{\lambda} \in \cO^\circ(\frX_K)$ by Corollary \ref{muK}. In particular, $\tilde{\lambda} \in \mu_K^{-1}(R\dcroc{Z})$.

Next, applying \cite[Lemma 4.6(8)]{ST} we see that for all $m \geq 0$, 
\[\tilde{\lambda}(Y^m/m!) = \{\mu_K(\tilde{\lambda}), Y^m/m!\} = \{F_\lambda, Y^m/m!\} = \left\{\sum\limits_{k=0}^\infty \lambda(Y^k/k!) \Omega^k t^k, Y^m/m!\right\} = \lambda(Y^m/m!).\]
Since the $Y^m/m!$ span $K[Y]$ as a $K$-vector space, we conclude that $\tilde{\lambda}_{|K[Y]} = \lambda$.
\end{proof}

Recall the isomorphism $\ev : \Int(o_L,o_K) \to o_K\dcroc{o_L}^\ast_\infty$ from Corollary \ref{IntDualIwasawa}. 

\begin{theorem}\label{EvT} We have $\ev(\check{R}) = \mu_K^{-1}(R\dcroc{Z})^\ast_\infty$.
\end{theorem}
\begin{proof} $T$ contains the $o_K$-submodule of $K[Y]$ generated by $\{\rho_{j,j}(Y) : j \geq 0\}$ and $\deg \rho_{j,j}(Y) = j$ for each $j \geq 0$ by Lemma \ref{RhoDegJ}(3). Hence $\check{R}$ spans $K[Y]$ as a $K$-vector space.  On the other hand, $\check{R}_n := \check{R} \cap K[Y]_{\leq n}$ is contained in $\Int(o_L,o_K)_n$ by Lemma \ref{RhoDegJ}(4), which is a finitely generated $o_K$-module by Remark \ref{RegularBasis}(2). Since $K$ is discretely valued, $\check{R}_n$ is a finitely generated $o_K$-module for each $n \geq 0$. So we can find an $o_K$-module basis $\{t_0,t_1,\cdots,t_n, \cdots \}$ for $\check{R}$ such that $\{t_0,\cdots,t_n\}$ is an $o_K$-module basis for $\check{R}_n$ for each $n \geq 0$. It follows that the natural map $\check{R} \otimes_{o_K} K \to K[Y]$ is an isomorphism, and we may identify $\Hom_{o_K}(\check{R}, o_K)$ with $\{\phi \in \Hom_K(\check{R}, K) : \phi(\check{R}) \subseteq o_K\}$.

Let $\{t_m^\ast : m \geq 0\} \subset \Hom_{o_K}(\check{R}, o_K)$ be determined by 
\[t_m^\ast(t_n) = \delta_{m,n} \qmb{for all} m,n \geq 0.\] 
Then by Proposition \ref{AlgToLAnDists}, $t_m^\ast$ extends to some $\lambda_m \in \mu_K^{-1}(R\dcroc{Z})$ such that $\lambda_{m|K[Y]} = t_m^\ast$. In particular, we have $\lambda_m(t_n) = \delta_{m,n}$ for all $m,n \geq 0$.

Now suppose that $g \in \mu_K^{-1}(R\dcroc{Z})^\ast_\infty \subseteq o_K\dcroc{o_L}^\ast_\infty$. Then $g = \ev(h)$ for some $h \in \Int(o_L,o_K)_m$ by Proposition \ref{DualIwasawa}. Since $h \in K[Y]_{\leq m}$ and since $\{t_0,\cdots,t_m\}$ is a $K$-vector space basis for $K[Y]_m$, we can write $h = \sum_{n=0}^m c_n t_n$ for some $c_n \in K$. But then 
\[ g(\lambda_n) = \ev(h)(\lambda_n) = \lambda_n(h) = t_n^\ast(h) = c_n \qmb{for all} n \geq 0.\]
Since $\lambda_n \in \mu_K^{-1}(R\dcroc{Z})$ and $g \in \mu_K^{-1}(R\dcroc{Z})_\infty^\ast$, we conclude that $g(\lambda_n) \in o_K$ for all $n \geq 0$. Hence $h \in \sum_{n=0}^m o_K t_n \subseteq \check{R}$ and $g = \ev(h) \in \ev(\check{R})$. Hence $\mu_K^{-1}(R\dcroc{Z})_\infty^\ast \subseteq \ev(\check{R})$. 

Conversely, let $\lambda \in \mu_K^{-1}(R\dcroc{Z})$. Then $\lambda(\check{R}) \subseteq o_K$ by Corollary \ref{LambdaT} and thus for all $g \in \check{R}$, $\ev(g)(\lambda) = \lambda(g) \in o_K$. Hence $\ev(\check{R}) \subseteq  \mu_K^{-1}(R\dcroc{Z})_\infty^\ast$. \end{proof}

\begin{corollary}\label{CheckRevs} Let $S \subseteq R$ be two admissible subalgebras of $K[\Omega]$. Then $\check{R} \subseteq \check{S}$. 
\end{corollary}
\begin{proof} We have $\mu_K^{-1}(S\dcroc{Z}) \subseteq \mu_K^{-1}(R\dcroc{Z})$, so $\mu_K^{-1}(R\dcroc{Z})^\ast_\infty \subseteq \mu_K^{-1}(S\dcroc{Z})^\ast_\infty$ by Lemma \ref{Lex}(2). Hence $\ev(\check{R}) \subseteq \ev(\check{S})$ by Theorem \ref{EvT}. Hence $\check{R} \subseteq \check{S}$ because $\ev$ is an isomorphism by Corollary \ref{IntDualIwasawa}.
\end{proof}

Note that Theorem \ref{EvT} implies that the $o_K$-module $\check{R}$ depends only on the admissible subalgebra $R$ and not the particular choice of regular basis $\{b_n(\Omega) : n \geq 0\}$ for $R$. 

\begin{lemma}\label{LambdaI} Let $\lambda \in D^{L-\an}(o_L,K)$. Then $\lambda \in o_K\dcroc{o_L}$ if and only if $\lambda(\Int(o_L,o_K)) \subseteq o_K$.
\end{lemma}
\begin{proof} Suppose that $\lambda(\Int(o_L,o_K)) \subseteq o_K$. The $\pi$-adic completion of $I$ is naturally isomorphic to the ring $\cC^0(o_L,o_K)$ of $o_K$-valued continuous functions on $o_L$. Since $\lambda(I) \subseteq o_K$, $\lambda$ extends to an $o_K$-linear form $\tilde{\lambda} : \cC^0(o_L,o_K) \to o_K$ which is automatically continuous. View $\tilde{\lambda}$ as an element of $o_K\dcroc{o_L} = D^{\cts}(o_L,K)$. The restrictions of $\tilde{\lambda}$ and of $\lambda \in D^{L-\an}(o_L,K)$ to $K[Y]$ agree by construction. Since $K[Y]$ is dense in $C^{\an}(o_L,K)$, we conclude that $\lambda$ lies in $o_K\dcroc{o_L}$.

Conversely, if $\lambda \in o_K\dcroc{o_L} = \cC^0(o_L,o_K)^\ast$, then $\lambda$ must take integer values on $\Int(o_L,o_K) \subset \cC^0(o_L,o_K)$.\end{proof}

\begin{theorem}\label{muKRZ} Let $R$ be an admissible subalgebra of $K[\Omega]$. Then $\mu_K^{-1}(R\dcroc{Z}) = o_K\dcroc{o_L}$ if and only if $\check{R} = I$.
\end{theorem}
\begin{proof} $(\Leftarrow).$ Suppose that $\check{R} = I$, and let $\lambda \in \mu_K^{-1}(R\dcroc{Z})$. Then $\lambda(\check{R}) \subseteq o_K$ by Corollary \ref{LambdaT}. Since $\check{R} = I$, this means that $\lambda(I) \subseteq o_K$. Hence $\lambda \in o_K\dcroc{o_L}$ by Lemma \ref{LambdaI}. 

$(\Rightarrow).$ Suppose that $\check{R} < I$. Since $K$ is discretely valued, $K/o_K$ is an injective cogenerator of the category of $o_K$-modules. Hence $\Hom_{o_K}(I/\check{R}, K/o_K)$ is non-zero. So there exists an $o_K$-linear map $\lambda : I \to K$ such that $\lambda(\check{R}) \subseteq o_K$, but $\lambda(I) \nsubseteq o_K$. Regard $\lambda$ as an element of $\Hom_K(K[Y],K)$; then by Proposition \ref{AlgToLAnDists}, $\lambda$ extends to some $\tilde{\lambda} \in \cO^\circ(\frX_K)$ such that $\tilde{\lambda}_{|K[Y]} = \lambda$. Since $\lambda(\check{R}) \subseteq o_K$, using Theorem \ref{bkZj}  we see that $\mu_K(\tilde{\lambda}) \in R\dcroc{Z}$. However, $\tilde{\lambda} \notin o_K\dcroc{o_L}$ by Lemma \ref{LambdaI} because $\tilde{\lambda}(I) \nsubseteq o_K$, so $\tilde{\lambda} \in \mu_K^{-1}(R\dcroc{Z}) \backslash o_K\dcroc{o_L}$.
\end{proof}

We will now see what implications the above general results have for particular choices of the admissible subalgebra $R$. Let $B = K[\Omega] \cap o_{\Cp}$ be the largest possible admissible subalgebra of $K[\Omega]$, and let $U := \sum\limits_{n = 0}^\infty o_K P_n(\Omega)$ be the smallest possible one. Recall from Example \ref{PnRegBas} that $\{P_n(\Omega) : n \geq 0\}$ forms a regular basis for $U$.
\begin{corollary}\label{UandBcriteria} \hsp
\begin{enumerate}
 \item $\check{U} = \Int(o_L,o_K)$ if and only if $\mu_K^{-1}(U\dcroc{Z}) = o_K\dcroc{o_L}$.
 \item $o_K\dcroc{o_L} = \Lambda_K(\frX)$ if and only if $\check{B} = \Int(o_L,o_K)$. 
 \end{enumerate}
 \end{corollary}
\begin{proof} (1) This is an immediate consequence of Theorem \ref{muKRZ} with $R = U$.

(2) Theorem \ref{muKRZ} tells us that $\check{B} = I$ if and only if $o_K\dcroc{o_L} = \mu_K^{-1} (B\dcroc{Z})$. However $\mu_K^{-1}(B\dcroc{Z}) = \mu_K^{-1}( \Cp\dcroc{Z}^{G_L,\ast} \cap B\dcroc{Z})$ since $\mu_K(\cO(\frX)_K)$ is fixed by the twisted $G_L$-action on $\Cp\dcroc{Z}$ by Lemma \ref{KTwistedInv}. Hence $\mu_K^{-1}(B\dcroc{Z}) = \mu_K^{-1}( o_{\Cp}\dcroc{Z}^{G_L,\ast}) = \Lambda_K(\frX)$ by Corollary \ref{muK}, and the result follows.
\end{proof}

Recall the matrix coefficients $\sigma_{i,j}(a)$ from Corollary \ref{MatrixCoeffs}.
\begin{lemma}\label{CaseR=U} Let $R = U$ and let $b_n := P_n$ for each $n \geq 0$. Then 
\begin{enumerate}
\item $\rho_{ij}(Y) = \sigma_{i,j}(Y)$ for all $j \geq i \geq 0$, and 
\item $[a](Z)^i = \sum\limits_{j=i}^\infty \sigma_{i,j}(a) Z^j$ for any $a \in o_L, i \geq 0$.
\end{enumerate}
\end{lemma}
\begin{proof} (1) This follows by comparing Corollary \ref{MatrixCoeffs} with Lemma \ref{RhoDegJ}(1).

(2) Using Definition \ref{HypDef}(5) and Lemma \ref{PnU} we see that $\langle P_k(s), Z^i\rangle = \delta_{ki}$ for all $i,k \geq 0$. By Corollary \ref{MatrixCoeffs} we have $P_j(as) = \sum\limits_{k=0}^j \sigma_{kj}(a) P_k(s)$. Fix $i \geq 0$ and apply $\langle -, Z^i\rangle$ to this equation: using equation $(\ref{TransOact})$ we then have
\[ \sigma_{i,j}(a) = \left\langle \sum\limits_{k=0}^j \sigma_{kj}(a) P_k(s), Z^i \right\rangle = \langle P_j(as), Z^i\rangle = \langle P_j(s), [a](Z)^i\rangle.\]
Hence $\sigma_{i,j}(a)$ is precisely the coefficient of $Z^j$ in the power series $[a](Z)^i$.\end{proof}
This justifies the definition of the polynomials $\sigma_{i,j}(Y)$ which was given in $\S \ref{PnSectIntro}$.  We can now give the proof of Theorem \ref{intropolint} from the Introduction.

\begin{theorem} If $\Lambda_L(\frX) = o_L \dcroc{o_L}$, then $\Pol=\Int$.
\end{theorem}
\begin{proof} Note that $\Pol = \check{U}$, in view of Lemma \ref{CaseR=U}(1) and Definition \ref{CheckR}. Now $\Lambda_L(\frX) = \cO^\circ(\frX_L)$, so if this is equal to $o_L\dcroc{o_L}$, then $\check{B} = \Int(o_L,o_L)$ by Corollary \ref{UandBcriteria}(2). But $U \subseteq B$, so $\check{B} \subseteq \check{U} \subseteq \Int(o_L,o_L)$ by Corollary \ref{CheckRevs}. Hence $\check{U} = \Int(o_L,o_L)$ as claimed.
\end{proof}

\subsection{Calculating the matrix coefficients $\sigma_{i,j}(Y)$} Here we will assume that the coordinate $Z$ on the Lubin-Tate formal group is chosen in such a way that 
\[ \log_{LT}(Z) = \sum\limits_{n=0}^\infty \frac{Z^{q^n}}{\pi^n}.\]
It turns out that the polynomials $P_j(s)$ are \emph{sparse}: the coefficient of $s^i$ in $P_j(s)$ is non-zero \emph{only} if $i \equiv j \mod (q-1)$. We will obtain more information about these coefficients; this will require developing some notation to deal with this sparsity. The calculations that follow rest on the following observation.

\begin{proposition}
\label{polypm}
For every $n \geq 0$, we have
\[ P_n(Y) = \sum_{k_0+qk_1+\cdots+q^d k_d=n} \frac{Y^{k_0+\cdots+k_d}}{k_0! \cdots k_d! \cdot \pi^{1 \cdot k_1 + 2 \cdot k_2 + \cdots + d \cdot k_d}} \ .\]
\end{proposition}

\begin{proof}
If $\log_{\LT}(Z) = \sum_{k = 0}^\infty Z^{q^k}/\pi^k$ and $\exp$ is the usual exponential, then
\[ \sum_{n=0}^{\infty} P_n(Y) Z^n = \exp(Y \cdot \log_{\LT}(Z)) = \prod_{\ell \geq 0} \exp( Y \cdot Z^{q^\ell}/\pi^\ell) = \prod_{\ell \geq 0} \sum_{k \geq 0} ( Y \cdot Z^{q^\ell}/\pi^\ell)^k / k! \]
The coefficient of $Z^n$ in this product is the sum of $Y^{k_0+\cdots+k_d} / k_0! \cdots k_d! \cdot \pi^{1 \cdot k_1 + 2 \cdot k_2 + \cdots + d \cdot k_d}$ over all tuples $(k_0,\cdots,k_d)$ of positive integers such that $k_0+qk_1+\cdots+q^d k_d=n$. 
\end{proof}
The following formula for the derivative $\frac{d}{dY} P_n(Y)$ will be very useful in the calculations.
\begin{proposition} For every $n \geq 0$, we have 
\label{derivpm}
 $\frac{d}{dY} P_n(Y) = \sum_{k \geq 0} \pi^{-k} \cdot P_{n-q^k}(Y)$.
\end{proposition}
\begin{proof}
By \cite[Lemma 4.2(4)]{ST}, we have $P_n(Y+Z) = P_n(Y) + \sum_{j=1}^n P_j(Z)P_{n-j}(Y)$. Hence it is enough to determine which $P_j(Z)$ have a  term of degree $1$ in them, and what the corresponding coefficient is in this case. The answer now follows from Proposition \ref{polypm}.
\end{proof}

\textbf{We fix $m \in \{0,1, 2, \cdots, q-2\}$ from now on}. We will use the convenient notation 
\[\underline{i} := m + i(q-1) \quad\mbox{for all} \quad i \geq 0.\]

\begin{definition}\label{DefOfQ} For each $j \geq i \geq 0$, we define
\[ Q_m(i,j) := \left\{ \mathbf{k} \in \bN^\infty : \sum_{\ell=0}^\infty k_\ell = \underline{i},\quad \sum_{\ell=1}^\infty k_\ell \left(\frac{q^\ell-1}{q-1}\right) = j - i\right\},\quad\mbox{and}\]
\[ r_{i,j}^{(m)} := \sum\limits_{\mathbf{k} \in Q_m(i,j)} \binom{\underline{i}}{k_0;k_1;k_2;\cdots} \cdot \pi^{-\sum\limits_{\ell=1}^\infty \ell \cdot k_\ell}  .\]
\end{definition}
Here $\binom{\underline{i}}{k_0;k_1;k_2;\cdots} = \frac{ (\ul{i})!}{k_0!\cdot k_1! \cdot k_2! \cdots}$ is the multinomial coefficient. 
\begin{lemma}\label{r-Diagonal} We have $r_{jj}^{(m)} = 1$ for all $j \geq 0$.
\end{lemma}
\begin{proof} If $i = j$, then the second condition on a vector $\mathbf{k} \in \mathbb{N}^\infty$ to lie in $Q_m(i,j)$ forces $k_1 = k_2 = \cdots = 0$ because $\frac{q^\ell-1}{q-1} > 0$ for all $\ell \geq 1$. But then $k_0 = \ul{i} = \ul{j}$ from the first condition, so the formula for $r_{jj}^{(m)}$ collapses to give $1$.
\end{proof}

\begin{proposition}\label{CoeffOfPn} Let $n = \underline{j}$ for some $j \geq 0$. Write
\[ P_n(s) = \sum_{k=0}^n b_k^{(n)} s^k\]
with $b_k^{(n)} \in L$ for $k = 0,\ldots, n$. 
\begin{enumerate}
\item We have $b_k^{(n)} = 0$ if $k \not\equiv n \mod (q-1)$.
\item For each $0 \leq i \leq j$, we have $b_{\ul{i}}^{(\ul{j})} = \frac{r_{i,j}^{(m)}}{\ul{i}!}$.
\end{enumerate}
\end{proposition}
\begin{proof} By Proposition \ref{polypm}, the coefficient $b_k^{(n)}$ of $s^k$ in $P_n(s)$, is given by 
\[b_k^{(n)} = \sum\limits_{\mathbf{k}}\frac{1}{(k_0!k_1!k_2!\cdots)\pi^{0\cdot k_0 + 1\cdot k_1 + 2 \cdot k_2 + \cdots }},\]
where the sum runs over all possible sequences $\mathbf{k} = (k_0,k_1,k_2,\cdots)$ of non-negative integers satisfying the following two conditions:
\[ k_0 + k_1 + k_2 + \cdots = k,\quad\mbox{and}\quad k_0 + qk_1 + q^2k_2 + \cdots = n.\]
Of course given any such sequence, necessarily $k_\ell$ must be zero for all sufficiently large $\ell$ depending only on $n$ and $k$, and the set of solutions to these equations is always finite, so the sum of all these fractions makes sense. 

Next note that if $k_0,k_1,\cdots$ satisfies these two conditions, then necessarily 
\[n \equiv k \mod (q-1).\] 
This implies part (1). For part (2), let $k = \underline{i}$ and $n = \underline{j}$, and suppose that the non-negative integers $k_0,k_1,\cdots$ satisfy $k_0+k_1+\cdots = k$; then subtracting gives
\[ k_0 + q k_1 + q^2k_2 + \cdots = m + (q-1)j \quad \Leftrightarrow\quad (q-1)k_1 +(q^2-1)k_2 + \cdots = (q-1)(j-i).\]
In this way, we see that $Q_m(i,j)$ is precisely the set of sequences that contribute to the coefficient of $s^{\ul{i}}$ in $P_{\ul{j}}(s)$. This coefficient is then
\[b_k^{(n)} = \frac{1}{k!} \sum\limits_{\mathbf{k}\in Q_m(i,j)} \frac{k!}{k_0!k_1! \cdots} \cdot \pi^{-\sum\limits_{\ell=1}^\infty \ell \cdot k_\ell} = \frac{r_{i,j}^{(m)}}{k!}. \qedhere\]
\end{proof}

\begin{lemma}\label{OldNotation} Suppose that $j \geq i \geq 0$. Then $r^{(m)}_{ij}$ is the coefficient of $Z^{\ul{j}}$ in $\log_{LT}(Z)^{\ul{i}}$. 
\end{lemma}
\begin{proof}
Write $\log_{LT}(Z)^k = \sum_{n=k}^\infty d_n^{(k)} Z^n$. Then
\begin{equation*}
  \sum_{n=0}^\infty P_n(Y) Z^n = \exp(Y \log_{LT} (Z)) = \sum_{k=0}^\infty \frac{1}{k!} \log_{LT}(Z)^k Y^k = \sum_{k=0}^\infty \frac{1}{k!} \sum_{n=k}^\infty d_n^{(k)} Z^nY^k\ .
\end{equation*}
Equating the coefficent of $Z^n Y^k$ shows that 
\begin{equation*}
  b_k^{(n)} = \frac{1}{k!} d_n^{(k)}  \qquad\text{for $1 \leq j \leq n$}.
\end{equation*}
Applying Proposition \ref{CoeffOfPn}(2), we have $r_{i,j}^{(m)} = \ul{i}!  b_{\ul{i}}^{(\ul{j})} = d_{\ul{j}}^{(\ul{i})}$.
\end{proof}

\begin{corollary}\label{mRjDef} Define polynomials $R^{(m)}_j(t) \in L[t]$ for $j \geq 0$ by the formula
\[ R^{(m)}_j(t) := \sum\limits_{i=0}^j  \frac{ r_{i,j}^{(m)}}{(\underline{i})!} t^i.\]
Then for all $j \geq 0$ we have $P_{\underline{j}}(s)  = s^m\cdot{R^{(m)}_j(s^{q-1})}.$
\end{corollary}

\begin{lemma}\label{PolySigmas} For each $j \geq i \geq 0$ there exist $\sigma_{i,j}(Y) \in \Int(o_L,o_L)$ such that
\[ P_j(Ys) = \sum\limits_{i=0}^j \sigma_{i,j}(Y) P_i(s).\]
\end{lemma}
\begin{proof} By Example \ref{PnRegBas}, $\{P_n(\Omega) : n \geq 0\}$ forms a regular basis for the admissible subalgebra $\sum\limits_{n = 0}^\infty o_K P_n(\Omega)$ of $L[\Omega]$. Apply Lemma \ref{RhoDegJ} and use the transcendence of $\Omega$ over $L$.
\end{proof}

Of course this is just another way of rephrasing Corollary \ref{MatrixCoeffs}. We will now see that the matrix of polynomials $(\sigma_{i,j}(Y))_{i,j}$ is sparse as well.
\begin{proposition}\label{OldToNew} Let $j \geq 0$ and suppose that $0 \leq k \leq \ul{j}$.
\begin{enumerate} 
\item $\sigma_{k,\ul{j}}(Y) = 0$ if $k \not\equiv m \mod (q-1)$.
\item For each $i = 0,\ldots, j$ there exists $\tau^{(m)}_{i,j}(X) \in L[X]$ such that
\[\sigma_{\underline{i}, \underline{j}}(Y) = Y^m \cdot {\tau^{(m)}_{i,j}(Y^{q-1})}.\]
\end{enumerate}
\end{proposition}
\begin{proof} Using Lemma \ref{PolySigmas}, we have
\[ P_{\underline{j}}(Ys) = \sum_{k=0}^{\underline{j}} \sigma_{k, \underline{j}}(Y) P_k(s).\]
Dividing both sides by $Y^m s^m$ we obtain an equality of Laurent polynomials
\begin{equation}\label{mRj} R^{(m)}_j( Y^{q-1} s^{q-1} ) = \sum_{k=0}^{\underline{j}} Y^{-m} \sigma_{k, \underline{j}}(Y) \cdot s^{-m} P_k(s).\end{equation}
The left hand side of $(\ref{mRj})$ is a polynomial in $s^{q-1}$ with coefficients in $L[Y]$. The Laurent polynomial $s^{-m} P_k(s)$ lies in $s^{k-m}L[s^{q-1}, s^{1-q}]$ by Proposition \ref{CoeffOfPn}. Since
\[ L[Y][s, s^{-1}] = \bigoplus\limits_{c=0}^{q-2} s^c L[Y][s^{q-1},s^{1-q}],\]
looking at the component of the right hand side of $(\ref{mRj})$ that lies in $s^c L[Y][s^{q-1},s^{1-q}]$ for $c \in \{1,\cdots,q-2\}$ and then looking at the leading coeffiicent of $s^{-m} P_k(s)$ implies (1).

Using Corollary \ref{mRjDef}, we can now rewrite $(\ref{mRj})$ as follows:
\begin{equation}\label{mRj2} R^{(m)}_j( Y^{q-1} s^{q-1} ) = \sum_{i=0}^{j} Y^{-m} \sigma_{\underline{i}, \underline{j}}(Y) \cdot R^{(m)}_i(s^{q-1}).\end{equation}
Since the left hand side of $(\ref{mRj2})$ is now a polynomial in $Y^{q-1}$ with coefficients in $L[s^{q-1}]$, we deduce by looking at the right hand side of $(\ref{mRj2})$ that the \emph{a priori} Laurent polynomial $Y^{-m} \sigma_{\ul{i},\ul{j}}(Y)$ in $Y$ in fact lies in $L[Y^{q-1}]$. Part (2) follows.
\end{proof} 
Setting $t = s^{q-1}$ and $X = Y^{q-1}$, we deduce the following
\begin{corollary}\label{NewMCs} The polynomials $R^{(m)}_j(tX)$ satisfy 
\[ R^{(m)}_j(tX) = \sum\limits_{i=0}^j {\tau^{(m)}_{i,j}(X)} \hsp {R^{(m)}_i(t)}. \]
\end{corollary}

\begin{definition}\label{rTauDX} Consider the following infinite upper-triangular matrices.
\begin{enumerate}
\item $[r^{(m)}]_{ij} = r^{(m)}_{ij}$ for $j \geq i \geq 0$,
\item $\cT^{(m)}_{ij} = \tau_{i,j}^{(m)}(X)$, and
\item $\cD_X := \diag(1, X, X^2, \cdots)$.
\end{enumerate}
\end{definition}
\begin{lemma}\label{FirstMeq} We have the matrix equation
\[ r^{(m)} \cdot \cT^{(m)} = \cD_X \cdot r^{(m)}.\]
\end{lemma}
\begin{proof} Note that each matrix appearing on the right hand side has infinitely many rows and columns, but each one is also upper triangular, so matrix multiplcation makes sense. Moreover, because $r_{jj}^{(m)} = 1$ for all $j \geq 0$ by Lemma \ref{r-Diagonal}, the matrix $r^{(m)}$ is invertible, with inverse matrix having entries on $L$. 

Substitute the definition of $R^{(m)}_j(t)$ from Corollary \ref{mRjDef} into Corollary \ref{NewMCs} to obtain
\[  \sum\limits_{\ell=0}^j  \frac{ r^{(m)}_{\ell,j}}{(\underline{\ell})!} t^\ell X^\ell \quad = \quad \sum\limits_{i=0}^j {}\tau^{(m)}_{i,j}(X) \sum_{\ell=0}^i \frac{r^{(m)}_{\ell,i}}{(\ul{\ell})!} t^\ell.\]
Equate the coefficients of $t^\ell$ to get
\[ r^{(m)}_{\ell,j} X^\ell = \sum_{i=0}^j {}\tau^{(m)}_{i,j}(X) \cdot r^{(m)}_{\ell,i} .\]
The right hand side is the $(\ell,j)$-th entry of $r^{(m)} \cdot \cT^{(m)}$. The left hand side is the $(\ell,j)$-th entry of $\cD_X \cdot r^{(m)}$. The result follows. \end{proof}

The following two results on the coefficients $r_{i,j}^{(m)}$ are strictly speaking not needed for the calculations appearing in Appendix A, but they are nevertheless interesting in their own right.
\begin{lemma}\label{BigBrackets} For each $j \geq i \geq 0$, we have
\[ r_{i,j}^{(m)} = \left( \sum\limits_{\mathbf{k} \in Q_m(i,j)} \binom{\ul{i}}{k_0;k_1;\cdots} \pi^{\sum\limits_{\ell=1}^\infty k_\ell \left(\frac{q^\ell - 1}{q-1} - \ell\right)}\right) \cdot \pi^{i-j}.\]
\end{lemma}
\begin{proof} Let $\mathbf{k} \in Q_m(i,j)$. Then $\sum_{\ell=1}^\infty k_\ell \left(\frac{q^\ell-1}{q-1}\right) = j - i$, and therefore
\[\pi^{\sum\limits_{\ell=1}^\infty k_\ell \left(\frac{q^\ell - 1}{q-1} - \ell\right)} \cdot \pi^{i-j} = \pi^{j-i} \cdot \pi^{-\sum\limits_{\ell=1}^\infty\ell k_\ell} \cdot \pi^{i-j} = \pi^{-\sum\limits_{\ell=1}^\infty\ell k_\ell}.\]
The result now follows from Definition \ref{DefOfQ}.\end{proof}

\begin{proposition}\label{IntegralMCs} Let $j \geq i \geq 0$. Then
\begin{enumerate}
\item $\pi^{j-i} \cdot r_{i,j}^{(m)} \in o_L$, and
\item $\pi^{j-i} \cdot r_{i,j}^{(m)} \equiv \binom{\ul{i}}{j-i} \mod \pi^{q-1}o_L$.
\end{enumerate}
\end{proposition}
\begin{proof}(1) Note that for every $\ell \geq 1$ we have
\[\begin{array}{lll} \alpha_\ell := \frac{q^\ell - 1}{q-1} - \ell &=& \frac{ \left(1 + (q-1)\right)^\ell - 1}{q-1} - \ell = \frac{1 + \ell (q-1) + \binom{\ell}{2}(q-1)^2 + \cdots + (q-1)^\ell - 1}{q-1} - \ell\\
&=& \binom{\ell}{2}(q-1) + \binom{\ell}{3} (q-1)^2 + \cdots + (q-1)^{\ell-1}. \end{array}.\]
Thus $\alpha_\ell \geq 0$ always. Hence the expression in the big brackets in Lemma \ref{BigBrackets} lies in $o_L$.

(2) The exponent of $\pi$ appearing in the term in the sum corresponding to $\mathbf{k} \in Q_m(i,j)$ is equal to $\sum_{\ell=1}^\infty k_\ell \alpha_\ell$. It follows from the formula for $\alpha_\ell$ established above that $\alpha_1 = 0$. Hence this exponent is a positive multiple of $q-1$, \emph{unless} $k_\ell = 0$ for all $\ell \geq 2$. In this case, the exponent is $0$ and the corresponding term is equal to $\binom{\ul{i}}{j-i}$ because in this case $k_1 = \sum\limits_{\ell=1}^\infty k_\ell \frac{q^\ell-1}{q-1} = j - i$.
\end{proof}

\section{Consequences of the Katz isomorphism}
\label{consekatz}
\subsection{Equivariant endomorphisms of $L_\infty$}
\label{subequiv}

Throughout this {\S}, we assume that $L = \bQ_{p^2}$ and that $\pi=p$. 
In particular, $L_\infty$ is the completion of $L(\cG[p^\infty])$.
We recall the statement of the Katz isomorphism (Theorem \ref{thm:katzisom}): if $S$ is a $\pi$-adically complete $o_L$-algebra, then the map $\cK^\ast : \Hom_{o_L}(\cC^0_{\Gal}(o_L,o_{\Cp}), S) \to S \dcroc{Z}^{\psiqint}$ is an isomorphism. 

Note the following criterion.

\begin{lemma}
\label{suppsi}
A measure $\mu \in \Hom_{o_L}(\cC^0_{\Gal}(o_L,o_{\Cp}), S)$ is supported in $o_L^\times$ if and only if $\psi_q(\cK^\ast(\mu)) = 0$.
\end{lemma}

There is the usual $G_L,*$ action on $o_{\Cp}\dcroc{X}$, and on $\Hom_{o_L}(\cC^0_{\Gal}(o_L,o_{\Cp}),o_{\Cp})$ it is given by $g^*(\mu)(f) = g(\mu(g^{-1}(f))) = g(\mu(a \mapsto f(\tau(g)^{-1} \cdot a))$ since $f$ is Gal continuous. In particular, Theorem \ref{thm:katzisom} applied with $S=o_{\Cp}$ implies the following.

\begin{corollary}
\label{olkpz}
We have 
\begin{enumerate}
\item $\Hom_{o_L}(\cC^0_{\Gal}(o_L,o_{\Cp}),o_{\Cp})^{G_L,*} =  \Lambda_L(\frX)^{\psiqint}$.
\item $\Hom_{o_L}(\cC^0_{\Gal}(o_L^\times,o_{\Cp}),o_{\Cp})^{G_L,*} =  \Lambda_L(\frX)^{\psi_q=0}$.
\end{enumerate}
\end{corollary}

Since $L = \bQ_{p^2}$, the map $\tau$ is surjective. Let $\Gamma_L = \Gal(L(\cG[p^\infty])/L)$.

\begin{lemma}
\label{cogalinf}
The map $\cC^0_{\Gal}(o_L^\times,o_{\Cp}) \to o_\infty$ given by $f \mapsto f(1)$ is an isomorphism of $o_L$-modules.
\end{lemma}
\begin{proof}
This follows from the surjectivity of $\tau$. More precisely, if $x \in o_\infty$, let $f_x \in \cC^0_{\Gal}(o_L^\times,o_{\Cp})$ be given by $f_x(1) = x$ and $f_x(\tau(g)) = g(x)$. Every element of $\cC^0_{\Gal}(o_L^\times,o_{\Cp})$ is of this form. 
\end{proof}

Theorem \ref{thm:katzisom} applied with $S=o_L$ now gives us the following

\begin{theorem}
\label{dualoinfty}
The map $\cK^\ast$ gives rise to an $o_L$-linear isomorphism $o_\infty^* \simeq o_L \dcroc{Z}^{\psi_q=0}$.
\end{theorem}

\begin{proposition}
\label{iwasprop}
The space $\Hom_{o_L}(\cC^0_{\Gal}(o_L^\times,o_{\Cp}),o_{\Cp})^{G_L,*}$ is naturally isomorphic to 
the space of $\Gamma_L$-equivariant $o_L$-linear maps $o_\infty \to o_\infty$.
\end{proposition}

\begin{proof}
If $x \in o_\infty$, let $f_x \in \cC^0_{\Gal}(o_L^\times,o_{\Cp})$ be as in the proof of Lemma \ref{cogalinf} above. If $\mu \in \Hom_{o_L}(\cC^0_{\Gal}(o_L^\times,o_{\Cp}),o_{\Cp})^{G_L,*}$, we define a map $T : o_\infty \to o_\infty$ by $T(x) = \mu(f_x)$. We have $f_{x+y} = f_x + f_y$ and $f_{ax} = a f_x$ if $a \in o_L$ so that $T$ is $o_L$-linear. In addition, $T$ is $\Gamma_L$-equivariant because $\mu$ is fixed under the $G_L,*$-action. Indeed, $g(T(x)) = g(\mu(f_x)) = \mu(g(f_x))$ and $g(f_x)(1) = g(x)$ so that $g(f_x) = f_{g(x)}$. Therefore, $g(T(x)) = T (g(x))$.

Conversely, a $\Gamma_L$-equivariant $o_L$-linear map $T : o_\infty \to o_\infty$ as above gives an element $\mu \in \Hom_{o_L}(\cC^0_{\Gal}(o_L^\times,o_{\Cp}),o_{\Cp})^{G_L,*}$ via $\mu(f_x)= T(x)$.
\end{proof}

Combining Corollary \ref{olkpz} and Proposition \ref{iwasprop}, we get the following.

\begin{theorem}
\label{contgalend}
We have $\End^{G_L}_{o_L}(o_\infty) \simeq \Lambda_L(\frX)^{\psi_q=0}$.
\end{theorem}

\begin{corollary}
\label{iwascrit}
We have $\Lambda_L(\frX) = o_L\dcroc{o_L}$ if and only if every $\Gamma_L$-equivariant $o_L$-linear map $o_\infty \to o_\infty$ comes from an element of $o_L\dcroc{\Gamma_L}$.
\end{corollary}

\begin{proof}
By Lemma \ref{psizall} below, we have $\Lambda_L(\frX) = o_L\dcroc{o_L}$ if and only if $\Lambda_L(\frX)^{\psi=0} = \Lambda(o_L^\times)$. If $\mu \in \Lambda_L(\frX)^{\psi=0}$, then it corresponds to an element of $\Hom_{o_L}(\cC^0_{\Gal}(o_L^\times,o_{\Cp}),o_{\Cp})^{G_L,*}$ by Corollary \ref{olkpz}. By Proposition \ref{iwasprop}, the element $\mu \in \Lambda_L(\frX)^{\psi=0}$ comes from an element $\nu \in o_L\dcroc{\Gamma_L}$. The element $\mu$ then corresponds to the image of $\nu$ in $\Lambda(o_L^\times)$ via $\tau$. Indeed, if $g \in \Gamma_L$ and $T$ is given by $x \mapsto g(x)$, then it corresponds to $\mu : f_x \mapsto g(x)$ and $g(x) = f_x(\tau(g))$ so that $\mu = \delta_{\tau(g)}$.
\end{proof}

Using Corollary \ref{iwascrit}, we get the following

\begin{theorem}
\label{critcge}
We have $\Lambda_L(\frX) = o_L \dcroc{o_L}$ if and only if every continuous $L$-linear and $G_L$-equivariant map $f : L_\infty \to L_\infty$ comes from the Iwasawa algebra $L \otimes_{o_L} o_L \dcroc{\Gamma_L}$.
\end{theorem}
\begin{proof}Indeed, by Corollary \ref{tensint}, $\Lambda_L(\frX) \cap (L \otimes_{o_L} o_L \dcroc{o_L}) = o_L \dcroc{o_L}$.
\end{proof}

\begin{lemma}
\label{psizall}
If $\Lambda_L(\frX)^{\psi=0} = \Lambda(o_L^\times)$, then $\Lambda_L(\frX) = o_L\dcroc{o_L}$. 
\end{lemma}

\begin{proof}
If $f \in \Lambda_L(\frX)$, then $\delta_1 \cdot \varphi(f) \in \Lambda_L(\frX)^{\psi=0}$. So $\varphi(f) \in o_L\dcroc{o_L}$ and $f=\psi_q\varphi(f) \in o_L\dcroc{o_L}$. 
\end{proof}

The following is pretty much in Fourquaux's PhD; it implies that there are no Tate trace maps $L_\infty \to L$ or $L_\infty \to L_n$ (recall that $L_\infty$ is the completion of $L(\cG[p^\infty])$).

\begin{proposition}
\label{fouphd}
Let $f : L_\infty \to L_\infty$ be a continuous, $\Gamma_L$-equivariant and $L$-linear map. If $f(L_\infty)$ is included in a finite field extension of $L$, then $f(1)=0$. 
\end{proposition}

\begin{proof}
We have $\log \Omega \in L_\infty$ and $(g-1) \log \Omega = \log \tau(g)$ if $g \in \Gamma_L$. Hence 
\[ (g-1)f(\log \Omega) = f((g-1) \log \Omega) = f(\log \tau(g)) = \log \tau(g) \cdot f(1). \] Therefore if $f(1) \neq 0$, then $f(\log \Omega)$ is a period for $\log \tau$, and in particular does not belong to a finite extension of $L$.
\end{proof}

Proposition \ref{fouphd} can be strengthened. Almost the same proof gives us the following.

\begin{proposition}
\label{fouplus}
Let $f : L_\infty \to L_\infty$ be a continuous, $\Gamma_L$-equivariant and $L$-linear map. 
If $f \neq 0$, then there exists $a_1 \neq 0$, $a_0 \in L(\cG[p^\infty])$ such that $f(L_\infty)$ contains $a_1 \log \Omega + a_0$.
\end{proposition}

\begin{proof}
We have $\log \Omega \in L_\infty$ and $(g-1) \log \Omega = \log \tau(g)$ if $g \in \Gamma_L$. Take $x \in L(\cG[p^\infty])$ such that $f(x) \neq 0$, and choose (recall that $f(L_n) \subset L_n$ by Ax-Sen-Tate) some $n$ such that $x,f(x) \in L_n$. If $g \in \Gamma_n$, then
\[ (g-1)f(x \cdot \log \Omega) = f((g-1) (x \cdot \log \Omega)) = f(x \cdot \log \tau(g)) = \log \tau(g) \cdot f(x). \] Therefore $(g-1) ( f(x \cdot \log \Omega) - f(x) \cdot \log \Omega) = 0$ for all $g \in \Gamma_n$, so that $f(x \cdot \log \Omega) - f(x) \cdot \log \Omega \in L_n$ by Ax-Sen-Tate. We can take $a_1 = f(x)$ and $a_0 = f(x \cdot \log \Omega) - f(x) \cdot \log \Omega$.
\end{proof}

This can be strengthened even further. Let $L_\infty^{\alg}$ denote the locally algebraic vectors in $L_\infty$. Let $c(g) = \log \tau (g) = \log \chi_p^\sigma(g)$. The set $L_\infty^{\alg}$ is the set of $x \in L_\infty$ such that there exists an open subgroup $\Gamma_x$ of $\Gamma_L$ and $d \geq 0$ and $x_0=x,x_1,\hdots,x_d \in L_\infty$ such that $g(x) = x_0 + x_1 c(g) + \cdots + x_d c(g)^d$ if $g \in \Gamma_x$. Note that technically, these are the locally $\sigma$-analytic locally algebraic vectors in $L_\infty$. However since $L=\bQ_{p^2}$, every locally analytic vector is locally $\sigma$-analytic (see \cite{BerCol}). 

\begin{lemma}
\label{logomla}
We have $L_\infty^{\alg} = L(\cG[p^\infty]) [\log \Omega]$.
\end{lemma}

\begin{proof}
One inclusion is easy. Now take $x \in L_\infty^{\alg}$ and write $g(x) = x_0 + x_1 c(g) + \cdots + x_d c(g)^d$ if $g \in \Gamma_x$. On $L_\infty^{\alg}$ we have the derivative $\nabla : x \mapsto x_1$ and we know (from the theory of locally analytic vectors) that $\nabla^j (x) /j! = x_j$ for all $j$. In particular, $\nabla(x_d) = 0$, so that $x_d \in L(\cG[p^\infty])$. The element $x-x_d \log^d \Omega$ is then in $L_\infty^{\alg}$ and it is of degree $\leq d-1$, which allows us to prove the lemma by induction. 
\end{proof}

We see that $\nabla = \frac{d}{d\log \Omega}$. For all $n$, the map $\nabla : L_n[\log \Omega]  \to L_n[\log \Omega] $ is surjective, and its kernel is $L_n$. If $f : L_\infty \to L_\infty$ is a continuous, $\Gamma_L$-equivariant and $L$-linear map, then $f(L_\infty^{\alg}) \subset L_\infty^{\alg}$. In addition, $\nabla = \lim_{g \to 1} (g-1)/c(g)$ so that $f \circ \nabla = \nabla \circ f$.

\begin{proposition}
\label{foulocalg}
Let $f : L_\infty \to L_\infty$ be a continuous, $\Gamma_L$-equivariant and $L$-linear map. 
If $f \neq 0$,  there exists $n \geq 0$ such that $L_n \cdot f(L_n[\log \Omega])$ contains $L_n[\log \Omega]$.
\end{proposition}

\begin{proof}
Take $x \in L(\cG[p^\infty])$ such that $f(x) \neq 0$ and let $n \geq 0$ be such that $x,f(x) \in L_n$. We prove by induction on $d$ that $L_n \cdot f(L_n[\log \Omega])$ contains $L_n[\log \Omega]_{\deg \leq d}$. In order to do this, we prove that $f(x \cdot \log^d \Omega)$ is a polynomial (in $\log \Omega$) of degree $d$. The case $d=0$ follows from the fact that $f(x) \neq 0$. Now assume that the result holds for $d-1$. We have \[ \nabla f(x \cdot \log^d \Omega) = f( x \cdot \nabla \log^d \Omega) = f(d x \cdot \log^{d-1} \Omega), \]
so that $f(x \cdot \log^d \Omega)$ is a polynomial of degree $d$. This implies the claim.
\end{proof}

\subsection{The dual of the ring of integers of a $p$-adic Lie extensions}

Recall that $\pi \in o_L$ is a uniformiser and $k_L := o_L / \pi o_L$ is the residue field of $L$. In this {\S}, $L_\infty / L$ is an infinite Galois extension with Galois group $\Gamma = \Gal(L_\infty/L)$. We fix a chain 
\[\Gamma \supseteq \Gamma_1 \supseteq \Gamma_2 \supseteq \cdots \] 
of open normal subgroups of $\Gamma$ such that $\bigcap\limits_{n=1}^\infty \Gamma_n = 1$.
\begin{definition} Let $n \geq 1$.
\be \item $L_n  := L_\infty^{\Gamma_n}$, a finite Galois extension of $L$ with Galois group $\Gamma/\Gamma_n$.
\item $o_n$ is the integral closure of $o_L$ in $L_n$.
\item $o_n^\ast := \Hom_{o_L}(o_n, o_L)$. 
\item $k_n := o_n / \pi o_n$.
\item $k_n^\vee := \Hom_{k_L}(k_n, k_L)$.
 \ee \end{definition}
Note that $o_n$ and $o_n^\ast$ are naturally $o_L[\Gamma/\Gamma_n]$-modules, both free of finite rank as an $o_L$-module, and $k_n$ and $k_n^\ast$ are $k_L[\Gamma/\Gamma_n]$-modules, both finite dimensional over $k_L$.

\begin{remark}\label{InvDiff} Let $n \geq 1$.
\be \item $o_n^\ast$ can be identified with the \emph{inverse different} $\mathfrak{d}^{-1}_{L_n/L}$ of the extension $L_n/L$.
\item Applying the duality functor $(-)^\ast = \Hom_{o_L}(-,o_L)$ to the natural inclusion of $o_L$-modules $o_n \to o_{n+1}$, we obtain a natural connecting map $o_{n+1}^\ast \to o_n^\ast$. This map is surjective, because the $o_{n+1}/o_n$ is a finitely generated and torsion-free $o_L$-module.
\ee\end{remark}

\begin{lemma}\label{onmodpi} For each $n \geq 1$, there is a short exact sequence of $o_L[\Gamma/\Gamma_n]$-modules
\[ 0 \to o_n^\ast \stackrel{\pi}{\longrightarrow} o_n^\ast \to k_n^\vee \to 0.\]
\end{lemma}
\begin{proof} Let $M$ be an $o_L$-module and consider the complex of $o_L$-modules
\[ 0 \to M^\ast \stackrel{\pi}{\longrightarrow} M^\ast \stackrel{\eta_M}{\longrightarrow} (M / \pi M)^\vee \to 0\]
where $M^\ast := \Hom_{o_L}(M,o_L)$, $(M / \pi M)^\vee = \Hom_{k_L}(M/\pi M, k_L)$ and  $\eta_M(f)(m + \pi M) = f(m) + \pi o_L \in k_L$. This complex commutes with finite direct sums and is exact in the case when $M = o_L$. So the complex is exact whenever $M$ is a finitely generated free $o_L$-module. If $M$ also happens to be an $o_L[G]$-module for some group $G$, then the maps in the complex are $o_L[G]$-linear. The result follows when we set $M = o_n$, an $o_L[\Gamma/\Gamma_n]$-module which is free of finite rank as an $o_L$-module.
\end{proof}

We now pass to the limit as $n \to \infty$.
 
\begin{definition} Recall the Iwasawa algebras $\Lambda(\Gamma) = \varprojlim o_L[\Gamma/\Gamma_n]$ and  $\Omega(\Gamma) = \varprojlim k_L[\Gamma/\Gamma_n]$.
\be  
\item $o_\infty := \colim o_n$, an $o_L[\Gamma]$-module.
\item $o_\infty^\ast := \varprojlim o_n^\ast$, a $\Lambda(\Gamma)$-module.
\item $k_\infty := \colim k_n$, a $k_L[\Gamma]$-module.
\item $k_\infty^\vee := \varprojlim k_n^\vee $, an $\Omega(\Gamma)$-module.
\ee
\end{definition}

\begin{lemma}\label{oinftydualmodpi} There is a short exact sequence of $\Lambda(\Gamma)$-modules
\[ 0 \to o_\infty^\ast \stackrel{\pi}{\longrightarrow} o_\infty^\ast \to k_\infty^\vee \to 0.\]
\end{lemma}
\begin{proof} The short exact sequences from Lemma \ref{onmodpi} are compatible with variation in $n$, in other words we get a short exact sequence of towers of $\Lambda(\Gamma)$-modules. Applying the inverse limit functor gives a long exact sequence
\[ 0 \to o_\infty^\ast \stackrel{\pi}{\longrightarrow} o_\infty^\ast \to k_\infty^\vee \to \varprojlim{}^{(1)} o_n^\ast .\]
The $\varprojlim{}^{(1)}$ term on the right vanishes in view of Remark \ref{InvDiff}(2), whence the result.
\end{proof} 

Remark \ref{InvDiff}(2) also implies that the natural maps $o_\infty^\ast \to o_n^\ast$ are surjective.

\begin{proposition}\label{oinfdualfful} The $\Lambda(\Gamma)$-modules $o_\infty$ and $o_\infty^\ast$ are faithful.
\end{proposition}
\begin{proof} Suppose $\xi \in \Lambda(\Gamma)$ kills $o_\infty$. Then its image $\xi_n \in o[\Gamma/\Gamma_n]$ kills $o_n$. Therefore $\xi_n \in L[\Gamma/\Gamma_n]$ kills $L_n = o_n \otimes_{o_L} L$. But $L_n$ is a free $L[\Gamma/\Gamma_n]$-module of rank $1$ by the Normal Basis Theorem. So, $\xi_n = 0$ for all $n \geq 0$ and therefore $\xi = 0$ as well.

Suppose now $\xi \in \Lambda(\Gamma)$ kills $o_\infty^\ast$. Then $\xi$ kills each the quotients $o_n^\ast$ of $o_\infty^\ast$. But the action of $\Lambda(\Gamma)$ on $o_n^\ast$ factors through $o_L[\Gamma/\Gamma_n]$, so the image $\xi_n$ of $\xi$ in $o_L[\Gamma/\Gamma_n]$ kills $o_n^\ast$. Since $\xi_n$ also kills $o_n \cong (o_n^\ast)^\ast$, we deduce from the above that  $\xi_n = 0$ for all $n$. Hence $\xi = 0$.
\end{proof}

\begin{proposition}\label{Tame} Suppose that $p \nmid |\Gamma/\Gamma_1|$. Then $k_1^\vee$ is a free $k_L[\Gamma/\Gamma_1]$-module of rank $1$.
\end{proposition}
\begin{proof} The field extension $L_1/L$ is tamely ramified by our assumption on $|\Gamma / \Gamma_1|$. Now it follows from Noether's Theorem on rings of integers in tamely ramified extensions that $o_1$ is a free $o_L[\Gamma/\Gamma_1]$-module of rank one --- see, e.g. \cite[Proposition 2.1]{Thomas}. Hence $o_1 / \pi o_1$ is a free $k_L[\Gamma/\Gamma_1]$-module of rank one, and we can apply Lemma \ref{onmodpi} to conclude.
\end{proof}

\begin{lemma}\label{LimCoinv} Suppose that $\Gamma$ is a $p$-adic Lie group. Let $M = \varprojlim M_n$ be an inverse limit of a tower of $\Omega(\Gamma)$-modules, where each $M_n$ is finite dimensional over $k_L$. Then the natural map on $\Gamma$-coinvariants
\[ M_{\Gamma} \to \varprojlim (M_n)_\Gamma\]
is an isomorphism.
\end{lemma}
\begin{proof} The Iwasawa algebra $\Omega(\Gamma)$ is Noetherian, so its augmentation ideal $J = (\Gamma - 1)\Omega(\Gamma)$ is finitely generated. Let $u_1,\cdots, u_r \in J$ be generators and let $N$ be an $\Omega(\Gamma)$-module; then
\[ N_\Gamma = N / (\Gamma - 1)\cdot N = N / J N = N / (u_1 N + \cdots + u_r N).\]
In other words, we have the short exact sequence of $k_L$-vector spaces
\begin{equation}\label{Nr} N^r \stackrel{(u_1,\cdots,u_r)}{\longrightarrow} N \to N_\Gamma \to 0.\end{equation}
Applying this to each $M_n$, we obtain an exact sequence of towers of $\Omega(\Gamma)$-modules
\[ M_n^r \stackrel{(u_1,\cdots,u_r)}{\longrightarrow} M_n \to (M_n)_\Gamma \to 0\]
where each term is a finite dimensional $k_L$-vector space. The inverse limit functor is exact on such towers, since they all satisfy the Mittag-Leffler condition. So passing to the inverse limit we obtain the exact sequence of $k_L$-vector spaces
\[ M^r \stackrel{(u_1,\cdots,u_r)}{\longrightarrow} M \to \varprojlim (M_n)_\Gamma  \to 0.\]
Comparing this with $(\ref{Nr})$ applied with $N = M$ gives the result.
\end{proof}
\begin{theorem}\label{OinftyFreeRk1} Suppose that 
\begin{itemize}
\item $\Gamma$ is abelian,
\item $p \nmid |\Gamma/\Gamma_1|$,
\item $\Gamma_1$ is a torsionfree pro-$p$ group of finite rank. 
\end{itemize}
Then $o_\infty^\ast$ is a free $\Lambda(\Gamma)$-module of rank $1$ if and only if the map $k_1 \to k_\infty^{\Gamma_1}$ is an isomorphism. 
\end{theorem}
 \begin{proof} $(\Leftarrow)$ Note that the connecting maps $k_n \to k_{n+1}$ in the colimit $k_\infty := \colim k_n$ are injective: if $x + \pi o_n \in k_n$ maps to zero in $k_{n+1}$ then there is $y \in o_{n+1}$ such that $x = \pi y$; but then $y \in L_n \cap o_{n+1} = o_n$ and hence $x = \pi y \in \pi o_n$. Under our hypothesis that $k_1 \to k_\infty^{\Gamma_1}$ is an isomorphism, it follows that for each $n \geq 1$, the map  $k_n^{\Gamma_1} \to k_{n+1}^{\Gamma_1}$ is an isomorphism. Applying the $(-)^\vee = \Hom_{k_L}(-,k_L)$ functor, we deduce that for each $n \geq 1$, the map on $\Gamma_1$-coinvariants
\[(k_{n+1}^\vee)_{\Gamma_1} \to (k_n^\vee)_{\Gamma_1}\] 
 is an isomorphism. Now, Lemma \ref{LimCoinv} tells us that
 \[ (k_\infty^\vee)_{\Gamma_1} \cong \varprojlim (k_n^\vee)_{\Gamma_1}.\]
Since the maps in the tower of $\Gamma_1$-coinvariants are all isomorphisms, we conclude that the natural map of $k[\Gamma/\Gamma_1]$-modules
\[(k_\infty^\vee)_{\Gamma_1} \to k_1^\vee\] 
must be an isomorphism. Now $k_1^\vee$ is a cyclic $k_L[\Gamma/\Gamma_1]$-module by Proposition \ref{Tame} and the ideal $J \Omega(\Gamma)$ generated by the augmentation ideal $J$ of $\Omega(\Gamma_1)$ is topologically nilpotent in the sense that $J^n \to 0$ as $n \to \infty$, because $\Gamma_1$ is assumed to be pro-$p$. In this situation we can apply the Nakayama Lemma for compact $\Lambda$-modules --- see \cite[Corollary to Theorem 3]{BalHow} --- to deduce that $k_\infty^\vee$ is a cyclic $\Omega(\Gamma)$-module: any lift of a $k_L[\Gamma/\Gamma_1]$-module generator for $k_1^\vee$ to $k_\infty^\vee$ will generate it as an $\Omega(\Gamma)$-module.
 
Now $o_\infty^\ast / \pi o_\infty^\ast \cong k_\infty^\vee$ by Lemma \ref{oinftydualmodpi}. The $\Lambda(\Gamma)$-module $o_\infty^\ast$ is profinite and $\pi^n \to 0$ as $n \to \infty$ in $\Lambda(\Gamma)$, so applying the Nakayama Lemma again, we conclude that $o_\infty^\ast$ is a cyclic $\Lambda(\Gamma)$-module.

Since $o_\infty^\ast$ is a faithful $\Lambda(\Gamma)$-module by Proposition \ref{oinfdualfful} and since $\Gamma$ is abelian, we deduce that $o_\infty^\ast$ must be a free $\Lambda(\Gamma)$-module of rank $1$.  

$(\Rightarrow)$ We reverse the argument above. Assume $o_\infty^\ast$ is a free $\Lambda(\Gamma)$-module of rank $1$. Then Lemma \ref{oinftydualmodpi} implies that $k_\infty^\vee$ is a free $\Omega(\Gamma)$-module of rank $1$. Hence $(k_\infty^\vee)_{\Gamma_1}$ is a free $k[\Gamma/\Gamma_1]$-module of rank $1$. By Lemma \ref{LimCoinv} we have $(k_\infty^\vee)_{\Gamma_1} \cong \varprojlim (k_n^\vee)_{\Gamma_1}$ and the connecting maps in the tower $(k_n^\vee)_{\Gamma_1}$  are surjective, with the bottom term being $(k_1^\vee)_{\Gamma_1} = k_1^\vee$. Since this is a free $k_L[\Gamma/\Gamma_1]$-module of rank $1$ by Proposition \ref{Tame}, the natural map $(k_\infty^\vee)_{\Gamma_1} \to k_1^\vee$ from the inverse limit to the bottom term is a surjection between two free $k_L[\Gamma/\Gamma_1]$-modules of rank $1$. So it is also an isomorphism. Dualising shows that $k_1 \to k_\infty^{\Gamma_1}$ is an isomorphism as well. \end{proof}

\begin{lemma}\label{OdualNotFreeRank1} In the situation of Proposition \ref{OinftyFreeRk1}, suppose that $o_\infty^\ast$ is a free $\Lambda(\Gamma)$-module of rank $1$. Then $L_n/L$ is tamely ramified for all $n \geq 1$. \end{lemma} 
\begin{proof} Consider the $\Gamma_n$-coinvariants of $o_\infty^\ast$. This must be a free rank $1$ $o_L[\Gamma/\Gamma_n]$-module by assumption. On the other hand, by construction, there's a surjective $o_L[\Gamma/\Gamma_n]$-linear map
\[ (o_\infty^\ast)_{\Gamma_n} \to o_n^\ast\]
(see the remark just before Proposition \ref{oinfdualfful}). Both sides are free $o_L$-modules of rank $[L_n:L]$, so this surjective map must actually be an isomorphism by the rank-nullity theorem. So, $o_n^\ast$ is a free rank $1$ $o_L[\Gamma/\Gamma_n]$-module. But then using, for example \cite[Lemma]{AB}, we see that
\[   o_n = \Hom_{o_L}(o_n^\ast, o_L) = \Hom_{o_L[\Gamma/\Gamma_n]}( o_n^\ast, o_L[\Gamma/\Gamma_n] )\]
must also be a free rank $1$ $o_L[\Gamma/\Gamma_n]$-module. In other words, $o_n$ has an integral normal basis, so by \cite[Proposition 2.1]{Thomas} $L_n/L$ must be tamely ramified.\end{proof}

The following result, which may be of independent interest, shows that the hypothesis that the action map $\rho : \Omega(\Gamma) \to \End_{\Omega(\Gamma)}(k_\infty^\vee)$ is an isomorphism has strong implications about ramification behaviour in the tower $L_\infty/L$.

\begin{lemma} Suppose that in the situation of Proposition \ref{OinftyFreeRk1}, we have $\Gamma_1 = \Gamma$ and that the action map $\rho : \Omega(\Gamma) \to \End_{\Omega(\Gamma)}(k_\infty^\vee)$ is an isomorphism. Then $L_n/L$ is tamely ramified for all $n \geq 1$.
\end{lemma}
\begin{proof} Let $a \in k_\infty^{\Gamma_1}$ and consider the multiplication-by-$a$ map $\ell_a : k_\infty \to k_\infty$. Since $a$ is fixed by $\Gamma = \Gamma_1$, this map is $\Omega(\Gamma)$-linear. By our assumption on $\rho$, we can find some $b \in \Omega(\Gamma)$ such that $\rho(b) = a$. Now $a$ is algebraic over $k_L$ and $\rho$ is injective by assumption, so $b \in \Omega(\Gamma)$ must be algebraic over $k_L$ as well. Since $\Gamma = \Gamma_1$, the mod-$p$ Iwasawa algebra $\Omega(\Gamma)$ is a power series ring over $k_L$ in finitely many variables. The only elements of such a power series ring that are algebraic over $k_L$ are constants. Hence $b \in k_L$ and so $a \in k_L = k_1$ since $\Gamma = \Gamma_1$. Hence $k_\infty^{\Gamma_1} = k_1$. Now the result follows from Theorem \ref{OinftyFreeRk1} and Lemma \ref{OdualNotFreeRank1}. \end{proof} 
Returning to the setting of \S \ref{KatzApps}, we have the following conclusion.
\begin{corollary}\label{PerrinRiou} Suppose that $L = \bQ_{p^2}$ and $\pi = p$, and let $\cG$ be the Lubin-Tate formal group attached to $\pi$. We have $L_\infty = L(\cG[p^\infty])$; let $\Gamma_L^{LT} = \Gal(L_\infty/L)$. Then $o_L \dcroc{Z}^{\psi_q=0}$ is \emph{not} a free $o_L \dcroc{\Gamma_L^{\LT}}$-module of rank $1$.
\end{corollary}
\begin{proof} It is well known that $L_n/L$ is not tamely ramified for any $n \geq 2$. Hence $o_\infty^\ast$ is not a free $\Lambda(\Gamma_L^{LT})$-module of rank $1$ by Lemma \ref{OdualNotFreeRank1}. Since $\cG$ is self-dual, the tower $L_\infty/L$ coincides with the one defined at Definition \ref{DualTowerDef}(1). The result now follows from Theorem \ref{introdualoinfty}.\end{proof}

\subsection{The operator $\psi$ and the span of the $P_n$}
\label{subspan}

We now turn to some consequences of the Katz isomorphism for the span of the $P_n$,
where $P_n$ is the element of $\cC^0_{\Gal}(o_L,o_{\Cp})$ given by 
$a \mapsto P_n(a \cdot \Omega)$. The Katz map $\cK^\ast : \Hom_{o_L}(\cC^0_{\Gal}(o_L,o_{\Cp}), S) 
\to S \dcroc{Z}^{\psiqint}$ is then given by $\mu \mapsto \sum_{n \geq 0} \mu(P_n) Z^n$.

\begin{proposition}
\label{pndense}
The $L$-span of the $P_n$ is dense in the $L$-Banach space $\cC^0_{\Gal}(o_L,\Cp)$.
\end{proposition}

\begin{proof}
Let $W$ denote the closure of the $L$-span of the $P_n$ in $\cC^0_{\Gal}(o_L,\Cp)$. If $W \neq \cC^0_{\Gal}(o_L,\Cp)$, then it has a closed complement in $\cC^0_{\Gal}(o_L,\Cp)$ and we can find a measure $\mu \neq 0$ that is zero on $W$ (and hence on all of the $P_n$). This is a contradiction.
\end{proof}

\begin{remark}
\label{pndbis}
There is another proof of this result. Indeed, locally analytic functions are dense in $\cC^0(o_L,\Cp)$ and for locally analytic functions, we have the generalized Mahler expansion of \cite[Theorem 4.7]{ST}. So it is enough to prove that locally analytic and Gal continuous functions are dense in $\cC^0_{\Gal}(o_L,\Cp)$. A Gal-continuous function is determined by $(f(p^n))_{n=0}^\infty$ where each $f(p^n) \in L_\infty$ and $f(0) \in L$ and $f(p^n) \to f(0)$. We can approximate each $f(p^n)$ by an element of $L_\infty$ and this way, we can show that Gal-continuous locally constant functions are dense in the Gal-continuous functions. More precisely, given a sequence $\{f_n\}$ as above and some $k \geq 0$, we have $f_n - f_\infty \in p^k o_{\Cp}$ for all $n \geq n(k)$, so we replace these $f_n$ by $f_\infty$, and approximate the others to within $p^{-k}$. 
\end{remark}

We now choose a coordinate $X$ on $\LT$ such that $[p]_{\LT}(X)=pX+X^q$. The polynomials $P_i$ depend on the choice of coordinate. However, the $o_L$-module $\oplus_{i=0}^n o_L \cdot P_i$ is independent of the coordinate.  Given this choice of coordinate, we have formulas and estimates for $\psi_q$ in \cite[\S 2A]{FX}.

\begin{lemma}
\label{psinilp}
If $k \geq 1$, then $\psi_q(X^k) \in L[X]_{k-1}$.
\end{lemma}

\begin{proof}
See \cite[Proposition 2.2]{FX}.
\end{proof}

Let $c^0(A)$ denote the set of sequences $\{c_n\}_{n \geq 0}$ with $c_n \in A$ and $c_n \to 0$ ($A=o_L$ or $L$).

\begin{corollary}
\label{pnindep}
The map $c^0(o_L) \to \cC^0_{\Gal}(o_L,o_{\Cp})$ given by $\{c_i\}_{i \geq 0} \mapsto \sum_{i \geq 0} c_i P_i$ is injective, as well as the same map $c^0(L) \to \cC^0_{\Gal}(o_L,\Cp)$.
\end{corollary}

\begin{proof}
Lemma \ref{psinilp} implies that for all $k \geq 0$, there exists $n=n(k)$ such that $p^n X^k \in o_L\dcroc{X}^{\psiqint}$. Let $\mu$ be the corresponding measure. We have $\mu(\sum_{i \geq 0} c_i P_i) = p^n c_k$ hence if $\sum_{i \geq 0} c_i P_i = 0$, then $c_k = 0$. The second assertion follows from the first.
\end{proof}

\begin{lemma}
\label{psiqpk}
If $k \geq 1$, then $\psi_q(p^k \cdot o_L[X]_{q^k}) \subset p^{k-1} \cdot o_L[X]_{q^{k-1}}$.
\end{lemma}
\begin{proof}
This follows from \cite[Proposition 2.2]{FX}.
\end{proof}

Let $H_n \subset L[\Omega]$ denote the set of $P(\Omega)$ such that $\deg P \leq n$ and $P(a \Omega) \in o_{\Cp}$ for all $a \in o_L$. Obviously, $U_n = \oplus_{i=0}^n o_L \cdot P_i(\Omega)  \subset H_n$. Let $\mu_i : \cC^0_{\Gal}(o_L,o_{\Cp}) \to L$ be the measure corresponding to $X^i$, so that $\mu_i(P_j) = \delta_{ij}$.

\begin{proposition}
\label{hupk}
If $Q(\Omega) = \sum_{i=0}^n c_i P_i(\Omega) \in H_n$, then $c_i \in p^{-m} o_L$ if $i \leq q^m$.
\end{proposition}

\begin{proof}
We have $Q(\Omega) \in \cC^0_{\Gal}(o_L,o_{\Cp})$. 
By Lemma \ref{psiqpk}, $p^m X^i \in o_L\dcroc{X}^{\psiqint}$ if $i \leq q^m$, and hence $p^m \mu_i  \in \Hom_{o_L}(\cC^0_{\Gal}(o_L,o_{\Cp}),o_L)$ for all $0 \leq i \leq q^m$. Hence $p^m c_i \in o_L$.
\end{proof}

\begin{corollary}
\label{hupkcoro}
We have $H_{q^k} \subset p^{-k} U_{q^k}$.
\end{corollary}

Let $\psi_p = p \cdot \psi_q$ so that $\psi_p( o_L\dcroc{X} ) \subset o_L\dcroc{X}$.

\begin{lemma}
\label{psipmodp}$\psi_p(X^{qk+(q-1)}) = X^k \bmod{p}$ and $\psi_p(X^m)=0 \bmod{p}$ if $m \neq -1 \bmod{q}$.
\end{lemma}

\begin{proof}
This follows from \cite[Proposition 2.2]{FX}.
\end{proof}

\begin{corollary}
\label{notbij}
The map $c^0(L) \to \cC^0_{\Gal}(o_L,\Cp)$ is not surjective.
\end{corollary}

\begin{proof}
By Corollary \ref{pnindep}, it is injective. If it is a bijection, then the continuous dual of $\cC^0_{\Gal}(o_L,\Cp)$ is naturally isomorphic to $o_L\dcroc{X} [1/p]$ via the map $\mu \mapsto \sum_{n \geq 0} \mu(P_n) X^n$. However by the Katz isomorphism, the image of this map is $o_L\dcroc{X}^{\psiqint} [1/p]$.

Take $f(X) = 1 + X^{q-1} + X^{q^2-1} + \cdots$. Lemma \ref{psipmodp} implies that $\psi_p(f) = f \bmod{p}$ and hence $\psi_p^n(f)  = f \bmod{p}$. We therefore have $\psi_q^n(f) \in p^{-n} f + p^{-(n-1)} o_L\dcroc{X}$ for all $n \geq 1$, so that $f(X)$ is not in $o_L\dcroc{X}^{\psiqint} [1/p]$. Hence $o_L\dcroc{X} [1/p] \neq o_L\dcroc{X}^{\psiqint} [1/p]$.
\end{proof}

In order to say more using Katz' result, we need more elements of $o_L\dcroc{X}^{\psiqint}$. There is $o_L\dcroc{X}^{\psi_q=0}$, which contains $X^i$ for $1 \leq i \leq q-2$ and $pX^{q-1}+(q-1)$ and hence $(\oplus_{i=1}^{q-2} X^i \cdot \varphi_q(o_L\dcroc{X})) \oplus (pX^{q-1}+(q-1)) \cdot \varphi_q(o_L\dcroc{X})$. If $f_n(X) \in (X \cdot o_L\dcroc{X})^{\psiqint}$ and the $b_n$ are in $o_L$, then $\sum_{n \geq 0} b_n \varphi_q^n(f_n) \in o_L\dcroc{X}^{\psiqint}$ as well (the sum converges for the weak topology, and $\psi_q$ is continuous for that topology).  For example, if $f(X) \in (X \cdot o_L\dcroc{X})^{\psi_q=0}$, then $\sum_{n \geq 0} \varphi_q^n(f) \in o_L\dcroc{X}^{\psi_q=1}$.

\begin{remark}
\label{formpsi}
We have 
\begin{enumerate}
\item $\psi_q(X^i)=0$ if $1 \leq i \leq q-2$ and $q+1 \leq i \leq 2q-3$ and $2q+1 \leq i \leq 3q-4$
\item $\psi_q(1)=1$ and $\psi_q(X^{q-1})=(1-q)/p$ and $\psi_q(X^q)=X$
\item $\psi_q(X^{2q-2})=q-1$ and $\psi_q(X^{2q-1})=X(1/p-2p)$ and $\psi_q(X^{2q})=X^2$
\item More generally, $\psi_q(X^k)=X\psi_q(X^{k-q})-p \psi_q(X^{k+1-q})$
\end{enumerate}
\end{remark}

\begin{lemma}
\label{expsin}
We have $p^k X^{q^k-1} \in o_L\dcroc{X}^{\psiqint}$, but not $p^{k-1} X^{q^k-1}$.
\end{lemma}

\begin{proof}
Recall that $\psi_q(X^{q-1})=(1-q)/p$. This implies that $\psi_q(1/X) = \psi_q((X^{q-1}+p)/\varphi_q(X)) = 1/pX$. If $k \geq 1$, then \[ \binom{q^{k-1}}{i} \cdot p^i = \binom{q^{k-1}-1}{i-1} \cdot q^{k-1} p^i / i \in p^k o_L. \] 
This implies that $\varphi_q(X^{q^{k-1}}) \in X^{q^k} + p^k X o_L[X]_{q^k-1}$. By Lemma \ref{psiqpk}, we have 
\[ \psi_q(X^{q^k-1}) = \psi_q\left(\frac{\varphi_q(X^{q^{k-1}}) + X^{q^k} - \varphi_q(X^{q^{k-1}})}{X} \right) \in \frac{X^{q^{k-1}-1}}{p} + o_L\dcroc{X}^{\psiqint}. \]
This implies the Lemma by induction on $k$.
\end{proof}

\begin{corollary}
\label{hnotu}
There is an $h \in H$ in which the coefficient of $P_{q^k-1}$ is in $p^{-k} o_L^\times$.
\end{corollary}

\begin{proof}
Let $c_{q^k-1} \in \cC^0_{\Gal}(o_L,\Cp)^*$ be the linear form corresponding to $X^{q^k-1}$. There is an $f \in \cC^0_{\Gal}(o_L,o_{\Cp})$ such that $c_{q^k-1}(f) \in p^{-k} o_L^\times$ (if it was in $p^{1-k} o_L$ for all $f$, then $p^{k-1} c_{q^k-1}$ would be an integral linear form, and we'd have $p^{k-1} X^{q^k-1} \in o_L\dcroc{X}^{\psiqint}$. This is not the case by lemma \ref{expsin}). By Corollary \ref{pndense}, the $L$-span of the $P_n$ is dense in $\cC^0_{\Gal}(o_L,\Cp)$. Therefore there is an $h \in H$ such that $\| f - h \| \leq p^{-1}$. We then have $c_{q^k-1}(h) \in p^{-k} o_L^\times$.
\end{proof}

\section{Other criteria}
\label{secother}

We indicate how to prove Theorems \ref{introderivcrit} and \ref{introfinext}. 

\subsection{The Lubin-Tate derivative}

As we said in the Introduction, Theorem \ref{introderivcrit} follows from Theorem \ref{intropsicrit} and Proposition \ref{psikatz} below.

\begin{lemma}
\label{sumomeg} The sum $\sum_{[p](\omega)=0} \omega^n$ is $q$ if $n=0$, it is $0$ if $(q-1) \nmid n$, and it is $(q-1)(-p)^k$ if $n=(q-1)k$ with $k \geq 1$.
\end{lemma}

\begin{proof}
Since $[p](T)=p T + T^q$, the sum is over $0$ and the roots of $T^{q-1} = -p$. If $\lambda$ is one of the roots, the set of all the roots is $\{ \eta \lambda \}_{\eta^{q-1}=1}$. The result follows (for $n=0$ it is a convention). 
\end{proof}

\begin{proposition}
\label{psikatz}
Assume that $L=\bQ_{p^2}$ and that $\pi=p$. Let $\lambda= \Omega^{q-1} / p(q-1)! \in o_{\Cp}^\times$.

If $f(Z) \in o_{\Cp}\dcroc{Z}$, then $\varphi \psi_q(f) - \lambda \cdot D^{q-1}(f) \in o_{\Cp}\dcroc{Z}$.
\end{proposition}

\begin{proof}
Recall from \cite[p. 667]{Ka2} that $f(Z \oplus Y) = \sum_{n \geq 0} Y^n P_n(\partial) f(Z)$. We have $\varphi \psi_q(f)(Z) = 1/q \cdot \sum_{[p](\omega)=0} f(Z \oplus \omega)$, so that 
\[ \varphi \psi_q(f)(Z) = \frac 1q \sum_{[p](\omega)=0} \sum_{n \geq 0}  \omega^n P_n(\partial) f(Z) = \frac 1q \sum_{n \geq 0} \left( \sum_{[p](\omega)=0} \omega^n \right) P_n(\partial) f(Z). \]
By Lemma \ref{sumomeg}, the $\sum \omega^n$ for $n$ not divisible by $q-1$ are zero, and  the $\sum \omega^n$ for $n=(q-1)k$ are divisible by $q$ except when $k=1$. Hence
\[ \varphi \psi_q(f) - \frac 1q (q-1) (-p) P_{q-1}(\partial)(f) \in o_{\Cp}\dcroc{Z}. \]
The proposition now follows from the fact that \[ P_{q-1}(\partial) = \frac{\partial^{q-1}}{(q-1)!} = pD^{q-1} \cdot \frac{\Omega^{q-1}}{p(q-1)!} = pD^{q-1} \cdot \lambda. \qedhere \]
\end{proof}

\subsection{Changing the base field}

We now turn to Theorem \ref{introfinext}. If $K$ is a subfield of $L$, we also have a character variety $\frX$ for $K$; write $\frX_K$ and $\frX_L$. An $L$-analytic character $\eta : o_L \to \Cp^\times$ can be restricted to $o_K$, and it is then $K$-analytic. This gives a rigid analytic map $\frX_L \to \frX_K$. This map in turn gives rise to a map $\res_{L/K} : \OO_{\Cp}(\frX_K) \to \OO_{\Cp}(\frX_L)$, which sends bounded functions to bounded functions, and $\OO_M(\frX_K)$ to $\OO_M(\frX_L)$ for all closed subfields $L \subset M \subset \Cp$.

\begin{lemma}
\label{resinj}
On bounded functions, $\res_{L/K} : \OO_{\Cp}^b(\frX_K) \to \OO_{\Cp}^b(\frX_L)$ is injective.
\end{lemma}

\begin{proof}
Suppose that $f \in \OO_{\Cp}^b(\frX_K)$ is zero on the restriction to $o_K$ of every $L$-analytic character of $o_L$. Since $o_K$ is a direct summand of $o_L$, every torsion character of $o_K$ extends to a torsion character of $o_L$. Hence $f$ is zero on all torsion characters of $o_K$. This implies that $f=0$ as $f$ is bounded.
\end{proof}

If $\mu$ is a distribution on $o_K$,  we define a distribution $\res_{L/K}(\mu)$  on $o_L$ as follows: if $f \in \cC^{an}(o_L)$, we let $\res_{L/K}(\mu)(f) = \mu(f {\mid}_{o_K})$. This is compatible with the above map if we view elements of $\OO_{\Cp}(\frX)$ as distributions.

\begin{lemma}
\label{resdist}
If $\mu$ is a distribution on $o_K$, whose image under $\res_{L/K}(\mu)$ is a measure on $o_L$, then there exists a measure $\tilde{\mu}$ on $o_K$ such that $\mu = \tilde{\mu}$ on $\LC(o_K)$.
\end{lemma}

\begin{proof}
Let $f$ be a locally constant function on $o_K$. Since $o_K$ is a direct summand in $o_L$, we can extend $f$ to a locally constant function $\tilde{f}$ on $o_L$, in a way that the sup norm of $\tilde{f}$ on $o_L$ is the sup norm of $f$ on $o_K$. Since $\res_{L/K}(\mu)$ is a measure, there exists $C$ such that $\| \res_{L/K}(\mu)(g) \|_{o_L} \leq C \cdot \| g \|_{o_L}$ for all locally constant functions $g$ on $o_L$. We then have \[ \| \mu(f) \|_{o_K} =  \| {\res_{L/K}(\mu)(\tilde{f})} \|_{o_L}  \leq C \cdot \| \tilde{f} \|_{o_L} = C \cdot \| f \|_{o_K}. \] 
We can now let $\tilde{\mu}(f) = \mu(f)$ for any $f \in \LC(o_K)$. The above estimate shows that $\tilde{\mu}$ extends continuously to $\cC^0(o_K)$.
\end{proof}

\begin{proposition}
\label{ltok}
If $\OO_L^b(\frX_L) = L \otimes_{o_L} \Lambda(o_L)$, then $\OO_L^b(\frX_K) = L \otimes_{o_K} \Lambda(o_K)$.
\end{proposition}

\begin{proof}
If $\mu \in \OO_L^b(\frX_K)$, then $\mu$ can be seen as a distribution on $o_K$, and it gives rise via $\res_{L/K}$ to an element of $L \otimes_{o_L} \Lambda(o_L)$. By Lemma \ref{resdist}, there is a measure $\tilde{\mu}$ on $o_K$ such that $\mu = \tilde{\mu}$ on $\LC(o_K)$. The image of the distribution $\mu-\tilde{\mu}$ under $\res_{L/K}$ belongs to $L \otimes_{o_L} \Lambda(o_L)$ and is zero on locally constant functions, hence $\res_{L/K}(\mu-\tilde{\mu} ) = 0$. By Lemma \ref{resinj}, $\mu=\tilde{\mu}$ and hence $\mu$ is a measure on $o_K$.
\end{proof}

\begin{theorem}
\label{finext}
If $K/L$ is finite and if $\Lambda_K(\frX_K) = o_K \dcroc{o_K}$, then $\Lambda_L(\frX_L) = o_L \dcroc{o_L}$.
\end{theorem}

% code-appendix
%---
%\pagebreak
\appendix
\label{Appendix}
\section{An algorithm for whether the $\sigma_{i, j}$'s span $\Int(o_L, o_L)$}
%\input{bddappendix.tex}
%---

\newcommand{\und}{\underline}
\newcommand{\ui}{\underline{i}}
\newcommand{\uj}{\underline{j}}
\newcommand{\Capp}{\operatorname{Cap}}

\theoremstyle{definition}
\newtheorem*{definition*}{Definition}

\begin{center}
    \textsc{Drago\cb{s} Cri\cb{s}an and Jingjie Yang}
\end{center}
%\date{Summer 2021}
%\maketitle

\stoptocwriting
\subsection*{Acknowledgements}
The authors took part in an Oxford summer project together with Shashidhara Balla, Jonathan Medcalf and William Whitehead. Financial support from the LMS, Brasenose College, University College, Merton College and the Mathematical Institute, Oxford is gratefully acknowledged.

\subsection{Introduction}
Let $\mathbb{Q}_p \subseteq L \subsetneq \mathbb{C}_p$ be a field of finite degree $d$ over $\mathbb{Q}_p$, 
$o_L$ the ring of integers of $L$,
$\pi \in o_L$ a fixed prime element,
and $q := |o_L / \pi_L o_L|$ the dimension of the residue field.

For an $o_L$-submodule $S$ of $L[Y]$ and an integer $n$, let $S_n = \{ f \in S : \deg(f) < n\}$. 

Recall that the polynomials $P_n(Y)$ are defined by 
$$
    \exp(Y \cdot \log_{\text{LT}}(Z) ) = \sum_{n=0}^{\infty} P_n(Y) Z^n.
$$

We will choose the coordinate $Z$ such that $\log_{\text{LT}}(Z) = \sum_{k=0}^{\infty} \pi^{-k} Z^{q^k}$. 

Define the upper-triangular matrix $\left( \sigma_{i,j} \right)_{i,j\geq 0}$ with entries in $L[Y]$ by
$$
    P_j(Ys) = \sum_{i=0}^j \sigma_{i,j}(Y) P_i(s).
$$

By Lemmas \ref{PolySigmas} and \ref{RhoDegJ}, we know that $\sigma_{i,j}(Y) \in \Int(o_L,o_L)$ and that $\deg(\sigma_{i,j}(Y)) \leq j$. The question is whether the $o_L$-linear span of $\left\{ \sigma_{i,j}(Y) : 0\leq i \leq j \right\}$ equals $\Int(o_L,o_L)$. In this write-up we develop an algorithm to check whether $\bigl( \Int(o_L,o_L) \bigr)_n$ is contained in the $o_L$-linear span of $\left\{ \sigma_{i,j}(Y) : 0\leq i \leq j < N \right\}$ for some fixed $N$, where for convenience we require $q-1 \mid N$.

\subsection{Theory}
\subsubsection{Reduction to $\tau_{i,j}^{(a)}$}
To ease notation, for a fixed $a\in \left\{0,1,\ldots, q-2 \right\}$, we denote $\ui=a+(q-1)i$.

By Proposition \ref{OldToNew}(2), there exist upper-triangular matrices $\tau_{i,j}^{(a)}(Y)$ such that
\begin{align}
    \sigma_{\ui,\uj}(Y) = Y^a \cdot \tau_{i,j}^{(a)}(Y^{q-1}).
    \label{OldToNew2}
\end{align}

\begin{definition}
For a polynomial $P(x)$, we denote by $\gamma_n(P)$ the coefficient of $x^n$ in $P$.
\end{definition}

\begin{definition}
Let $M$ be the $o_L$-linear span of $\left\{ \sigma_{i,j}(Y) : 0\leq i \leq j \right\}$. For a fixed $a$, let $M^{(a)}$ be the $o_L$-linear span of $\left\{ \sigma_{\ui,\uj}(Y) : 0\leq i \leq j \right\}$. Let $S^{(a)}$ be the $o_L$-linear span of $\left\{ \tau_{i,j}^{(a)}(Y) : 0\leq i \leq j \right\}$.
\end{definition}

\begin{lemma}\label{CheckSpanThroughCoeff}
    Let $(f_b^{(a)})_{b \geq 0}$ be a regular basis for $S^{(a)}$ --- that is, each $f_b^{(a)}$ has degree $b$. 
    Then, $M=\Int(o_L,o_L)$ if and only if for all $a \in \{0,1, \ldots q-2\}$ and $b\geq 0$, we have $$\nu_{\pi} (\gamma_b ( f_b^{(a)})) = -w_q(a+b(q-1)).$$
\end{lemma}

\begin{proof}
For a fixed $a\in \{ 0,1, \ldots q-2\}$, by (\ref{OldToNew2}), we have $\gamma_{s} (\sigma_{i,j}(Y)) = 0$ if $s \not\equiv j \pmod{q-1}$. So, by definition, $M=\bigoplus_{a=0}^{q-2} M^{(a)}$.

We write $S^{(a)} (Y^{q-1}) = \{ f(Y^{q-1}) : f \in S^{(a)} \}$. Equation~(\ref{OldToNew2}) shows that
$$
    M^{(a)} = Y^a \cdot N^{(a)} (Y^{q-1}).
$$

Having chosen a regular basis $(f_b^{(a)})_{b \geq 0}$, these give regular bases $\left( f_b^{(a)} (Y^{q-1}) \right)_{b\geq 0}$ for $S^{(a)} (Y^{q-1})$.
 
So, we get regular bases $\left(Y^a f_b^{(a)} (Y^{q-1}) \right)_{b\geq 0}$ for $M^{(a)}$ and thus a regular basis $\{Y^a f_b^{(a)} (Y^{q-1}) : a\in \{ 0,1, \ldots q-2\}, b \geq 0 \}$ for $M$.

Then, $M=\Int(o_L,o_L)$ is equivalent to $\nu_{\pi} (\gamma_{a+b(q-1)}(Y^a f_b^{(a)} (Y^{q-1}))) = -w_q(a+b(q-1))$, which is equivalent to $\nu_{\pi} (\gamma_b ( f_b^{(a)})) = -w_q(a+b(q-1))$.
\end{proof}

Let $n=a+b(q-1)$, where $a,b$ are integers, with $a \in \{0,1, \ldots q-2\}$. The proof above shows that a polynomial of degree $n$ with $\pi$-valuation of leading term equal to $-w_q(n)$ exists in $M_N$ if and only a polynomial of degree $b$ with the same valuation of leading term exists in $S^{(a)}_{N / (q-1)}$. So, the strategy will be to compute regular bases for $S^{(a)}_{N / (q-1)}$.

\subsubsection{A formula for $\tau_{i,j}^{(a)}$}\label{TauFormula}

One advantage of this approach is that the matrices $\tau^{(a)}_{i,j}(Y)$ can be computed quickly. Recall Definition \ref{DefOfQ} (where we merely change notation, calling $m$ by $a$ instead):

\begin{definition}
For each $j\geq i \geq 0$, let
\begin{align*}
    Q_a(i,j) &:= \left\{ \textbf{k} \in \mathbb{N}^{\infty} : \sum_{\ell=0}^{\infty} k_\ell = \ui, \sum_{\ell=1}^{\infty} k_\ell \left( \frac{q^\ell-1}{q-1} \right) = j-i \right\}; \\
    r_{i,j}^{(a)} &:= \sum_{\textbf{k} \in Q_a(i,j)} \binom{\ui}{k_0;k_1; \ldots} \cdot \pi^{-\sum_{\ell=1}^{\infty} \ell \cdot k_\ell}.
\end{align*}
\end{definition}
Define the upper triangular matrix $(D_{i,j})_{i,j}$ of coefficients as follows:
\begin{definition}
Let $D_{i,j} = i! \gamma_{i} P_j(Y)$.
\end{definition}
This does not depend on $a$. From Proposition \ref{derivpm}, we obtain the following recursion formula, valid for $i\geq 1$:
$$
    D_{i,j} = \sum_{r \geq 0} \pi^{-r} D_{i-1, j-q^r},
$$
with the initial conditions being $D_{0,j} = \delta_{0,j}$.

Now, by Proposition \ref{CoeffOfPn}(2) it follows that $r_{i,j}^{(a)} = D_{\ui,\uj}$. 
To tie this back to $\tau_{i, j}^{(a)}$, we recall from Definition \ref{rTauDX}(3) the notation $\mathcal{D}_Y := \text{diag}(1,Y,Y^2, \ldots)$. Then, Lemma \ref{FirstMeq} gives $\tau^{(a)} = (r^{(a)})^{-1} \cdot \mathcal{D}_Y \cdot r^{(a)}$. This gives a fast algorithm to compute the matrices $\tau^{(a)}$, as the recurrence relation for $D$ allows us to compute $r^{(a)}$ easily.

\subsubsection{Gaussian elimination over a (discrete) valuation ring}
Let $R$ be a (discrete) valuation ring and let $A$ be an $m \times n$ matrix with entries in $R$. We define notions of elementary row operations and row echelon form over $R$, similarly to the definitions over a field.

\begin{definition}
    Given a matrix $A$ as above, the elementary row operations are as follows.
    \begin{enumerate}
        \item Swap two rows.
        \item Multiply an entire row by a unit in $R$.
        \item Add an $R$-multiple of a row to another row.
    \end{enumerate}
\end{definition}

\begin{lemma}\label{row-op-preserves-span}
    Performing elementary row operations on a matrix preserves its $R$-row span. 
\end{lemma}

\begin{proof}
    For each elementary row operation on $A$, we define an $m \times m$ matrix $B$ with entries in $R$ such that the result of applying the elementary row operation on $A$ is $B A$. Observe that in each case, $B$ is invertible, so $BA$ has the same $R$-row span as $A$.
\end{proof}

\begin{lemma}[Gaussian Elimination]
    Let $A$ be a matrix as above. Assume that $m\geq n$ and that $A$ has rank $n$. Then, one can perform a sequence of elementary row operations on $A$ to produce an upper-triangular matrix of rank $n$.
    \label{gaussian_elimination}
\end{lemma}

\begin{proof}
    We will exhibit an algorithm that puts $A$ in the required form. 
    
    We start with the leftmost column. As $A$ has rank $n$, there is a non-zero entry on column $1$. Pick the one with minimal valuation and swap rows, so that the entry on column $0$ with minimal valuation is on position $(0,0)$. Let the new matrix be $B$.
    
    Then, for each row $i \geq 1$, subtract $\frac{b_{i0}}{b_{00}} \times \text{(row $0$)}$ from row $i$. After all of these operations, the matrix has block form:
    \begin{align*}
        \left[ 
        \begin{array}{c|c} 
            b_{00} & * \\ 
            \hline 
            0 & A' 
        \end{array} 
        \right]
    \end{align*}
    where $*$ denotes some $1 \times (n-1)$ matrix, and $A'$ is an $(m-1)\times (n-1)$ matrix. Observe that, as $A$ had rank $n$ and the elementary row operations don't change the rank, $A'$ will have rank $n-1$.
    
    Now, we can inductively apply the same procedure to $A'$. Observe that all row operations on $A'$ extend to row operations on the whole matrix that don't change the block structure (as the corresponding entries in the first column are all $0$'s). By construction, the end result is an upper-triangular matrix, which has the same rank as the initial matrix $A$.
\end{proof}

\begin{comment}
Now, we use elementary row operations to put $A$ in row echelon form. The algorithm is somewhat similar to the classical Gaussian elimination over a field. We have some rows, at the top of the matrix, which are already good (i.e. they are part of the final row echelon form). All the other rows will be bad. At the beginning, no row is good. We will only operate on bad rows, so we call an entry bad if it is on an bad row.

For each (column) $k$ from $1$ to $m$:
\begin{enumerate}
    \item If all bad entries in column $k$ are $0$, completely ignore the column.
    \item If there are non-zero bad entries in column $k$, find the element of minimum valuation among them. Let it be on row $i$.
    \item Swap row $i$ and the first bad row $i_0$.
    \item Subtract the correct multiple of row $i_0$ from all the other bad rows, such that all but the $i_0$-th bad entries in column $k$ are $0$ (this is possible, as the entry on line $i_0$ has minimum valuation among the other bad elements on column $k$).
    \item Mark row $
    i_0$ as good.
\end{enumerate}
\end{comment}

\subsection{Implementation}
We focus on the totally ramified extension $L = \mathbb{Q}_p(p^{1/d})$ and the unramified extension of degree~$d$,
where we take the prime $p$, the degree $d$, and the cutoff $N$ as input parameters.

Fix $a \in \{0, 1, \dots, q - 2\}$.
Firstly, we compute the matrices $(\tau^{(a)})_{0 \leq  i \leq j < N/(q-1)}$ following the method discussed in Section~\ref{TauFormula}.
Then, for $s = 0, \dots, N/(q-1) - 1$, we will appeal to the following result to inductively compute a basis $(g^{(a), s}_b)_{0 \leq b \leq s}$ for the $o_L$-span of $\{ \tau_{i, j}^{(a)} : 0 \leq i \leq j \leq s \}$, with each $g^{(a), s}_b$ having degree $b$.

\begin{proposition}
    Fix $s \geq 0$, and let $(g^{(a), s-1}_b)_{0 \leq b \leq s-1}$ be a basis for the $o_L$-span of $\{ \tau_{i, j}^{(a)} : 0 \leq i \leq j \leq s-1 \}$ such that each $g^{(a), s-1}_b$ has degree $b$.
    
    Record the coefficients of these polynomials $g^{(a), s-1}_*$ in $s$ row vectors, and append $s + 1$ new row vectors obtained from the coefficients of $\tau_{*, s}^{(a)}$ to obtain the $(2s + 1) \times (s + 1)$ matrix
    $$
    B := 
    \begin{blockarray}{ccccc}
    Y^s & Y^{s-1} &  & 1\\
    \begin{block}{(cccc)c}
        \bullet & * & \cdots & * & \tau_{s, s}^{(a)} \\
         & \bullet & \cdots & * & g_{s-1}^{(a), s-1}\\
         & & \ddots & \vdots \\
         & & & \bullet & g_0^{(a), s-1}\\
        * & * & \cdots & * & \tau_{0, s}^{(a)} \\
        \vdots & \vdots &  & \vdots \\
        * & * & \cdots & * & \tau_{s-1, s}^{(a)}\\
    \end{block}
    \end{blockarray}
    $$
    with coefficients in $L$.
    The $\bullet$'s are non-zero (where $B_{s, 0} \neq 0$ because $\sigma_{\underline{s}, \underline{s}} = Y^{\underline{s}}$ by Lemma \ref{PolySigmas} which by Equation~\ref{OldToNew2} implies that $\tau_{s, s}^{(a)} = Y^{s}$), so $B$ has rank $s+1$.
    
    Bring the full-rank matrix $B$ to upper-triangular form $B'$ using Gaussian elimination over the discrete valuation ring $o_L$ as per Lemma~\ref{gaussian_elimination}.
    Then 
    \begin{enumerate}[(i)]
        \item we can define the new polynomials $g_s^{(a), s}, g_{s-1}^{(a), s}, \dots, g_{0}^{(a), s}$ by reading off the first $s + 1$ rows of $B'$,
        so that each $g^{(a), s}_b$ has degree $b$ and $(g^{(a), s}_b)_{0 \leq b \leq s}$ form a basis for the $o_L$-span of $\{ \tau_{i, j}^{(a)} : 0 \leq i \leq j \leq s \}$;
        \item for each $b = 0, \dots, s-1$, the $\pi$-adic valuation of the leading coefficient in the new polynomial $g_b^{(a), s}$ is at most that of the old polynomial $g_b^{(a), s-1}$.
    \end{enumerate}
\end{proposition}
\begin{proof}
    By Lemma~\ref{gaussian_elimination} the upper-triangular matrix $B'$ still has rank $s+1$, so it has only non-zero elements on its main diagonal.
    Hence for each $b = 0, 1, \dots, s$, the polynomial $g_{b}^{(a), s}$ obtained by reading off the $b$-th row has degree $b$.
    Then of course these polynomials are linearly independent.
    Also they are the only non-zero rows in $B'$, so by Lemma~\ref{row-op-preserves-span} their $o_L$-span is the same as that of the rows of $B$, which by construction is precisely the $o_L$-span of $\{\tau_{i, j}^{(a)} : 0 \leq 1 \leq j \leq s\}$, giving (i).
    
    Now fix $0 \leq b \leq s-1$, and consider what happens to the $b$-th column when we reduce $B$ to $B'$.
    Observe that in the proof of Lemma~\ref{gaussian_elimination}, when we operate on the $j$-th column for $j = 0, \dots, s-b-1$, as the row for $g_b^{(a), s-1}$ has a $0$ entry in the $j$-th column,
    it is neither chosen to be the pivot row nor altered as we subtract off multiples of the pivot row.
    Thus when we operate on the $(s-b)$-th column to determine the $(s-b)$-th row and column of $B'$, the leading coefficient of $g_b^{(a), s-1}$ must be a candidate for the pivot.
    But the pivot $B'_{s-b, s-b}$ is chosen to have minimal valuation, so $\nu_\pi(\gamma_b(g_b^{(a), s-1})) \geq \nu_\pi(B'_{s-b, s-b})$.
    Now $B'_{s-b, s-b} = \gamma_b(g_b^{(a), s})$ by definition, giving (ii).
\end{proof}

For $b$ fixed, it follows that 
$$
\nu_\pi(\gamma_b(g_b^{(a), s})), \quad s = b, b+1, \dots
$$ 
is a non-increasing sequence. Moreover, as $g_b^{(a), s} \in S^{(a)}$ can be written as an $o_L$-linear combination of the $f_i^{(a)}$'s and each $f_i^{(a)}$ is of degree $i$, we must have $g_b^{(a), s} = \sum_{0 \leq i \leq b} \lambda_i f_i^{(a)}$ for some $\lambda_i \in o_L$;
by looking at the leading coefficient, it follows that $$\nu_\pi(\gamma_b(g_b^{(a), s})) \geq \nu_\pi(\gamma_b(f_b^{(a)})) \geq -w_q(a + b(q-1)).$$
These observations motivate us to look at the following

\begin{definition}
    For $n = a + b(q - 1)$, let $s_0(n)$ be the minimal $s \geq b$ such that $(g_b^{(a), s})_{0 \leq b \leq s}$ satisfies $\nu_\pi(\gamma_b(g_b^{(a), s})) = -w_q(n)$, if such $s$ exists;
    otherwise set $s_0(n) = \infty$.
    \label{s_0_def}
\end{definition}

Then whenever $s \geq s_0(n)$ in the computations, we can immediately conclude that the equality $\nu_\pi(\gamma_b (f_b^{(a)})) = -w_q(a+b(q-1))$ in Lemma~\ref{CheckSpanThroughCoeff} holds for this $n = a + b(q - 1)$.

We may thus make a small optimisation: at any stage $s$, if $s \geq s_0(a + b(q-1))$ for all $0 \leq b < d$
then we can just drop the last $d$ columns when carrying out Gaussian elimination.
Indeed for all $s' > s$ it is unnecessary to compute $(g_b^{(a), s'})_{0 \leq b < d}$ as the $\pi$-adic valuation of each leading term has already hit the desired minimum,
and to compute the leading terms of $(g_b^{(a), s'})_{d \leq b \leq s'}$ we do not need the lower-order terms in the last $d$ columns.

\subsection{Data}
\begin{figure}[H]
\makebox[\textwidth][c]{\includegraphics[width=1.2\textwidth]{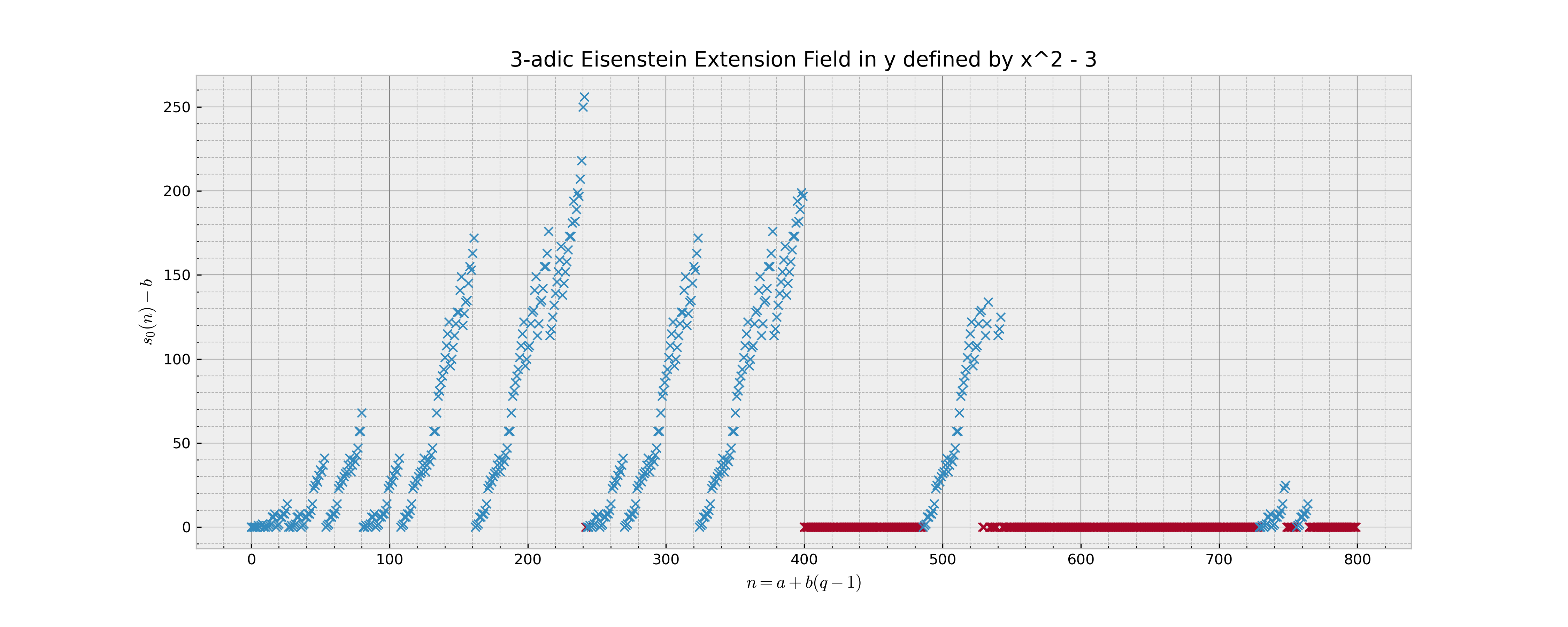}}
\caption{\texttt{extension = "3,2,800,ram"} --- $s_0(n)$ in the quadratic ramified extension $\mathbb{Q}_3(\sqrt 3)$ for $n < 800$. Red points are the $n$'s for which $s_0(n) \geq 800$.}\label{3-2-ram}
\end{figure}

For reference, the computations in Figure~\ref{3-2-ram} took
\begin{itemize}
    \item $227.04$ seconds for $D$;
    \item $616.45$ seconds for $\tau^{(0)}$ and $616.43$ seconds for $\tau^{(1)}$;
    \item $0.20$ seconds for $s=50$, $1.89$ seconds for $s=100$, $6.15$ seconds for $s=150$, $12.09$ seconds for $s=200$, etc.\ for $a = 0$, and slightly less for $a = 1$.
\end{itemize}

We see that $s_0(n) - b$ seems to depend on the $p$-adic digits of $n$; we only managed to prove a special case of this pattern, which we will discuss below.
Nonetheless, the data do suggest that $s_0(n)$ is finite for every $n$ and hence that $\Int(o_L, o_L)$ is spanned by the $\sigma_{i, j}$'s as an $o_L$-module.

A similar pattern emerges for larger $p$ and unramified extensions: see Figures~\ref{17-2-ram} and \ref{5-3-unram} below.

More data and plots can be found \href{https://github.com/Team-Konstantin/Bounded-Functions-on-Character-Varieties/tree/writeup}{on our GitHub repository}.

\begin{figure}[H]
\makebox[\textwidth][c]{\includegraphics[width=1.2\textwidth]{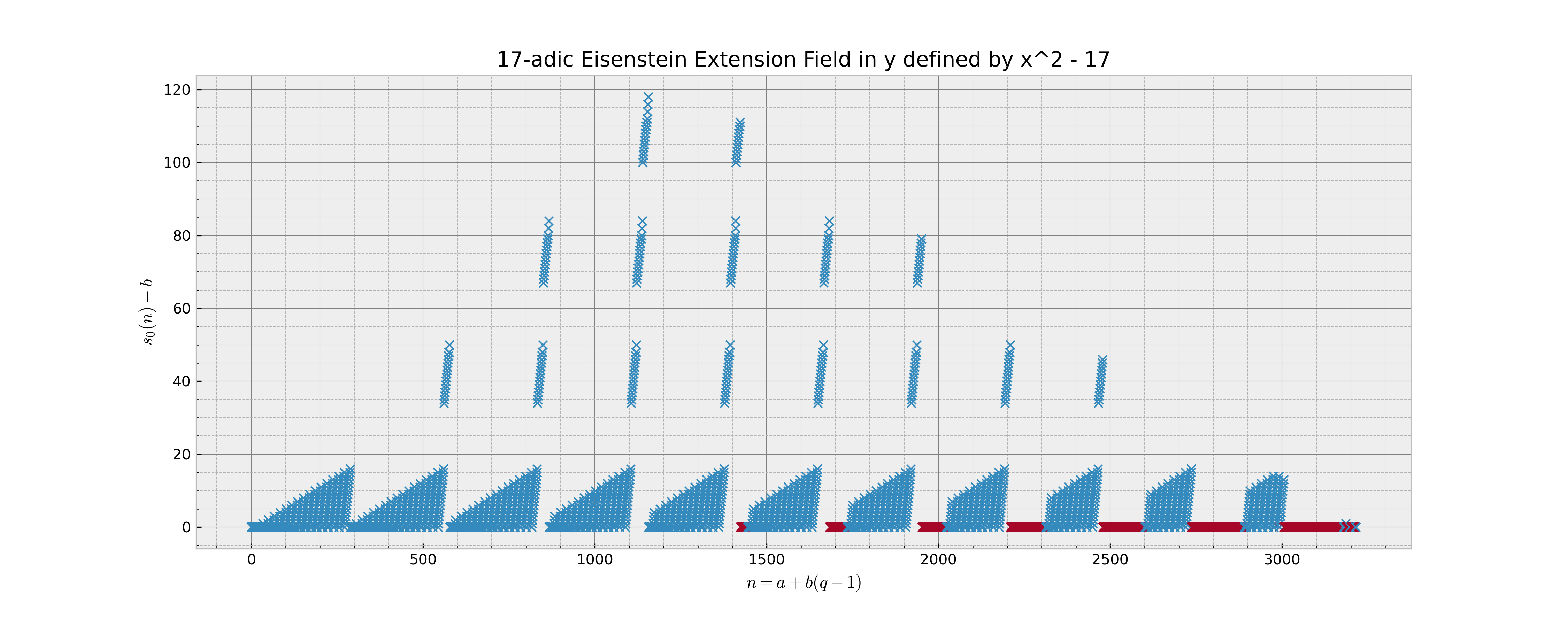}}
\caption{\texttt{extension = "17,2,3216,ram"} --- $s_0(n)$ in the quadratic ramified extension $\mathbb{Q}_{17}(\sqrt{17})$ for $n < 3216$. Note that red points are the $n$'s for which $s_0(n) \geq 3216$ --- not enough computation was done to unveil the pattern for the larger $n$'s!}
\label{17-2-ram}
\end{figure}

\begin{figure}[H]
\makebox[\textwidth][c]{\includegraphics[width=1.2\textwidth]{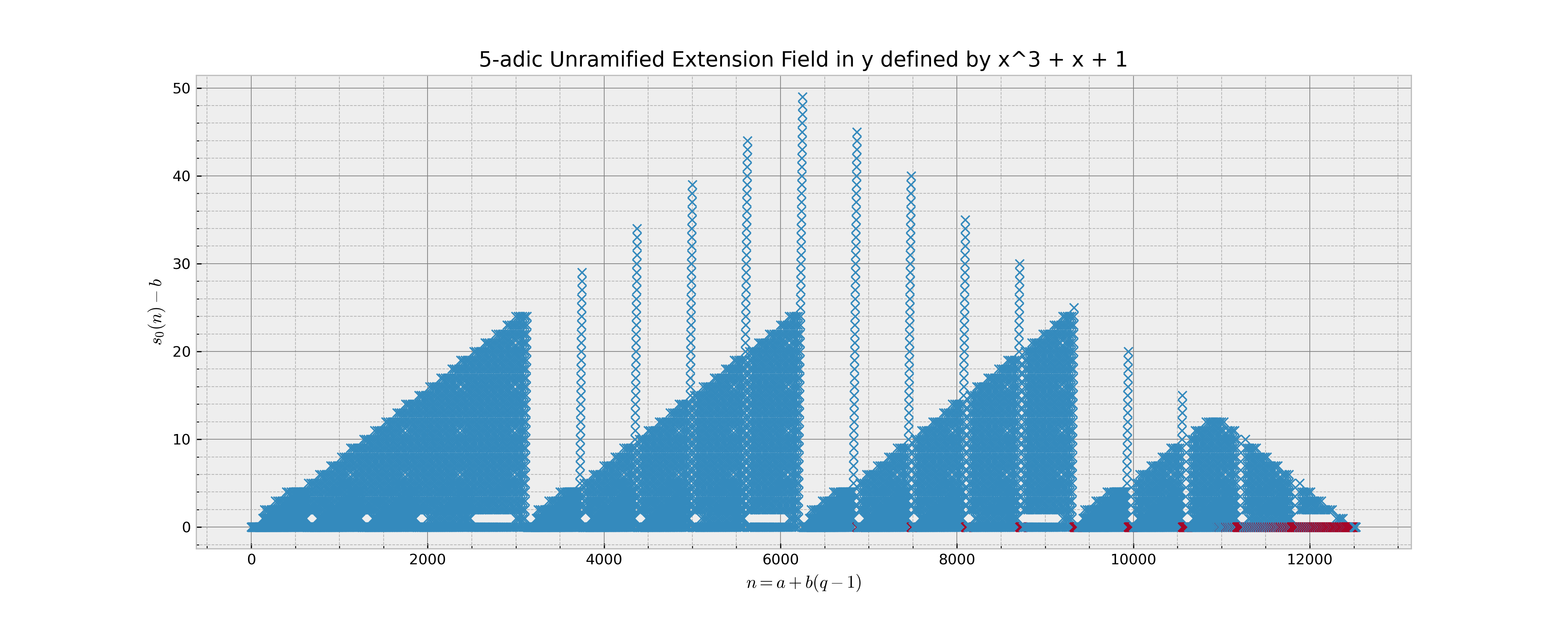}}
\caption{\texttt{extension = "5,3,12524,unram"} --- $s_0(n)$ in the cubic unramified extension of $\mathbb{Q}_5$ for $n < 12524$. Again, note how the red points --- the $n$'s for which $s_0(n) \geq 12524$ --- give the illusion of $s_0(n)-b$ decreasing.}
\label{5-3-unram}
\end{figure}

\subsection{Some results}

\begin{definition}
Given a natural number $n$, let $s_q(n)$ be the sum of digits of $n$ in base $q$.
\end{definition}

Recall Definition~\ref{s_0_def}:
\begin{definition*}
For $n = a + b(q - 1)$, let $s_0(n)$ be the minimal $s \geq b$ such that $(g_b^{(a), s})_{0 \leq b \leq s}$ satisfies $\nu_\pi(\gamma_b(g_b^{(a), s})) = -w_q(n)$, if such $s$ exists; otherwise set $s_0(n) = \infty$.
\end{definition*}

We define the following more intuitive quantity:
\begin{definition}
    For $n = a + b(q - 1)$, let $\Capp(n) = a + b s_0(n)$.
    Alternatively, $\Capp(n)$ is the minimal $N \geq n$ such that the $o_L$-span of $\left\{ \sigma_{i, j} : 0\leq i \leq j \leq N \right\}$ contains a polynomial of degree $n$ and $\pi$-valuation of the leading term $-w_q(n)$.
    \label{Cap_n}
\end{definition}

\noindent Here, the equivalence of the two definitions follows from the definition of $s_0(n)$.

Let $n = a + b(q-1)$. Analysing the computational results, we are led to believe that, if $s_q(n) < p $, then $s_0(n) = b$. This is made clear by the following:
\begin{theorem}
    Let $n$ be a positive integer such that $s_q(n) < p$. Let $j=n$ and $i=s_q(n)$. Then $\sigma_{i,j}$ is a polynomial of degree $n$, with $\pi$-valuation of leading term equal to $-w_q(n)$.
    \label{s_q(n)<p}
\end{theorem}

Recall the definition of the polynomials $c_n(Y)$ from \cite{dSIce}:
\begin{align*}
    [Y](t) = \sum_{n=1}^{\infty} c_n(Y) t^n
\end{align*}

\noindent Translating the definition of the polynomials $\sigma_{i,j}(Y)$ and using Lemma \ref{PolySigmas}, we get:
\begin{align*}
([Y](t))^i = \left( \sum_{n=1}^{\infty} c_n(Y) t^n \right)^i = \sum_{j=i}^{\infty} \sigma_{i,j}(Y) t^j.
\end{align*}

\noindent Using the binomial theorem, this gives:
\begin{align*}
    \sigma_{i,j} = \sum_{n_1 + n_2 + \ldots +n_i = j} c_{n_1} c_{n_2} \ldots c_{n_i}
\end{align*}

Of course, for $i=1$ we obtain $\sigma_{1,j}=c_j$. So, the proof of the Theorem 3.1 in \cite{dSIce} shows that $\Capp(n)=n$ for $n$ equal to some power of $q$. We will extend this result to all $n$ that have $s_q(n) < p$, where $s_q(n)$ is the sum of digits of $n$, written in base $q$. For this, we need the following lemma:

\begin{lemma}
\label{w_q inequality}
Let $n_1, n_2, \ldots, n_i$ be positive integers. Then, $w_q(n_1) + w_q(n_2) + \ldots + w_q(n_i) \leq w_q(n_1 + n_2 + \ldots + n_i)$. Equality holds if and only if $s_q(n_1) + s_q(n_2) + \ldots + s_q(n_i) = s_q(n_1 + n_2 + \ldots + n_i)$, that is, if there is "no carrying" in the sum $n_1 + n_2 + \ldots + n_i$.
\end{lemma}

\begin{proof}
Direct calculations show that
\begin{align*}
    w_q(n) = \frac{n-s_q(n)}{q-1}
\end{align*}
Substituting into our inequality, we need to prove
\begin{align*}
    s_q(n_1) + s_q(n_2) + \ldots + s_q(n_i) \geq s_q(n_1 + n_2 + \ldots + n_i)
\end{align*}
which can be checked by direct calculations or by induction. Equality holds in the initial inequality if and only if it holds here, which is to say there is "no carrying" in the sum $n_1 + n_2 + \ldots + n_i$.
\end{proof}

Now, we are ready for:

\begin{proof}[Proof of Theorem \ref{s_q(n)<p}.]

Recall that
\begin{align*}
    \sigma_{i,j} = \sum_{n_1 + n_2 + \ldots +n_i = j} c_{n_1} c_{n_2} \ldots c_{n_i}
\end{align*}
where each $c_k$ is a polynomial of degree at most $k$, with $\pi$-valuation of the leading term at least $-w_q(n)$ (as it is in $\text{Int}(o_L,o_L)$).

Let's look at each of the terms $c_{n_1} c_{n_2} \ldots c_{n_i}$. As each $c_k$ has degree at most $k$, this contributes to the coefficient of $Y^k$ in $\sigma_{i,j}$ if and only if $\text{deg}(c_{n_1}) = n_1, \text{deg}(c_{n_2}) = n_2, \ldots, \text{deg}(c_{n_i}) = n_i$. For the moment, assume this is the case. Then, the coefficient of $Y^n$ in this product is the product of leading coefficients of the $c_{n_i}$'s, which has $\pi$-valuation at least $-(w_q(n_1) + w_q(n_2) + \ldots + w_q(n_i))$. Now, using Lemma~\ref{w_q inequality}, this is at least $-w_q(n_1 + n_2 + \ldots + n_i) = -w_q(n)$, with equality if and only if $s_q(n_1) + s_q(n_2) + \ldots + s_q(n_i) = s_q(n) = i$, so the $n_i$'s are powers of $q$. That is, the only contribution to the coefficient of $Y^n$ in $\sigma_{i,j}$ that has small enough valuation comes from permutations of the unique way of writing $n$ as a sum of $i$ powers of $q$. In other words, if $n=b_r b_{r-1} \ldots b_1 b_{0(q)}$ is the writing of $n$ in base $q$, then the only terms that have a possible contribution are obtained when $(n_1,n_2, \ldots, n_i)$ is a permutation of $(q^0, q^0, \ldots, q^1, \ldots, q^r)$, where each $q^k$ appears $b_k$ times. 

But, by \cite{dSIce}, when $k$ is a power of $q$, $c_k$ is a polynomial of degree exactly $k$, with $\pi$-valuation of leading term exactly $-w_q(k)$. So, when $(n_1,n_2, \ldots, n_i)$ is a permutation as above, the product $c_{n_1} c_{n_2} \ldots c_{n_i}$ is a polynomial of degree $n$, with $\pi$-valuation of leading term equal to $-w_q(n)$. Moreover, as proved before, if $(n_1,n_2, \ldots, n_i)$ is not such a permutation, the product $c_{n_1} c_{n_2} \ldots c_{n_i}$ has the coefficient of $Y^n$ either $0$ or of $\pi$-valuation larger than $-w_q(n)$.

As there are $\binom{i}{b_0,b_1, \ldots, b_r}$ such permutations, with $p \nmid \binom{i}{b_0,b_1, \ldots, b_r}$ (because $i<p$ by the initial assumption on $n$), the final sum $\sigma_{i,j}$ has degree $n$, with $\pi$-valuation of leading term $-w_q(n)$.
\end{proof}

Definition~\ref{Cap_n} then gives:

\begin{corollary}
Let $n$ be a positive integer such that $s_q(n) < p$. Then $\Capp(n)=n$.
\end{corollary}

The numerical data suggests that this is the largest set on which $\Capp(n)=n$.

\subsection{SageMath Code} (tested on Sage 9.4)

\definecolor{codegreen}{rgb}{0,0.6,0}
\definecolor{codegray}{rgb}{0.5,0.5,0.5}
\definecolor{codepurple}{rgb}{0.58,0,0.82}
\definecolor{backcolour}{rgb}{1,1,1}
\lstdefinestyle{mystyle}{
    backgroundcolor=\color{backcolour},
    commentstyle=\color{codegreen},
    keywordstyle=\color{magenta},
    numberstyle=\tiny\color{codegray},
    stringstyle=\color{codepurple},
    basicstyle=\ttfamily\footnotesize,
    breakatwhitespace=false,
    breaklines=false,
    captionpos=b,
%    keepspaces=true,
    numbers=left,
    numbersep=5pt,
    showspaces=false,
    showstringspaces=false,
    showtabs=false,
    tabsize=4
}
\lstset{basicstyle=\ttfamily\footnotesize}
\lstset{style=mystyle}
\begin{lstlisting}[language=Python]
extension = "3,2,100,ram"  # Choose the extension to compute with
precision = 1000           # Choose the precision that Sage will use

parse = extension.split(',')
p = int(parse[0])  # Prime to calculate with
d = int(parse[1])  # Degree to calculate with
N = int(parse[2])  # Cutoff; must be divisible by q-1
ram = parse[3]


# Python imports
from time import process_time
import matplotlib.pyplot as plt
import numpy as np

# Definitions
from sage.rings.padics.padic_generic import ResidueLiftingMap
from sage.rings.padics.padic_generic import ResidueReductionMap
import sage.rings.padics.padic_extension_generic

power = p^d - 1
t_poly = ""

if ram == "ram":
    t_poly = f"x^{d}-{p}"
else:
    # generate poly for unramified case
    Fp = GF(p)
    Fp_t.<t> = PolynomialRing(Fp)
    unity_poly = t^(power) - 1
    factored = unity_poly.factor()
    factored_str = str(factored)
    start = factored_str.find("^"+str(d))
    last_brac_pos = factored_str.find(")",start)
    first_brac_pos = len(factored_str) \
                     - factored_str[::-1].find("(",len(factored_str)-start)
    t_poly = factored_str[first_brac_pos:last_brac_pos].replace('t','x')


# Define the polynomial to adjoin a root from
Q_p = Qp(p,precision)
R_Qp.<x> = PolynomialRing(Q_p)
f_poly = R_Qp(t_poly)

# Define the p-adic field, its ring of integers and its residue field
# These dummy objects are a workaround to force the precision wanted
dummy1.<y> = Zp(p).ext(f_poly)
dummy2.<y> = Qp(p).ext(f_poly)

o_L.<y> = dummy1.change(prec=precision)
L.<y> = dummy2.change(prec=precision)
k_L = L.residue_field()
print(L)

# Find the generator of the unique maximal ideal in o_L. 
Pi = o_L.uniformizer()

# Find f, e and q
f = k_L.degree()             # The degree of the residual field extension
e = L.degree()/k_L.degree()  # The ramification index
q = p^f

# Do linear algebra over the ring of polynomials L[X]
# in one variable X with coefficients in the field L:
L_X.<X> = L[]
L_Y.<Y> = L[]

v = L.valuation()
    
# The subroutine Dmatrix calculates the following sparse matrix of coefficients.
# Let D[k,n] be equal to k! times the coefficient of Y^k in the polynomial P_n(Y). 
# I compute this using the useful and easy recursion formula 
#      D[k,n] = \sum_{r \geq 0} \pi^{-r} D[k-1,n-q^r] 
# that can be derived from Laurent's Prop 1.20 of "outline9".
# The algorithm is as follows: first make a zero matrix with S rows and columns 
# (roughly, S is (q-1)*Size), then quickly populate it one row at a time,
# using the recursion formula.
def Dmatrix(S):
    D = matrix(L, S,S)
    D[0,0] = 1
    for k in range(1,S):
        for n in range(k,S):
            r = 0
            while n >= q^r:
                D[k,n] = D[k,n] + D[k-1,n-q^r]/Pi^r  # the actual recursion 
                r = r+1
    return D


# \Tau^{(m)} in Definition 10.10 of "bounded26":
def TauMatrix(Size, m, D=None):
    if D is None:
        D = Dmatrix((q - 1) * (Size + 1))
    R = matrix(L, Size,Size, lambda x,y: D[m + (q-1)*x, m + (q-1)*y])

    # Define a diagonal matrix:
    Diag = matrix(L_X, Size,Size, lambda x,y: kronecker_delta(x,y) * X^x)

    # Compute the inverse of R:
    S = R.inverse()
    
    # Compute the matrix Tau using Lemma 10.11 in "bounded26":
    Tau = S * Diag * R
    
    return Tau

def underscore(m, i):
    return m + i*(q-1)

def w_q(n):
    return (n - sum(n.digits(base=q))) / (q-1)

def compute_s(N, filename=None):
    assert N%(q-1) == 0
    
    t_start = process_time()
    D = Dmatrix(N)
    t_end = process_time()
    print(f"D matrix: {t_end-t_start : .2f} sec")
        
    s0_s = [-1 for _ in range(N)]
    
    for a in range(q-1):
        t_start = process_time()
        Tau_a = TauMatrix(N//(q-1), a, D)
        t_end = process_time()
        print(f"a={a}, Tau matrix: {t_end-t_start : .2f} sec")
        
        B_old = Matrix(0,0)
        d = 0
        for s in range(N // (q-1)):
            t_start = process_time()
            
            # 1. Use the non-zero rows from previous calculations
            # 2. Add a 0 column to its left
            # 3. Add rows corresponding to entries from the j_th column of Tau_a
            B = Matrix(L, 2*s-d+1, s-d+1)
            B[0,0] = 1  # Tau_a[s, s]
            B[1:s-d+1, 1:] = B_old
            for i in [0 .. s-1]:
                coeffs = Tau_a[i, s].list()
                B[s-d+1+i, B.ncols()-len(coeffs)+d:] = vector(L, reversed(coeffs[d:]))
            
            # Perform Gaussian elimination
            i0 = 0
            ks = []
            for k in range(B.ncols()):
                valuation_row_pairs = [ 
                    (v(B[i,k]), i) for i in range(i0, B.nrows()) if B[i,k] != 0]

                if not valuation_row_pairs:
                    raise ValueError("B is not full-rank")
                minv, i_minv = min(valuation_row_pairs)
                ks.append(k)

                # Swap the row of minimum valuation with the first bad row
                B[i0, :], B[i_minv, :] = B[i_minv, :], B[i0, :]

                # Divide the top row by a unit in o_L
                u = B[i0, k] / Pi^int(e * v(B[i0, k]))
                B[i0, :] /= u

                # Cleave through the other rows
                for i in range(i0 + 1, B.nrows()):
                    if v(B[i, k]) >= v(B[i0, k]):
                        B[i, :] -= B[i, k]/B[i0, k] * B[i0, :]

                i0 += 1
            
            d_is_updated = False
            for b in [d .. s]:
                n = a + b*(q-1)
                if v(B[s-b, s-b]) * e == -w_q(n): 
                    if s0_s[n] == -1:
                        s0_s[n] = s
                else:
                    if not d_is_updated:
                        d = b
                        d_is_updated = True
            B_old = B[:s-d+1, :s-d+1]
            
            t_end = process_time()
            print(f"a={a}, s={s}: {t_end-t_start : .2f} sec", end='\r')
            if filename is not None:
                with open(filename, 'w') as f:
                    f.write("n,s0\n")
                    for n, s0 in enumerate(s0_s):
                        f.write(f"{n},{s0}\n")
        print()

    plt.style.use('bmh')
    fig = plt.figure(figsize=(15,6), dpi=300)
    for n, s0 in enumerate(s0_s):
        if s0 != -1:
            b = n // (q-1)
            plt.plot(n, s0-b, 'x', c='C0')
        else:
            plt.plot(n, 0, 'x', c='C1')
    plt.xlabel(r"$n = a + b(q-1)$")
    plt.ylabel("$s_0(n) - b$")
    plt.title(str(L))
    plt.minorticks_on()
    plt.grid(which='both')
    plt.grid(which='major', linestyle='-', c='grey')
    
    return s0_s, fig


s0_s = compute_s(N);
\end{lstlisting}

\resumetocwriting

%\bibliographystyle{amsalpha}
%\bibliography{Bounded}

\providecommand{\bysame}{\leavevmode\hbox to3em{\hrulefill}\thinspace}
\providecommand{\MR}{\relax\ifhmode\unskip\space\fi MR }
\providecommand{\MRhref}[2]{%
  \href{http://www.ams.org/mathscinet-getitem?mr=#1}{#2}
}
\providecommand{\href}[2]{#2}

\end{document}